\documentclass[11 pt,reqno]{amsart}
 \usepackage{amsmath,amssymb}
\usepackage{latexsym,graphicx, marginnote}
\usepackage{color,soul}
\usepackage{comment}

\usepackage{graphicx}
\graphicspath{ {./images/} }

\usepackage{framed} 

\theoremstyle{plain}
\newtheorem{theorem}{Theorem}
\newtheorem{conjecture}{Conjecture}
\newtheorem{maintheorem}{Theorem}
\newtheorem{lemma}{Lemma}[maintheorem]
\newtheorem{sublemma}{Lemma}[lemma]

\newtheorem{question}{Question}

\newtheorem{perturbation}{Perturbation}
\theoremstyle{definition}

\def\be{\begin{eqnarray}}
\def\ee{\end{eqnarray}}
\def\beal{\begin{aligned}}
\def\enal{\end{aligned}}

\usepackage[utf8]{inputenc}
\usepackage[english, russian]{babel}
\usepackage[T1, T2A]{fontenc}

\def\L{\mathbb L}
\def\T{\mathbb T}
\def\R{\mathbb R}
\def\N{\mathbb N}

\def\del{\partial}

\def\eps{\varepsilon}

\newtheorem*{Pf}{Proof}

\newcommand\beqa[1]{ \begin{eqnarray} \label{#1}}
\newcommand{\eeqa}{ \end{eqnarray} }
\newcommand{\beqano}{ \begin{eqnarray*} }
\newcommand{\eeqano}{ \end{eqnarray*} }

\def\beal{\begin{aligned}}
\def\enal{\end{aligned}}

\renewcommand \L {\Lambda}

\renewcommand \phi {\varphi}

\renewcommand \c[1]{\mathcal{#1}}

\newcommand \area {{\text{area}}}

\newcommand \Id {{\text{Id}}}

\def\~{\tilde}
\def\Bbb{\mathbb}

\def\A{\Bbb A}

\def\R{\Bbb R}

\def\Z{\Bbb Z}
\def\T{\Bbb T}

\def\L{\mathcal L}

\def\eps{\varepsilon}

\title{Absolutely Periodic Billiard Orbits of Arbitrarily High Order}
\author{Keagan G. Callis}
\begin{document}
\allowdisplaybreaks
\maketitle

\selectlanguage{english}
\begin{abstract}
    We show that for all $n\in \N$, the set of domains containing absolutely periodic orbits of order $n$ are dense in the set of bounded strictly convex domains with $C^\infty$ boundary. The proof that such an orbit exists is an extension to billiard maps of the results of \cite{GST}, where it is proved that such maps are dense in Newhouse domains in regions of real-analytic area-preserving two-dimensional maps.  Our result is a step toward disproving a conjecture that no absolutely periodic billiard orbits of infinite order exist in Euclidean billiards, and is also an indication that Ivrii's Conjecture about the measure set of periodic orbits may not be true.
\end{abstract}
\tableofcontents
\section{Introduction}

We begin by considering a bounded domain $\Omega \subset \R^2$ with boundary $\del\Omega$, and examine the eigenvalue problem for the laplacian, i.e. solutions $\phi$ to
$$
 \Delta\phi = \lambda \phi, \qquad \phi = 0\text{ on }\del\Omega
$$
for some eigenvalue $\lambda$, and consider the set of all the eigenvalues $\{\lambda_1 < \lambda_2 \leq \lambda_3 \leq ... \leq \lambda_n \leq ...\}$ which admit solutions.  This is called the Laplace spectrum of $\Omega$, and is denoted by $Sp(\Omega)$.  Obviously, the boundary $\Omega$ completely determines the spectrum $Sp(\Omega)$.  This has some interesting consequences however.  For example, consider the counting function of the Laplace spectrum, defined by $N(\lambda)=|\{\lambda_i\leq \lambda\}|$.  Weyl was able to prove in \cite{weyl} that the first term in the asymptotic expansion of $N$ depends only on the area of $\Omega$.  For the second term in the expansion, there is the following conjecture (named after Weyl):

\begin{conjecture}[\textbf{Weyl's Conjecture}]
    Let $\Omega$ have a $C^\infty$ smooth boundary.  Then,
    \begin{align}\begin{split}\nonumber
        N(\lambda) = \frac{\area(\Omega)}{4\pi}\cdot\lambda + \frac{\L(\del\Omega)}{8\pi}\cdot\lambda^{1/2} + o(\lambda^{1/2}),
    \end{split}\end{align}
    where $\L(\del\Omega)$ is the length of $\del \Omega$.
\end{conjecture}  

What we mean by "$C^\infty$ smooth boundary" is that the curvature of the boundary, $\kappa:\R/\Z\L(\del\Omega)\rightarrow \R$, is a $C^\infty$ function.  What is interesting in this result is that purely geometric properties of $\Omega$ are found inside the expansion of the counting function.  One might imagine that if you continue to expand the counting function, you continue to obtain more quantities determined by $\Omega$ at each step (thus restricting which potential $\Omega$ may produce any given Laplace spectrum).  Indeed, prompted by Weyl's result, this inverse spectral problem was famously phrased by M. Kac in \cite{kac} as "Can one hear the shape of a drum:"

\begin{question}[\textbf{"Can one hear the shape of a drum?"}]
    Given a spectrum  $S=\{\lambda_1 < \lambda_2 \leq \lambda_3 \leq ... \}$, is there a unique $\Omega$ which produces $S$?
\end{question}

The answer depends on which class of domains one examines. For example, if one examines non-convex domains with non-smooth boundaries, there are indeed well-known counterexamples (see \cite{wolpert}).  On the other hand, there is a positive answer in the space of $\Z^2$ symmetric analytic domains \cite{zelditch}. However, if one restricts their view to domains which have only smooth boundaries, the question remains open.

 Interestingly, there is also a natural connection between the Laplace spectrum and another quantity derived from $\Omega$ called the length spectrum.  To define the length spectrum, first we need to define what is called a closed geodesic in $\Omega$.  A geodesic in $\Omega$ is defined as an orbit which travels in straight lines within $\Omega$ and reflects at the boundary $\del \Omega$ with the rule "the angle of incidence is equal to the angle of reflection."  Such a geodesic is called closed if it comes back to itself in a finite number of steps (i.e. is periodic).  Then, the Length spectrum, denoted by $L(\Omega)$, is defined as $$ 
L(\Omega) = \N\{\text{length of all closed geodesics of }\Omega\} \cup \N\{\L(\del\Omega)\}.
$$
For instance, one way this is related to the Laplace spectrum is as follows:
\begin{align}\begin{split}\nonumber
    \text{sing supp}\Bigg(t \rightarrow \sum_{\lambda\in Sp(\omega)}\text{exp}(i\lambda^{1/2}t)\Bigg) \subset \pm L(\Omega)\cup\{0\}.
\end{split}\end{align}
Thus, one can imagine that much information on questions related to the Laplace spectrum can be obtained through also studying the length spectrum (see \cite{hormander}, \cite{popov}).  Indeed, Weyl's conjecture was proven by Ivrii in \cite{Ivrii}, \textit{provided Ivrii's conjecture holds,} which states:

\begin{conjecture}[\textbf{Ivrii's Conjecture}]
    For any $\Omega$ with boundary in $C^{\infty}(\R)$, the set of periodic billiard orbits in $\Omega$ has measure zero .  
\end{conjecture}

Ivrii's conjecture has been proven when one restricts their view to specific families of billiards (for example in \cite{vasilliev_paper}), and can be considered when one restricts their view to specific sets of periodic orbits (e.g. if one restricts their view to only $k$-periodic orbits for fixed values of $k$).  Another question which is one step down from Ivrii's conjecture is whether or not there are open sets of periodic orbits in some family of domains.  For example, recently it was proven in \cite{Corentin} that for projective billiards (which are a generalization of billiards - see \cite{tabachnikov1},\cite{tabachnikov2}) there are no open subsets of 3-periodic orbits in higher dimension\footnote{He also proves that in the plane, the only counterexample is the so called right spherical billiard.}. This has also been proven for the case of convex domains in the plane with smooth boundary up to periods of period 4 (see \cite{Glutsyuk1},\cite{Glutsyuk2}).

In this view, it makes sense to study the prevalence of periodic orbits in smooth billiard systems. An obvious corollary of Ivrii's conjecture would be that there are no open sets of periodic orbits.  Considering that seemingly wild possibility, we observe that if one were to somehow find an open set of $q$-periodic points in a billiard system, then the differential at any point in this set must be the identity, i.e. $df^q(x_0)=\Id$ for each $x_0$ in this open set, where f is the billiard map.

We make the definition (following the definition given in \cite{SafarovVisilliev}\footnote{Note that in \cite{SafarovVisilliev}, the definition of absolutely periodic is given in terms of geodesic flows.  We prove the equivalence of our definition and theirs in Appendix \ref{appendix_abs_t_periodic_implies_identity}.}), that a $q-$periodic orbit is called \textbf{absolutely periodic} if the differential of $f^q$ is the identity and all second order and higher partial derivatives of $f^q$ are $0$ at any point in the orbit.  Stated with this definition, in \cite{SafarovVisilliev} there is also the following conjecture\footnote{Note that in the discussion around their conjecture they note that it should be possible to obtain such orbits as we do in this paper, though they do not make any mention on how often they appear (e.g. their density among smooth convex domains).}:

\begin{conjecture}[\textbf{Safarov-Vassiliev}]
    There are no absolutely periodic orbits for euclidean billiards.
\end{conjecture}

This has been proven in the case of convex domains with analytic boundary (\cite{SafarovVisilliev}), and in a few other cases (\cite{vasilliev}).  Considering all the discussion so far, if one were to find such an orbit, it would be quite surprising.

In this paper, we detail a proof on the existence of billiard systems exhibiting an \textbf{absolutely periodic orbit of order $n$}, which we define to be an orbit such that $df^q(x_0) = Id + F(x_0),$ where $F(x_0)=0$ up to order $n$.  Hence, with these definitions, an absolutely periodic orbit is an absolutely periodic orbit of infinite order. 

Before delving further into the result, we first describe in more detail the domains we consider.  Denote by $\mathcal{D}^r$ the space of unit length strictly convex domains in $\mathbb{R}^2$ with $C^{r+1}$ smooth boundary endowed with the $C^{r+1}$ topology, for $r\in \N$ or $r = \infty$.  We label a parameterization of a boundary $\del \Omega$ by $\gamma(s)$, the curvature at $\gamma(s)$ by $\kappa(s)=\theta'(s)$ (where $\theta(s)$ is the angle the tangent vector of $\del \Omega$ at $\gamma(s)$ makes with the horizontal axis), and the radius of curvature $\rho(s)=\frac{1}{\kappa(\theta^{-1}(s))}$, with arc-length parameterization s.  We identify each domain $\Omega$ with the curvature of its boundary.  Thus the domains in $\mathcal{D}^r$ correspond to curvatures satisfying (see \cite{mm})

\begin{align}
    \begin{split}\label{Dr_curvatures}
         &\bullet\mathcal L(\del \Omega) = 1
         \\&\bullet\kappa(s)\in C^{r}(\mathbb{R}/\mathbb{Z})
         \\&\bullet\kappa(s) >0 
         \\&\bullet\int_0^{1} \rho(t)\cos(2\pi t)dt = \int_0^{1} \rho(t)\sin(2\pi t)dt = 0
    \end{split}
\end{align}

Also, by abuse of notation (whenever it causes no difficulties) we say the billiard map $f$ associated to $\Omega$ is in $\mathcal D^r$ when $\Omega$ is in $\mathcal D^r$.  Additionally, when there is an orbit under the map $f$, we sometimes say $\Omega$ 'has' that orbit or $f$ 'has' that orbit. With these definitions given, we may now state our main result:

\begin{maintheorem}\label{MainTheorem}
For any $n\in \N$, and any $\Omega \in \mathcal D^\infty$, there exists a domain $\tilde{\Omega} \in \mathcal D^\infty$ arbitrarily close to $\Omega$ in the $C^{\infty}$ topology such that 
$\tilde{\Omega}$ has an absolutely periodic orbit of order $n$.
\end{maintheorem}

To prove this we seek to use the methods from \cite{GST} to obtain
a similar result for billiard maps.  In their setting, something stronger is proven 
that involves what are called Newhouse domains.  The notion of Newhouse 
domains comes from another field of dynamical systems - that of 
homoclinic orbits and their bifurcations. Newhouse proved in 
\cite{Newhouse1}, \cite{Newhouse2}, \cite{Newhouse3}, that there exist open 
sets (Newhouse domains) in which maps that exhibit homoclinic tangencies 
are dense.  Moreover, it was shown that these domains exist around any map 
that exhibits a homoclinic tangency.  Later, this result was extended by Duarte to the class of area-preserving maps in \cite{duarte}.

In these regions, one can create many interesting phenomenon by perturbing the map so that these homoclinic orbits move in such a way we call "unfolding."  For instance, in \cite{asgor} (see also \cite{gorodetsky}) they were able to find invariant hyperbolic sets of arbitrarily large Hausdorff dimension by studying these generic unfoldings.  Related to our paper, by studying unfoldings it was shown in (\cite{GST}, Theorem 3) that in 
the Newhouse domains in the space of area
preserving $C^\infty(\mathbb{R}^2)$ maps, maps 
with infinitely many homoclinic tangencies of 
all orders are dense 
in the $C^r$-topology for any $r>0$, including $r=\infty$.  
In this paper we prove this result in the case of 
billiard maps, though here we are only interested 
in obtaining a map close to our original map with at least one periodic orbit which to high order is the identity.

While we follow many of the same arguments found in \cite{GST}, it should be mentioned there are natural difficulties in extending the problem to billiard maps.  There are many ways which we tackle these various problems, which we will go into in detail, but one particularly interesting result used to analyze the billiard map is the following:

\setcounter{theorem}{1}

\begin{theorem}\label{theorem_perturb_many_independently}
    Consider the map $T$ which is the billiard map composed with itself $n+3$ times. Given an orbit $\mathcal O = ((s_0,\phi_0),...,(s_n,\phi_n),...)$ (not necessarily periodic) with $s_i\neq s_j$ for $i\neq j$ and close to the boundary, one may vary each partial derivative of $T$ at $(s_0,\phi_0)$ up to order $n$ independently by perturbing the derivatives of the curvature at $s_0,...,s_n$ up to order $n-1$, while leaving the orbit $\mathcal O$ fixed.
\end{theorem}

For a more precise statement of this theorem, see Appendix \ref{appendix_all_perturbations}, Perturbation \ref{perturbation_which _lets_you_vary_all_derivatives_independently}.  The proof of this theorem is quite technical, and is detailed in Appendix \ref{appendix_perturb_nth_derivative_independently}.  It also highlights one of the difficulties in working with billiard maps; namely, that to perturb within the class of billiard maps one must perturb the curvature function of the associated domain and then analyze how this change effects the jets of the billiard map. 

Now, we go into more detail on the various difficulties in restricting the results of \cite{GST} to the class of billiard maps.

\begin{center}
    \textbf{
    Generic Conservative Systems
    }
\end{center}

In general, proving generic properties of dynamical systems often depends on the class of systems you examine.  The wider the class of systems, the more flexibility one has in proving a property holds.  For instance, there is a closing lemma for a range of dynamical systems such as $C^1$ diffeomorphisms and $C^1$ Hamiltonian flows (\cite{pugh_robinson}), yet an analogous result for geodesic flows under Riemannian metrics is still open\footnote{For Hamiltonians on closed surfaces there is a $C^\infty$ closing lemma (see \cite{asaoka_irie}), however, the proof does not apply to geodesic flows and does not use local perturbation techniques.}.  Similarly in \cite{Contreras} the existence of positive topological entropy and nontrivial hyperbolic sets was proven for Riemannian metrics in high dimension, even though for many years prior the same result was known for hamiltonian flows (\cite{Newhouse3}). 

The main difficulty in these restricted settings is that there are no local perturbations in the phase space. For example, let $(M,g)$ be a compact Riemannian manifold.  Then geodesic flow is a flow on a unit cotangent bundle, denoted by $U^*M$.  Consider then $\pi:U^*M\rightarrow M$ where $\pi$ is a natural projection.  If we perturb the metric $g$ on a small open neighborhood $V$ of some point $x$ in $M$, then in $U^*M$ we perturb on $\pi^{-1}V$, which contains many fibers (see Figure \ref{fig_riemannian_difficulty}).

\begin{figure}[h]
\centering
\includegraphics[width=9cm, height=4.5cm]{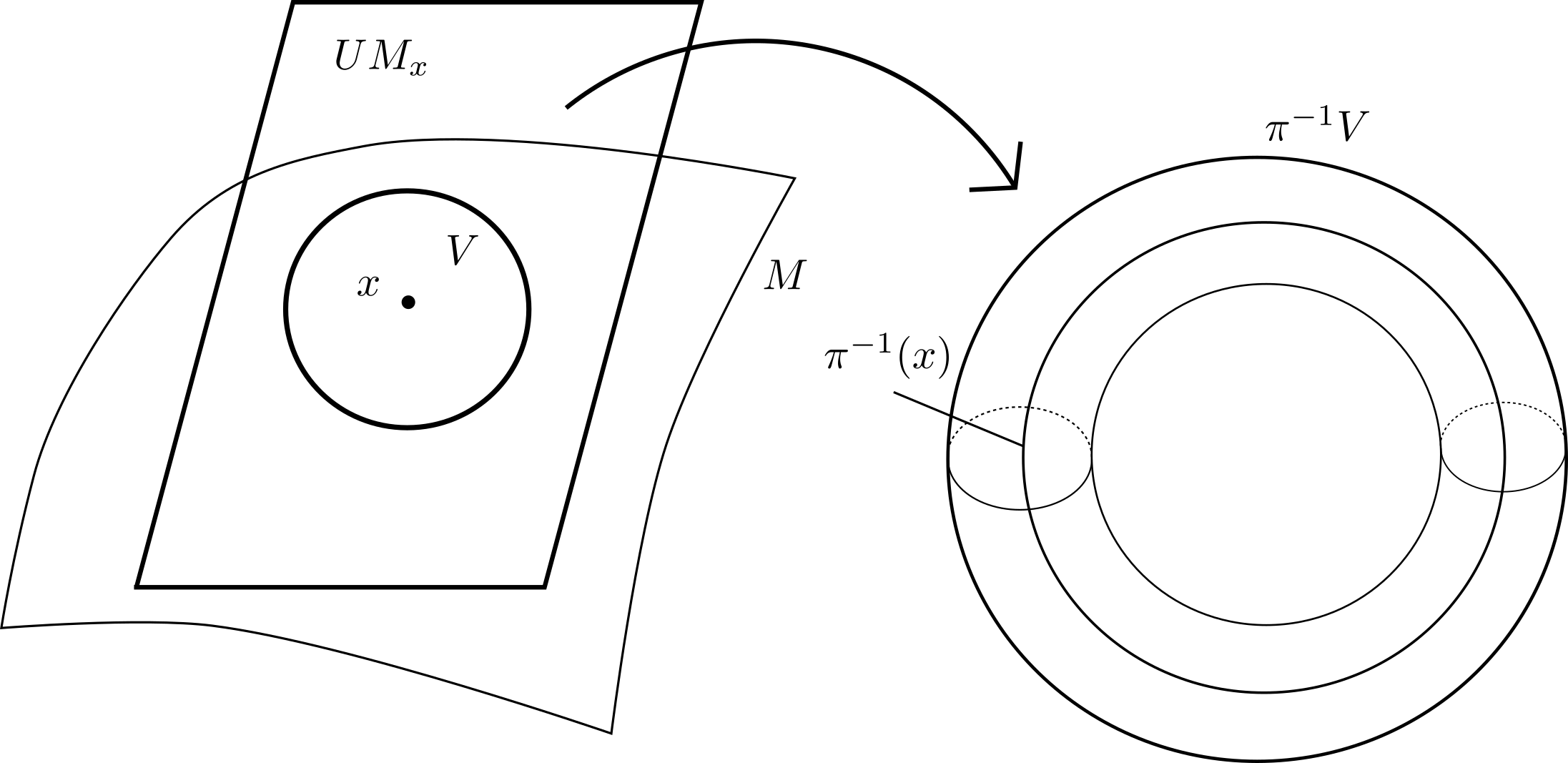}
\caption{A picture showing how local perturbations of g in $M$ lead to non-local perturbations in $U^*M$.}
\label{fig_riemannian_difficulty}
\end{figure}

So, naturally there are significant difficulties in doing the same analysis as in \cite{GST} while restricted to only billiard maps.  The main issues that exist and which we overcome in this paper are
\begin{itemize}
    \item Given a finite set of orbits, ensuring the deformations used change the differential of the billiard map at one given orbit in this set without changing the other orbits in this set.  This difficulty arises from there being no local perturbations in the phase space, analogous to the difficulty found in other dynamical systems mentioned above.
    \item Constructing deformations which do not lead outside of the class of billiard maps and unfold high order HT generically.
    \item Constructing deformations which do not lead outside of the class of billiard maps and accurately vary high order derivatives of the billiard map in controlled ways.
\end{itemize}

More will be said on how we overcome these challenges, but briefly we mention the general strategies.  For the first difficulty, we make sure new orbits satisfy what we call an injectivity condition. Essentially, this condition says that there are, in each orbit, a sufficient number of points which do not hit the same spot on the boundary as the other orbits considered hit.

For the last two points: to ensure we stay in the class of billiard maps as we deform, we study how changes of the curvature function effect the associated billiard map.  Then, in order to accurately change higher derivatives of the billiard map, we carefully go through (in much technical detail) how changing the higher derivatives of the curvature effects the higher derivatives of the billiard map.

Now we will go over some standard definitions used for the concepts considered throughout the paper.  Those who are familiar with these can find the outline of the proof given in Section \ref{overview_of_proof_of_theorem1}.

\section{Billiard Maps and Homoclinic Orbits}

Here we describe briefly billiard orbits in $\mathbb{R}^2$ (see also \cite{mather-forni},\cite{siburg}). Geometrically 
speaking, billiard orbits are curves obtained by considering geodesics on 
the inside of a domain with the rule "angle of incidence equal to angle of 
reflection" to describe how the curve behaves when it hits the boundary.

\begin{figure}[h]
\centering
\includegraphics[width=6cm, height=6cm]{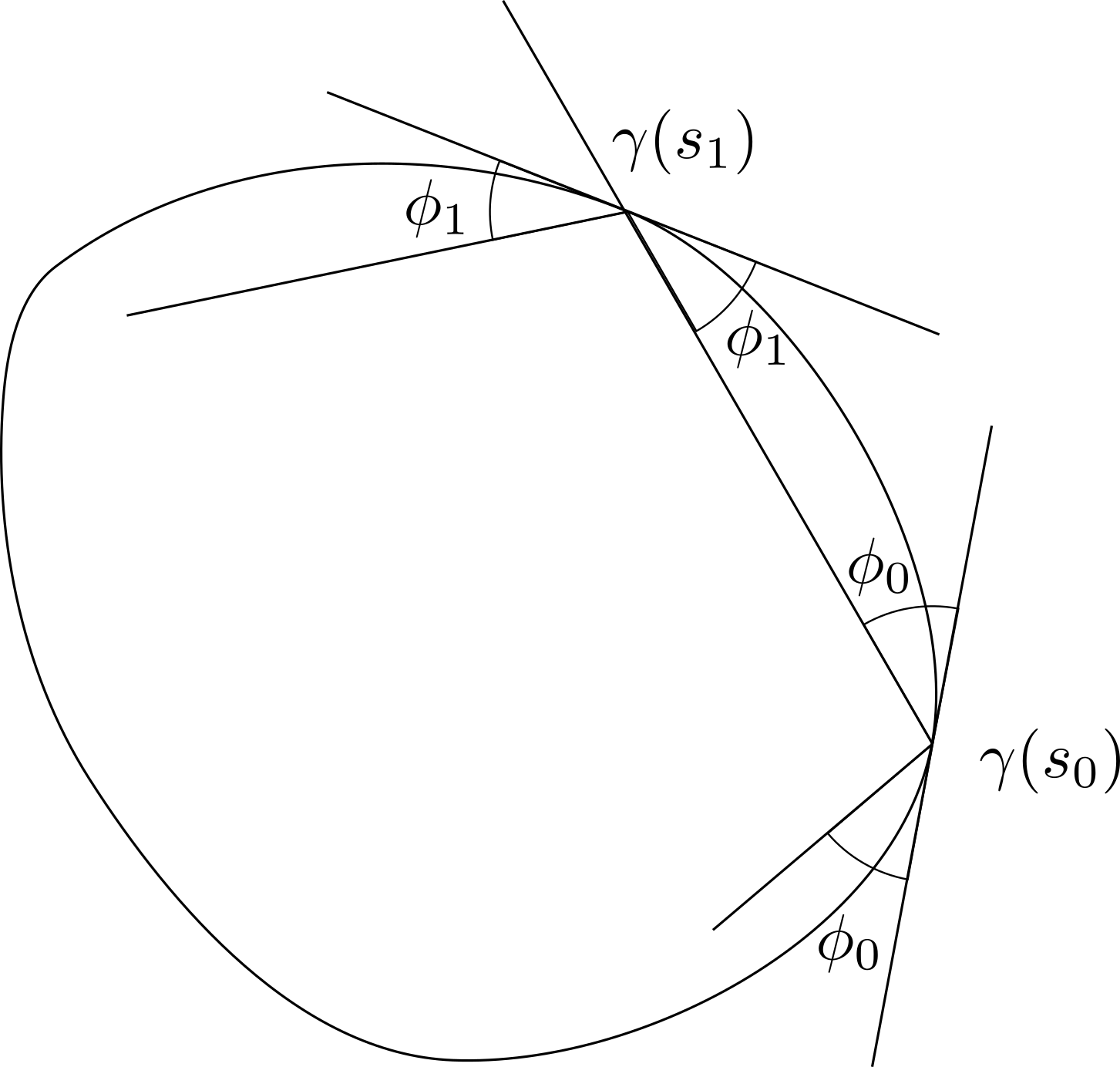}
\caption{A billiard orbit.}
\end{figure}

As before, we consider domains $\Omega$ in $\mathcal D^r$, with boundary parameterized by $\gamma$. For these domains we have the associated billiard map $f$ acts as $f:\mathbb A \rightarrow \mathbb A$, where $\mathbb A = \T \times [0,\pi]$ is a cylinder.  The map $f$ is defined by $f(s_0,\phi_0) = (s_1,\phi_1)$, where 
$\phi$ is the angle the trajectory of the billiard orbit makes with the tangent 
of $\del \Omega$ at $\gamma(s)$.

It is easy to see from elementary geometry that if we define 
$L(s_0,s_1)=\|\gamma(s_1)-\gamma(s_0)\|$, we have

\begin{align}\begin{split}\nonumber
\begin{cases}
    \frac{\del L(s_0,s_1)}{\del s_0} &= -\cos(\phi_0),\\
    \frac{\del L(s_0,s_1)}{\del s_1} &=\ \  \cos(\phi_1).
\end{cases}
\end{split}\end{align}

In this view, we note that billiard maps are a class of maps that act on the cylinder $\mathbb A$.

\begin{figure}[h]
\centering
\includegraphics[width=4cm, height=3.5cm]{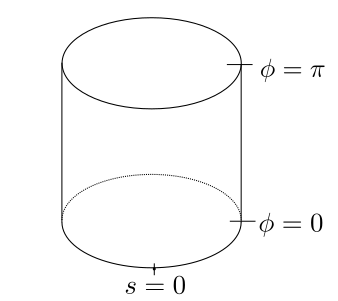}
\caption{We view the billiard map as a special class of maps acting on the cylinder $\mathbb A$.}
\end{figure}

We also note that these maps are actually twist diffeomorphisms, since in the case of strictly convex billiards we have
\begin{align}\begin{split}\nonumber
    \frac{\del s_1}{\del \phi_0} > 0.
\end{split}\end{align}
This implies that to have $df^q(x_0) = Id$ would indeed be, in a sense, rather unusual, since it would effectively be undoing the effects of the twist condition.

Finally we also mention that billiard maps are area preserving maps, with area form $\sin{(\phi)}\,d\phi \wedge ds$.  In our case, this is relevant as this implies that the determinant of the differential of a periodic orbit is always equal to $1$ (where the product is 1 since the map is area preserving).

In our paper, we consider periodic orbits in billiard maps, often of period $q$.  Specifically we begin by considering hyperbolic periodic orbits, so that $df^q(x_0)$ has one eigenvalue $\lambda$ greater than $1$, and the other $\lambda^{-1}$ less than $1$.

Further, we consider a special type of hyperbolic orbit that has what is called a homoclinic tangency (HT).  This means it has a point $p_0$ that is in both the unstable manifold $W^u(x_0)$ and the stable manifold $W^s(x_0)$ of our orbit, and that these manifolds intersect tangentially at $p_m=f^m(p_0)$.  If the tangency is of order $n$, we say the tangency is a homoclinic tangency of order $n$.

\begin{figure}[h]
\centering
\includegraphics[width=8cm, height=8cm]{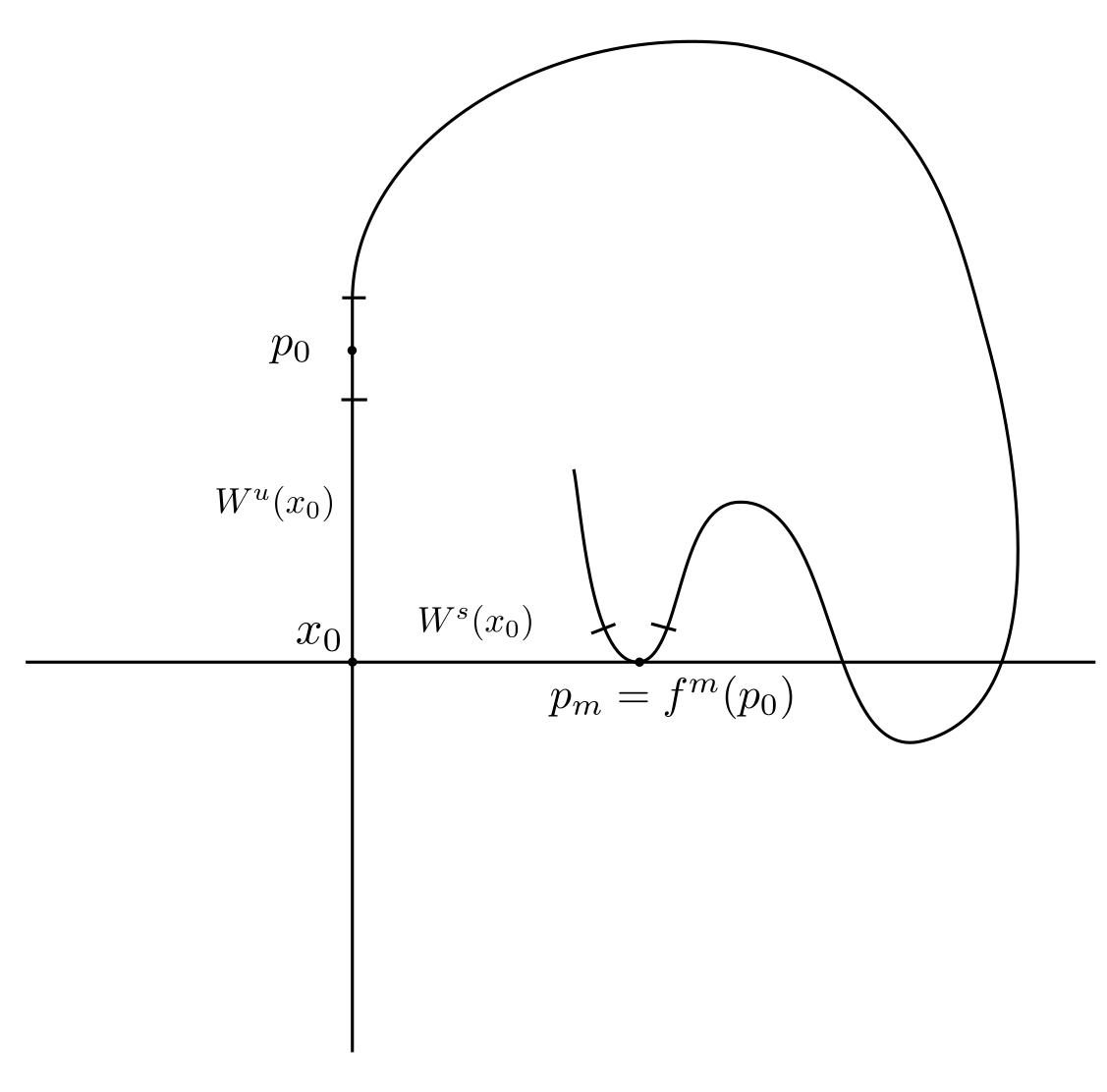}
\caption{A hyperbolic periodic point $x_0$ with a homoclinic tangency at $p_m$.}
\end{figure}

Homoclinic tangencies have been studied extensively, and it has been found that bifurcations of maps exhibiting homoclinic tangencies can produce a variety of surprising behaviours.  In \cite{GST}, the proof of theorem $3$ consists essentially 
of a series of bifurcations of maps with homoclinic (or heteroclinic) tangencies.

A major technical step of our proof is to construct a deformation that \textbf{unfolds (splits) a given HT of order n generically, while remaining in the class of billiard maps}.  
This means that, essentially, we can move each derivative of the curve $f^m(W^u(x_0))$ in a neighborhood of $p_m$ independently as we vary our map 
by a deformation dependent on a parameter $\eps = (\eps_0,...,\eps_n)$.  

Specifically, we let $p(t)\in f^m(W^u_{loc}(x_0))$ with $t\in (\delta,\delta)$ and $p(0)=p_m$, and let $\Phi(t)$ be the shortest distance between $p(t)$ and $W^s_{loc}(x_0)$.  We then embed our map into a smooth family of maps $(f_\eps)_{|\eps| \leq 1}$ dependent on $\eps$ and with $f_0 = f$. Then, for each $\eps$, we similarly define $\Phi_\eps(t)$ to be the shortest distance between $p_\eps(t)\in f^m(W^u_{loc,\eps}(x_0))$ and $W^s_{loc,\eps}(x_0)$. Then, defining for $j \geq 0$
\begin{align}\begin{split}\nonumber
    \Gamma_j(\eps)=\Phi_\eps^{(j)}(0),
\end{split}\end{align}
we say our HT \textbf{of order n unfolds generically} if  (see Figure \ref{fig_unfolding_the_tangency})
\begin{align}\begin{split}\nonumber
    \det\bigg(\frac{\del (\Gamma_0(\eps),\Gamma_1(\eps),...,\Gamma_n(\eps))}{\del\eps}\bigg)\neq 0.
\end{split}\end{align}.

\begin{figure}[h]
\centering
\includegraphics[width=8cm, height=5cm]{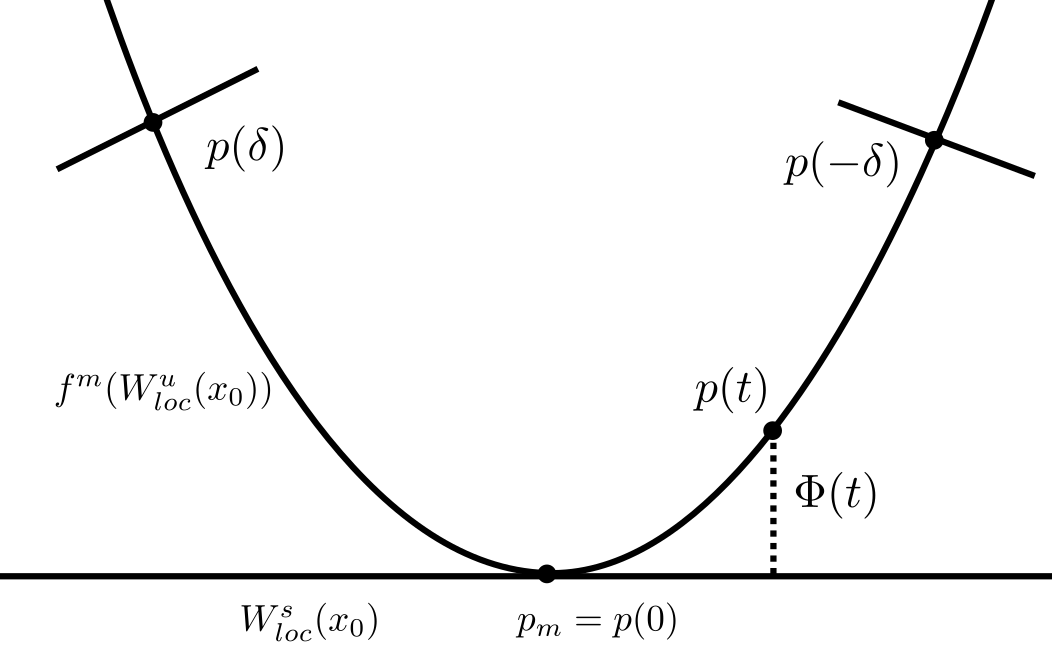}
\caption{Unfolding the tangency.}
\label{fig_unfolding_the_tangency}
\end{figure}

We now briefly describe in more detail some of the already mentioned challenges in the proof and how we overcome them.  This leads to some natural definitions which we use throughout the paper.

\section{Billiard map's relation to curvatures and the injectivity condition}

As mentioned before, the main difficulty in the problem of studying perturbations of billiard maps arises from the differences between general twist diffeomorphisms and billiard maps.  Recall that billiard maps are a distinct subset of are preserving twist differomorphisms.  In particular, we want to consider only billiard maps in $\mathcal D^n$, i.e. maps whose domains have their curvature function satisfy the Conditions \ref{Dr_curvatures}.  So, order to ensure our maps remain in $\mathcal D^n$ as we perform perturbations (and do not accidentally find themselves outside of the subset of maps we want to remain in), all of our perturbations are done by perturbing the curvature function of the domain.  Thus, we examine how perturbations of the curvature of our domain effects the billiard map.  This is a non-trivial technical aspect of our paper.

To tackle this, we detail a general perturbation wherein we perturb the $n^{th}$ derivatives of the curvature at $n+3$ sequential points in a periodic orbit $((s_0,\phi_0),...,(s_{n+3},\phi_{n+3}))$, and calculate how this effects the $n^{th}$ partial derivatives of $s_{n+3}$ and $\phi_{n+3}$, and show that one may vary each of these independently while not effecting lower partial derivatives of these functions - provided the period of these orbits is large enough and they are close enough to the boundary.  This result was already mentioned in Theorem \ref{theorem_perturb_many_independently}.

Additionally, behind all of the perturbations we do is an important condition we must ensure which we call the injectivity condition, which we begin to describe now.  Ensuring this condition is the second main difficulty of extending the results of \cite{GST} to the case of billiard orbits (the first difficulty being precisely determining how changes in the curvature effects changes in the billiard map).  First we give a motivation behind this condition, and then we give a precise definition of this condition.

Throughout our proof we will desire to perturb multiple heteroclinic tangencies (of varying degree) between various orbits independently. Now in the case of arbitrary $C^\infty$ area-preserving maps, we can perturb locally so that we only effect one orbit at a time.  In the case of billiard maps however, we can only perturb the curvature, which acts as perturbing in a strip in the cylinder that the billiard map acts on.  Thus, if we wish to perturb one orbit without effecting a set of other orbits, we must have that the strips where we perturb do not have any points of the other orbits entering them.

More precisely: Consider a group of orbits, $\{\mathcal O_j\}_{1 \leq j \leq N}$, where each orbit is defined by
\begin{align}\begin{split}\nonumber
    \mathcal O_j = {f^k(x_j):k \in \Z},
\end{split}\end{align}
where $x_j \in \A$ for each $1 \leq j \leq N$.  Then, if 
$\tilde{p}=(\tilde{s},\tilde{\phi})$ is a point in one of these orbit, we define the orbit 
to be \textbf{injective at} $\tilde{p}$ if for some $\delta > 0$ and all $x$ in $\{(s,\phi) \in \mathcal O_j:1\leq j\leq N\}$, 
$|s-\tilde{s}|\leq \delta$ implies $x=\tilde{p}$.  This is equivalent to the condition that some open neighborhood of the
the vertical line going through the point $x$ does not contain any other point 
in any orbit in our collection.  If there are at least three such points in an orbit, 
we say the orbit \textbf{satisfies the injectivity condition with respect to} 
$\{\mathcal O_j\}_{1 \leq j \leq N}$ (see Figure \ref{fig_injctivity_conditionnn_diagram}).  

\begin{figure}[h]
\centering
\includegraphics[width=8cm, height=8cm]{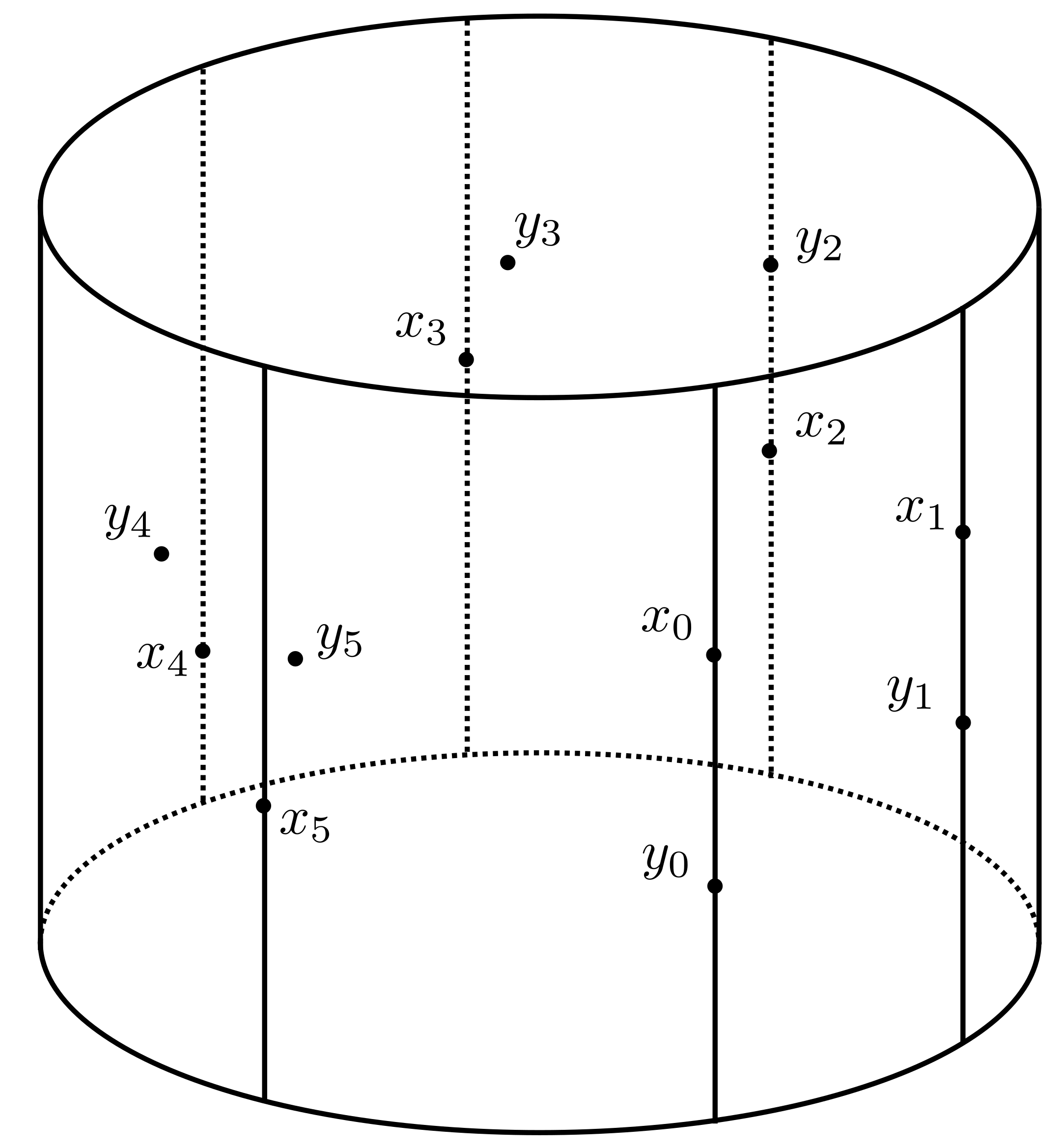}
\caption{Two orbits $(x_0,...,x_5)$ and $(y_0,...,y_5)$ on our cylinder.  
The injectivity condition is satisfied at $x_3,x_4,$ and $x_5$ since there are open neighborhoods around the vertical 
line through these points that does not contain any point in the orbit of $(y_0,...,y_5)$.}
\label{fig_injctivity_conditionnn_diagram}
\end{figure}

Thus, we must ensure that at each step of our proof whenever there is a new 
orbit considered, the injectivity condition remains satisfied with the new orbit and all of the previous orbits.
\\
\\
\section{Overview of Proof of Theorem 1}\label{overview_of_proof_of_theorem1}

Here we provide a road map of our proof.  For the rest of the paper, we take $\Omega$ to be a domain in $\mathcal D^\infty$.  We also let $f$ be the billiard map associated to $\Omega$.  For the proof at the end of this section and for the outline of the proof, we consider a fixed $n\in \N$.  With these definitions, the general steps we take in proving Theorem \ref{MainTheorem} are as follows:

\begin{framed}
\begin{center}
    \textbf{Outline of proof:}
\end{center}
\begin{enumerate}
    \item \textit{(Constructing an orbit with Homoclinic Tangency)} We first perturb to obtain a boundary with a hyperbolic periodic orbit that has an associated quadratic homoclinic tangency.  The ideas here use KAM theory and the fact (due to Lazutkin) that the billiard map is nearly integrable near the boundary.  \textbf{This is done in Appendix \ref{section_Constructing_an_Orbit_with_Homoclinic_Tangency}}.
    \item \textit{(Using a Tower Construction to obtain a Rotation to order n)} Then, following the steps of \cite{GST}, we construct a tower of heteroclinic tangencies.  We unfold this tower to obtain a HT of order $2n+4$.  Then, we unfold this HT of order $2n+4$ generically to obtain an elliptic periodic orbit which, to order $n$, is a rotation.  \textbf{This is done in section \ref{sec_constructing_the_tower}}.
    \begin{enumerate}
        \item \textit{(Unfolding an HT of order $n$ generically)} Throughout step 2 we obtain HT of order $k$ (where $2 \leq k \leq 2n+4$) which we want to unfold generically.  \textbf{We show how we do this in section \ref{sec_perturb_nth_order_ht_generically}}.
        \item \textit{(Verify the Injectivity condition)} Whenever a new orbit is considered in step 2 we perturb our domain so that the injectivity condition remains satisfied between this orbit and all the other orbits already under consideration.  \textbf{We show this perturbation in section \ref{sec_verify_injectivity_condition}}.
    \end{enumerate}
    \item \textit{(Rotate the differential to order n)} We then construct a perturbation at $n+3$ points of this elliptic periodic orbit that changes the differential up to order $n$ as a slight rotation, so that we obtain a rotation matrix with a guaranteed rational rotation angle.  \textbf{This is done in section \ref{sec_rotate_the_differential}}.  This then completes the proof.
\end{enumerate}
\end{framed}

We now go through the specific lemmas used to prove Theorem 1, following the outline given above. As stated, in Appendix \ref{section_Constructing_an_Orbit_with_Homoclinic_Tangency}, we describe the steps to obtain a domain whose associated billiard map has a hyperbolic periodic orbit with a quadratic homoclinic tangency.  The lemma used from this section is as follows:

\begin{lemma}\label{lemma_creating_quadratic_ht}
    Consider the set of $C^\infty$ domains whose associated billiard map has at least one hyperbolic periodic orbit such that:
    \begin{itemize}
        \item the orbit is of some period $q \gg 1$, and has rotation number $1/q$ (such a periodic orbit forms a $q$-gon inscribed inside $\Omega$),
        \item the invariant manifolds of a point in the orbit have a quadratic homoclinic tangency,
        \item the periodic orbit remains close to the boundary and is such that the local stable and local unstable manifolds at each point in the orbit have non-zero angles with the vertical axis.
    \end{itemize} These domains are dense among $C^\infty$ domains.
\end{lemma}

The main idea in proving this lemma is that near the boundary, the dynamics of the billiard map become almost integrable.

In Section \ref{sec_perturb_nth_order_ht_generically} we show how to perturb high order HT generically.  This is step 2a in the outline.  The lemma proven in this section is

\begin{lemma}\label{lemma_unfold_generically}
    Let $n \in \N$.  If $p$ is a point of homoclinic tangency for the billiard map $f$ and the injectivity condition is satisfied for the homoclinic orbit and its associated periodic orbit, then there exists a deformation of the boundary such that this $n^{th}$ order HT unfolds generically at $p$.
\end{lemma}

In Section \ref{sec_verify_injectivity_condition} we show step 2b of our outline.  The important lemma proven in this section is the following:

\begin{lemma}\label{lemma_on_injectivity}
    Let $n\in\N$.  Given a hyperbolic periodic orbit $\{(s_j,\phi_j):0\leq j \leq q-1\}$ with a homoclinic orbit containing a point of homoclinic tangency of order $n$ at $p_m$, there exists an arbitrarily small perturbation of the boundary on small neighborhoods of the points $\gamma(s_{j_1}),\gamma(s_{j_2}),\gamma(s_{j_3})$, where $j_1 < j_2 < j_3$, which keeps the hyperbolic periodic orbit fixed and is such that under the perturbed map, the hyperbolic periodic orbit has an associated homoclinic orbit which satisfies the injectivity condition with respect to itself and the periodic orbit, and also has a HT of order $n$.  Moreover, the point of HT of order $n$ under the perturbed map is close to the point which is an $n^{th}$ order HT under the unperturbed map by choosing the size of the perturbation to be small.
\end{lemma}

The main idea here is that by slightly varying the angles of the unstable and stable manifolds of our periodic point, we may slightly move the points of the in the homoclinic cycle, and it is shown that by considering many points in this cycle it is impossible that it fails to achieve the injectivity condition as we smoothly vary the angles of these manifolds.
\\
In Section \ref{sec_constructing_the_tower}, we go through step 2 of our outline.  The main result we prove is

\begin{lemma}\label{lemma_finding_elliptic_rotation_of_order_n}
    Let $n \in \N$.  Given a $C^\infty$ billiard map $f$ with a point of quadratic HT, there exists a $C^\infty$ billiard map $\tilde f$ which is arbitrarily close to $f$ in the $C^\infty$ topology such that $\tilde f$ contains an elliptic periodic orbit such that at this orbit the differential of $\tilde f$ is a rotation up to of order $n$.
\end{lemma}

To prove this we perform a series of unfoldings of various HT (following the steps of \cite{GST}) in order to construct a homoclinic tangency of order $2n+4$, and then unfold this HT to obtain an elliptic periodic orbit which is a rotation up to order $n$.

In Section \ref{sec_rotate_the_differential} we prove
\begin{lemma}\label{lemma_rotating_the_differential_to_order_n}
    For any $0 < \delta\ll 1$ and $n \in \N$, there is a perturbation of size $\delta$ in the $C^\infty$ topology at $n+3$ points of any given orbit $(s_0,\phi_0),...,(s_{q-1},\phi_{q-1})$ with $q > n+3$ which leaves this part of the orbit fixed and which changes the differential of the billiard map so that
    \begin{align}\begin{split}\nonumber
        df^q(s_0,\phi_0) \rightarrow R_{\delta}df^q(s_0,\phi_0) +\Delta_n(s_0,\phi_0),
    \end{split}\end{align}
    where $R_\delta$ is a rotation matrix which rotates by angle $\delta$ and where the error term $\Delta_n(s,\phi)$ satisfies
    \begin{align}\begin{split}\nonumber
        \frac{\del^k \Delta_n(s,\phi)}{\del s^{k-i}\del\phi^i}\Bigg|_{(s,\phi)=(s_0,\phi_0)} = 0
    \end{split}\end{align}
    for $0 \leq i \leq k \leq n$.
\end{lemma}

Applying this lemma to the degree $n$ elliptic periodic orbit obtained before, we ensure a rational rotation by performing an arbitrarily small perturbation if the rotation angle is irrational.

We now put all these lemma together in the proof of Theorem 1:
\medskip
\begin{center}
    \textbf{
    Proof of Theorem 1
    }
\end{center}

\begin{proof}
We begin with a boundary $\Omega \in \mathcal D^\infty$ and a fixed $n \in \N$.  By Lemma \eqref{lemma_creating_quadratic_ht}, we find a domain $\Omega_1$ which is $C^\infty$ close to $\Omega$ and contains a hyperbolic $q$-periodic 
point $O_1$ with a quadratic homoclinic tangency, satisfying the conditions described in Lemma \eqref{lemma_creating_quadratic_ht}.
Then by Lemma \eqref{lemma_finding_elliptic_rotation_of_order_n} we find a $C^\infty$ domain $\Omega_2$ which is $C^\infty$ close to  $\Omega_1$ so that the domain $\Omega_2$ contains an elliptic orbit such that the differential of the billiard map at that a point in that orbit is a rotation to order $n$.  Then we use Lemma \eqref{lemma_rotating_the_differential_to_order_n} to obtain a domain $\Omega_3$ in $C^\infty$ which is $C^\infty$ close to $\Omega_2$ and which has an orbit so that the differential of the billiard map at that point is a rotation by a rational angle, to order $n$.  Then, considering the billiard map composed with itself equal to the denominator of that rational angle, the proof is completed.    \end{proof}

In the following sections we will give proofs of the various lemmas and statements mentioned above.  Throughout these proofs we will be using various perturbations of the curvature at specific points on the boundary of the domain to effect the differential of the billiard map.  The proofs of how these perturbations effect the curvature are in Appendix \ref{appendix_all_perturbations}.  

\section{Unfolding an $n^{th}$ order HT generically}\label{sec_perturb_nth_order_ht_generically}
The main lemma we wish to prove in this section is Lemma \eqref{lemma_unfold_generically}.

To begin, we let $O$ be a point in the periodic orbit, and $p$ be a point of homoclinic tangency, and $\vec n$ the normal vector of $W^s(O)$ at $p$.  We also let $p(t)$ for $|t|\leq \delta$ and some small $\delta$ be unit speed parameterization of a section of $W^u(O)$  with $p(0)=p$ (see Figure \ref{fig_unfold_generically1}).

We then start with the following: 

\setcounter{lemma}{2}
\begin{sublemma}\label{sublemma_unfolding_first_pert}
One may construct perturbations for each $0 \leq k \leq n$ such that $p(t)\rightarrow p_\eps(t)$ such that
\begin{align}\begin{split}\nonumber
    p_\eps(t)=p(t)+R\begin{bmatrix}
    O(t^2)\\\eps t^k + \sum_{j=k+1}^{n+1} g_j(\eps)t^j
    \end{bmatrix}
\end{split}\end{align}
for some $C^\infty$ functions $g_j$ which satisfy $g_j(0)=0$, where $R$ is a rotation matrix which maps the vector $(0,1)$ to $\vec n$. Equivalently, these perturbations change $\Phi_\eps$ as
\begin{align}\begin{split}\nonumber
    \Phi_{\eps}(t) = \Phi(t) + \eps t^k + \sum_{j=k+1}^{n+1} g_j(\eps)t^j.
\end{split}\end{align}  
\end{sublemma}

\begin{figure}[h]
\centering
\includegraphics[width=12.5cm, height=6.25cm]{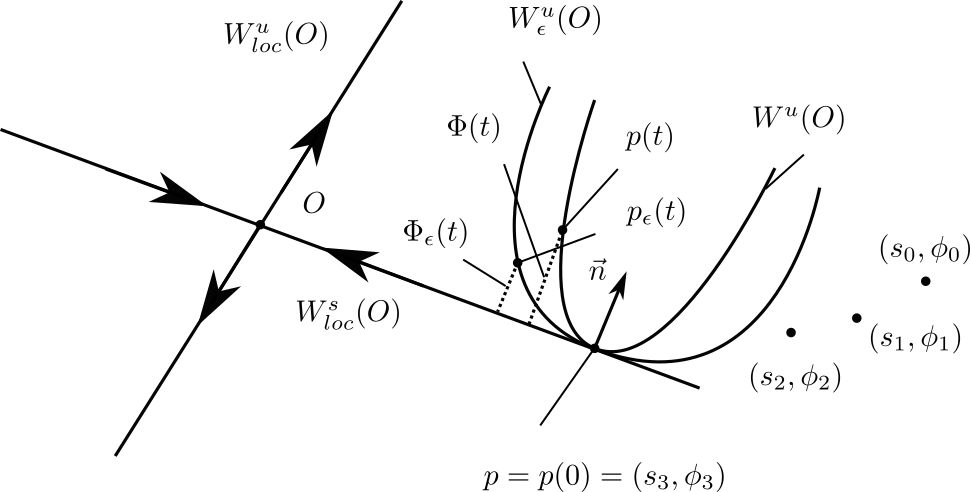}
\caption{Figure showing all the objects described.  We vary in small strips around the points $(s_1,\phi_1),(s_2,\phi_2)$ (corresponding to small neighborhoods of their position on the boundary of the domain) in order to change $\Phi(t)$ so that the derivatives of $\Phi$ at $t=0$ vary independently.}
\label{fig_unfold_generically1}
\end{figure}

We label the perturbations described as $pert_k(\eps)$.  These perturbations can then be used to obtain the main lemma of this section (Lemma \eqref{lemma_unfold_generically}) by the following inductive process:

\begin{sublemma}\label{sublemma_unfolding_final_pert}
    If we have a set of perturbations $pert_k(\eps)$ as described above, then there exist perturbations for each $k$ which have the effect on $\Phi$ as
    \begin{align}\begin{split}\nonumber
        \Phi_\eps(t) = \Phi(t) + \eps t^k + O(t^{n+1}).
    \end{split}\end{align}
\end{sublemma}
\begin{proof}[Proof of lemma 1.2.2]
Perform $pert_0(\eps_0)$ which yields the following $\Phi_{\eps_0}(t)$:

\begin{align}\begin{split}\nonumber
    \Phi_{\eps_0}(t) = \Phi(t) + \eps_0 + \sum_{k=1}^{n+1} g_{0,k}(\eps_0)t^k
\end{split}\end{align}
for some $C^\infty$ functions $g_{0,k},$ which have $g_{0,k}(0)=0$ for each $k$.  Then we perform $pert_1(\eps_1)$ to obtain
\begin{align}\begin{split}\nonumber
    \Phi_{(\eps_0,\eps_1)}(t) = \Phi(t) + \eps_0 + \sum_{k=1}^{n+1} g_{0,k}(\eps_0)t^k + \eps_1t + \sum_{k=1}^{n+1} \tilde g_{1,k}(\eps_1)t^k,
\end{split}\end{align}
again for some $C^\infty$ functions $\tilde g_{1,k},$ which have $\tilde g_{1,k}(0)=0$ for each $k$.  We thus set $\eps_1 = -g_{0,1}(\eps_0)$ and label each $g_{0,k}(\eps_0) + \tilde g_{1,k}(-g_{0,1}(\eps_0)) = g_{1,k}(\eps_0)$ to obtain
\begin{align}\begin{split}\nonumber
    \Phi_{\eps_0}(t) = \Phi(t) + \eps_0 + \sum_{k=2}^{n+1}g_{1,k}(\eps_0)t^k.
\end{split}\end{align}
We then repeat this process $n-2$ more times, each time eliminating the lowest order term in $\Phi_{\eps_0}(t)$ to obtain a total perturbation by the parameter $\eps_0$ which yields
\begin{align}\begin{split}\nonumber
    \Phi_{\eps_0}(t) = \Phi(t) + \eps_0 + O(t^{n+1}).
\end{split}\end{align}
Then, we perturb by a new parameter $\eps_1$, repeating the same process but at one degree higher.  In the end, we obtain a perturbation by the parameter $\eps = (\eps_0,...,\eps_n)$ which gives us
\begin{align}\begin{split}\nonumber
    \Phi_{\eps}(t) = \Phi(t) + \eps_0 + \eps_1t + ... + \eps_nt^n + O(t^{n+1}).
\end{split}\end{align}

\end{proof}
This then allows us to prove Lemma \eqref{lemma_unfold_generically}:
\begin{proof}
Observe that Lemma \eqref{sublemma_unfolding_final_pert} allows us to vary each derivative of $\Phi$ independently up to order $n$.
\end{proof}

So, what remains is to prove Lemma \eqref{sublemma_unfolding_first_pert}.

\begin{proof}
First, we set 
\begin{align}\begin{split}\nonumber
    p&=(s_3,\phi_3)=f^3(s_0,\phi_0)=f^2(s_1,\phi_1)=f(s_2,\phi_2)\\
    p(t)&=(s_3(t),\phi_3(t))=f^3(s_0(t),\phi_0(t))=f^2(s_1(t),\phi_1(t))=f(s_2(t),\phi_2(t)),
\end{split}\end{align}
and we perturb $\kappa$ in a small neighborhood around $s_1$ and $s_2$.  To begin we will do a general perturbation of adding $\Delta\kappa_\eps$ to $\kappa$, and later specify exactly what $\Delta\kappa_\eps$ is.  We also note that we may use Perturbation \ref{lemma_we_can_perturb_around_three_points_instead} in Appendix \ref{appendix_all_perturbations} in order to, without loss of generality, assume the injectivity condition is satisfied at $s_1$ and $s_2$.  

Then, sending $\kappa\rightarrow \kappa_\eps = \kappa+ \Delta\kappa_\eps,$ and defining our perturbation to send $\beta_1 \rightarrow \beta_1 + \Delta\beta_1(\eps,t), l_1 \rightarrow l_1 + \Delta l_1(\eps,t), \beta_2 \rightarrow \beta_2 + \Delta\beta_2(\eps,t), l_2 \rightarrow l_2 + \Delta l_2(\eps,t)$, we have
\begin{align}\begin{split}\nonumber
    df_\eps^3&((s_0(t),\phi_0(t))) = df^3((s_0(t),\phi_0(t))) + 
    \Delta df_\eps^3((s_0(t),\phi_0(t))) \\
    &+ O(\Delta l_1(\eps,t) + \Delta\beta_1(\eps,t) + \Delta l_2(\eps,t) + \Delta\beta_2(\eps,t)) + O(\Delta\kappa_\eps^2(t))
\end{split}\end{align}
where, following the same calculations done in Appendix \ref{appendix_all_perturbations} for Perturbation \ref{perturbation_which_varys_the_first_differential},
\begin{align}\begin{split}
    \Delta &df_\eps^3((s_0(t),\phi_0(t)))  = \\
    &\begin{bmatrix}
    a_1(t)\Delta\kappa_{\eps}(s_1(t))+a_2(t)\Delta\kappa_{\eps}(s_2(t)) & b_1(t)\Delta\kappa_{\eps}(s_1(t))+b_2(t)\Delta\kappa_{\eps}(s_2(t))\\
    c_1(t)\Delta\kappa_{\eps}(s_1(t))+c_2(t)\Delta\kappa_{\eps}(s_2(t)) &
    d_1(t)\Delta\kappa_{\eps}(s_1(t))+d_2(t)\Delta\kappa_{\eps}(s_2(t))
    \end{bmatrix}, \label{expansion_of_HT_perturbation}
\end{split}\end{align}
with
\begin{align}\begin{split}\nonumber
    \begin{bmatrix}
    a_1(t) & b_1(t) \\ c_1(t) & d_1(t)
    \end{bmatrix}
   & =
    df^2((s_1(t),\phi_1(t)))\begin{bmatrix}
    0&0\\2&0
    \end{bmatrix}
    df((s_0(t),\phi_0(t)))\\ \label{abcd_matricies2}
    \begin{bmatrix}
    a_2(t) & b_2(t) \\ c_2(t) & d_2(t)
    \end{bmatrix}
   & =
    df((s_2(t),\phi_2(t)))\begin{bmatrix}
    0&0\\2&0
    \end{bmatrix}
    df^2((s_0(t),\phi_0(t)))
\end{split}\end{align}
First off, we consider the case of $k \geq 1$ (we treat the case $k=0$ later).  For this case, we consider perturbations with the following properties:
\begin{equation}\label{equation_degree_of_perturbation}
\begin{split}
    &\Delta\kappa_\eps(s_i)=...=\Delta\kappa_\eps^{(k-2)}(s_i)=0\\ &\Delta\kappa_\eps^{(k-1)}(s_1)=\eps_1
    \\
    &\Delta\kappa_\eps^{(k-1)}(s_2)=\eps_2
\end{split}
\end{equation}
for $i=1,2$.  Then it is easy to show using $\del_1l(s,s')=-\cos(\phi),\del_2l(s,s')=\cos(\phi')$ that this implies for $i=1,2$ and $j+l=k-1$
\begin{equation}\begin{split}\label{equation_partials_of_l_and_beta_are_0}
    \frac{\del ^{k-1}\Delta l_i(\eps,t)}{\del^ls_0\del^j\phi_0}\bigg|_{t=0}=
    \frac{\del ^{k-1}\Delta \beta_i(\eps,t)}{\del^ls_0\del^j\phi_0}\bigg|_{t=0}=0.
\end{split}\end{equation}
From here we have all we need to expand $f^3_{\eps}((s_0(t),\phi_0(t)))$ in t, up to order $k$.  We do this just for $s_{3,\eps}(s_0(t),\phi_0(t))$, as the other component is similar.  We obtain through Taylor expansion
\begin{align}\begin{split}\nonumber
    s_{3,\eps}(t) = s_{3,\eps}(0) + \sum_{l=1}^k \frac{t^l}{l!}\sum_{j=0}^l\Bigg\{\frac{\del^l s_{3,\eps}(t)}{\del s_0^j\del\phi_0^{l-j}}\bigg|_{t=0}(s_0'(0))^j(\phi_0'(0))^{l-j}{l \choose j}\\
    +O(t^{k+1}) + (\text{derivatives of }s_{3,\eps}\text{ of order less than }l)\Bigg\}
\end{split}\end{align}
and from Equation \eqref{expansion_of_HT_perturbation}, we have (including the $j=0$ case, which one can check using Relations \eqref{equation_reltion_between_partials})
\begin{equation}\begin{split}\label{definition_of_a_b_c_d}
    \Delta\frac{\del^l s_{3,\eps}(t)}{\del s_0^j\del\phi_0^{l-j}}\bigg|_{t=0} & = a_1(0)\Delta\kappa^{(l-1)}_{\eps}(s_1)\bigg(\frac{\del s_1}{\del s_0}\bigg)^{j-1}\bigg(\frac{\del s_1}{\del \phi_0}\bigg)^{l-j}\\
    & +a_2(0)\Delta\kappa^{(l-1)}_{\eps}(s_2)\bigg(\frac{\del s_2}{\del s_0}\bigg)^{j-1}\bigg(\frac{\del s_2}{\del \phi_0}\bigg)^{l-j}\\
    & + (\text{lower derivatives of }\Delta\kappa_\eps)\\
    & + (\text{derivatives up to order }l-1\text{ of }\Delta l_1)\\
    & + (\text{derivatives up to order }l-1\text{ of }\Delta l_2)\\
    & + (\text{derivatives up to order }l-1\text{ of }\Delta \beta_1)\\
    & + (\text{derivatives up to order }l-1\text{ of }\Delta \beta_2),
\end{split}\end{equation}
So we get that with Equations \eqref{equation_degree_of_perturbation} and \eqref{equation_partials_of_l_and_beta_are_0} most of these terms are 0, and in fact we are left with only the final term:
\begin{align}\begin{split}\nonumber
    s_{3,\eps}(t) = s_3(t) 
    & + \frac{t^k}{k!}\bigg\{\sum_{j=0}^k\bigg(a_1(0)\eps_1\bigg(\frac{\del s_1}{\del s_0}\bigg)^{j-1}\bigg(\frac{\del s_1}{\del \phi_0}\bigg)^{k-j}\\
    + & a_2(0)\eps_2\bigg(\frac{\del s_2}{\del s_0}\bigg)^{j-1}\bigg(\frac{\del s_2}{\del \phi_0}\bigg)^{k-j}\bigg)
    (s_0'(0))^j(\phi_0'(0))^{k-j}{k \choose j}\bigg\}.
\end{split}\end{align}
Simplifying (one may check the equivalences), we get
\begin{align}\begin{split}\nonumber
    \Delta s_{3,\eps}(t) = \frac{t^k}{k!}\bigg[\eps_1a_1(0)\bigg(\frac{\del s_1}{\del s_0}\bigg)^{-1}(s_1'(0))^k + \eps_2a_2(0)\bigg(\frac{\del s_2}{\del s_0}\bigg)^{-1}(s_2'(0))^k\bigg].
\end{split}\end{align}
Similarly we get
\begin{align}\begin{split}\nonumber
    \Delta \phi_{3,\eps}(t) = \frac{t^k}{k!}\bigg[\eps_1c_1(0)\bigg(\frac{\del s_1}{\del s_0}\bigg)^{-1}(s_1'(0))^k + \eps_2c_2(0)\bigg(\frac{\del s_2}{\del s_0}\bigg)^{-1}(s_2'(0))^k\bigg].
\end{split}\end{align}
Note that we chose to use $a_1,a_2,c_1,c_2$.  However, we could use the other variables $b_1,b_2,d_1,d_2$ and perform the same analysis to get equivalent expressions.  These can all be related by the fact that for $i=1,2$
\begin{align}\begin{split}\label{equation_reltion_between_partials}
    a_i(0) &= b_i(0)\bigg(\frac{\del s_i}{\del \phi_0}\bigg)^{-1}\bigg(\frac{\del s_i}{\del s_0}\bigg)\\
    c_i(0) &= d_i(0)\bigg(\frac{\del s_i}{\del \phi_0}\bigg)^{-1}\bigg(\frac{\del s_i}{\del s_0}\bigg)
\end{split}\end{align}
Thus, we can choose $\eps_1,\eps_2$ such that we perturb $(s_3(t),\phi_3(t))$ in the direction of $\vec n$ if we have the following:
\begin{align}\begin{split}\label{equation_unfolding_determinant}
    \det\begin{bmatrix}
    a_1(0)\bigg(\frac{\del s_1}{\del s_0}\bigg)^{-1}(s_1'(0))^k&
    a_2(0)\bigg(\frac{\del s_2}{\del s_0}\bigg)^{-1}(s_2'(0))^k\\
    c_1(0)\bigg(\frac{\del s_1}{\del s_0}\bigg)^{-1}(s_1'(0))^k&
    c_2(0)\bigg(\frac{\del s_2}{\del s_0}\bigg)^{-1}(s_2'(0))^k\\
    \end{bmatrix} \neq 0.
\end{split}\end{align}
From here, we suppose that $s_1'(0)\neq0$ and $s_2'(0)\neq 0$.  We prove that that is indeed the case in Lemma \eqref{sublemma_unfolding_derivatives_nonzero}.  Now, given these inequalities, Equation \eqref{equation_unfolding_determinant} is equivalent to 
\begin{align}\begin{split}\nonumber
    \det \begin{bmatrix}
    a_1(0) & a_2(0) \\ c_1(0) & c_2(0)
    \end{bmatrix} \neq 0,
\end{split}\end{align}
or from our relations, equivalently
\begin{align}\begin{split}\nonumber
    \det \begin{bmatrix}
    a_1(0) & b_2(0) \\ c_1(0) & d_2(0)
    \end{bmatrix} \neq 0.
\end{split}\end{align}
Now, from Equation \eqref{definition_of_a_b_c_d} we have
\begin{align}\begin{split}\nonumber
    \begin{bmatrix}
    a_1(0) & b_2(0) \\ c_1(0) & d_2(0)
    \end{bmatrix}
    =
    \begin{bmatrix}
    \frac{\del s_1}{\del s_0}\frac{\del s_3}{\del \phi_1} &&
    \frac{\del s_2}{\del \phi_0}\frac{\del s_3}{\del \phi_2} \\ 
    \frac{\del s_1}{\del s_0}\frac{\del \phi_3}{\del \phi_1} &&
    \frac{\del s_2}{\del \phi_0}\frac{\del \phi_3}{\del \phi_2}
    \end{bmatrix},
\end{split}\end{align}
so taking the determinant we obtain
\begin{align}\begin{split}\nonumber
    &\det{\bigg(
    \begin{bmatrix}
    a_1(0) & b_2(0) \\ c_1(0) & d_2(0)
    \end{bmatrix}
    \bigg)}\\
    &=
    \frac{\del s_1}{\del s_0}\frac{\del s_3}{\del \phi_1}\frac{\del s_2}{\del \phi_0}\frac{\del \phi_3}{\del \phi_2} - \frac{\del s_2}{\del \phi_0}\frac{\del s_3}{\del \phi_2}\frac{\del s_1}{\del s_0}\frac{\del \phi_3}{\del \phi_1}\\
    &= \frac{\del s_1}{\del s_0}\frac{\del s_2}{\del \phi_0}\bigg( \frac{\del s_3}{\del \phi_1}\frac{\del \phi_3}{\del \phi_2} - \frac{\del s_3}{\del \phi_2}\frac{\del \phi_3}{\del \phi_1}\bigg)\\
    &= \frac{\del s_1}{\del s_0}\frac{\del s_2}{\del \phi_0}\bigg(
    \bigg(\frac{\del s_3}{\del s_2}\frac{\del s_2}{\del \phi_1} + \frac{\del s_3}{\del \phi_2}\frac{\del \phi_2}{\del \phi_1}\bigg) \frac{\del \phi_3}{\del \phi_2} - 
    \frac{\del s_3}{\del \phi_2}\bigg(
    \frac{\del \phi_3}{\del s_2}\frac{\del s_2}{\del \phi_1} + \frac{\del \phi_3}{\del \phi_2}\frac{\del \phi_2}{\del \phi_1}
    \bigg)
    \bigg)\\
    &=\frac{\del s_1}{\del s_0}\frac{\del s_2}{\del \phi_0}\frac{\del s_2}{\del \phi_1}\det{\bigg(df(x_2)\bigg)}\\
    &\neq 0.
\end{split}\end{align}
Thus we may find $\eps_1,\eps_2$ so that our perturbation moves $W^u(O_1)$ in a small neighborhood of the point of tangency with $W^s(O_2)$ by $\eps t^k+O(t^{k+1})$ in the direction $\vec n$.  In order to obtain that the changes to order $k+1$ and above are $C^\infty$ smooth in $\eps$, we simply choose perturbations which are $C^\infty$ and which satisfy Equations \eqref{equation_degree_of_perturbation}, and observe that at $\eps=0$ there is no change $\Phi$.  This finishes the proof of the lemma for cases $k \geq 1$.

Now we deal with the case of $k=0$.  For this case, we are able to achieve the desired result by perturbing around three points of the periodic orbit $\mathcal O = (O,f(O),...,f^{q-1}(O))$ which satisfy the injectivity condition with respect to $\mathcal O$ and the homoclinic orbit containing $p$.  The perturbation we perform is from Appendix \ref{appendix_all_perturbations}, Perturbation \ref{perturbation_which_lets_you_move_unstable_manifold}, which allows us to change the angles of the invariant manifolds of $O$.  This allows us to in a sense "raise" or "lower" the height of the moment of tangency (which in turn of course destroys it and either yields no point of tangency, or two points of transverse homoclinic intersection) (see Figure \ref{fig_unfold_generically2}).  

\begin{figure}[h]
\centering
\includegraphics[width=12.5cm, height=6.25cm]{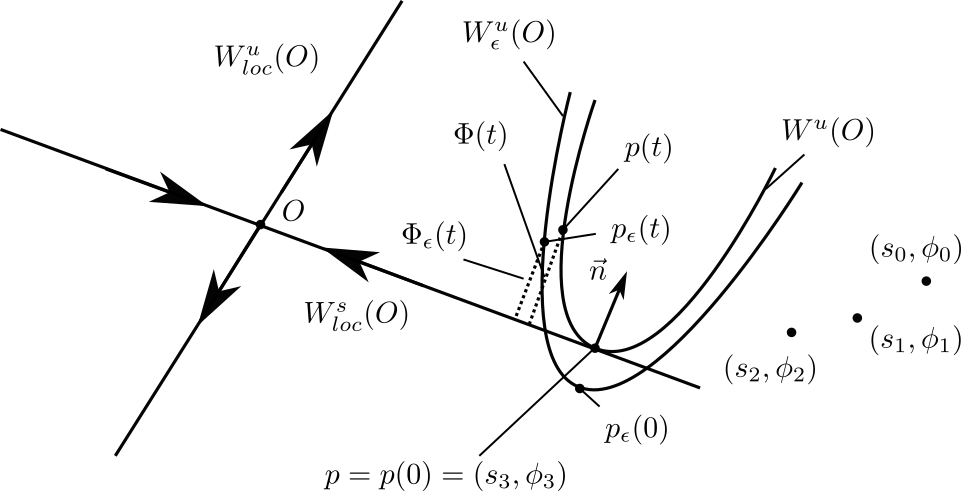}
\caption{We vary in small strips around points in the orbit of $O$ in order to change $\Phi(t)$ so that the position of $\Phi(0)$ changes in the direction of $\vec n$.}
\label{fig_unfold_generically2}
\end{figure}

The exact details are worked out in Appendix \ref{appendix_all_perturbations}, Perturbation \ref{perturbation_which_lets_you_move_unstable_manifold}, which allows us to obtain exactly
\begin{align}\begin{split}\nonumber
    \Phi_\eps(t) = \Phi(t) + \eps + O(\eps^2).
\end{split}\end{align}
\\
\end{proof}

Now, we deal with the following technical issue that came up during the proof:
\begin{sublemma}\label{sublemma_unfolding_derivatives_nonzero}
    With the given definitions, we have
    \begin{align}\begin{split}\nonumber
        s_1'(0) \neq 0, s_2'(0) \neq 0.
    \end{split}\end{align}
\end{sublemma}
\begin{proof}
    This follows from our manifolds near $O$ having a non-zero angle with the vertical line.  Recall that the point $(s_3(0),\phi_3(0))$ is a point of homoclinic tangency between $W^u(O)$ and $W^s(O)$.  This implies that $(s_2(0),\phi_2(0))$ is a point of homoclinic tangency between  $W^u(f^{-1}(O))$ and $W^s(f^{-1}(O))$, and $(s_1(0),\phi_1(0))$ is a point of homoclinic tangency between  $W^u(f^{-2}(O))$ and $W^s(f^{-2}(O))$.
    
    Thus, if $(s_1'(0),\phi_1'(0))$ is a tangent vector of $W^u(f^{-2}(O))$ at the point of intersection with $W^s(f^{-2}(O))$ (and similarly $(s_2'(0),\phi_2'(0))$ for $W^u(f^{-1}(O))$), which means it is also a tangent vector of $W^s(f^{-2}(O))$.  Thus, if $s_1'(0) = 0$, this tangent vector would be vertical, which would contradict our assumption that the manifolds (in particular the stable manifold) have non-zero angles with the vertical line near the periodic point.
\end{proof}

\section{Satisfying the Injectivity Condition for a  HT of order $n$}\label{sec_verify_injectivity_condition}

In this section, we prove Lemma \eqref{lemma_on_injectivity}.  
\begin{proof}
We begin by considering a periodic orbit containing the point $O=(s_0,\phi_0)$, with an $n^{th}$ order homoclinic tangency at $p$.  Using Perturbation \ref{perturbation_which_lets_you_move_unstable_manifold} from Appendix \ref{appendix_all_perturbations}, we perturb around a small strip $U$ of width $\sigma$ around $O$ in order to embed our map $f$ into a smooth deformation $(f_\tau)_{|\tau| \ll 1}$ which leaves the periodic orbit fixed and satisfies $f_0 = f$ as well as that the difference in the direction normal to $W_{loc}^s(O)$ at $p$ between $W_{loc,\tau}^u(O)$ (the local unstable manifold corresponding to $f_\tau$) and $W_{loc}^u(O)$ in a small neighborhood of $p$ is $\tau + O(\tau^2)$ (See Figure 7).

We now fix some $\tau$.  Then, we use the same type of perturbation, but while considering the inverse of the billiard map $f_\tau$ in order to move the stable manifold.  This lets us embed $f_\tau$ into a smooth deformation $(f_{\tau,\eps})_{|\eps| \ll 1}$, with $f_{\tau,0}=f_{\tau}$, with the local stable manifold $W_{loc,(\tau,\eps)}^s$ corresponding to $f_{\tau,\eps}$ such that the difference from $W_{loc,\tau}^s(O)$ in the direction of the normal to $W_{loc}^s(O)$ at $p$ is $\eps + O(\eps^2)$.  Thus, choosing $\eps$ as a function of $\tau$, we may find a specific map $\tilde f_\tau$ which still has a moment of tangency, but now in a different location than the tangency for $f$.  We may arrange for the position to be a distance $\tau + O(\tau^2)$ away from $p$ in the direction perpendicular to $W_{loc}^s(O)$ at $p$.  

We label this new moment of tangency $p_\tau$, and the manifolds of $\tilde f$ by $W_\tau^u(O)$ and $W_\tau^s(O)$, and locally by $W_{loc,\tau}^u(O)$ and $W_{loc,\tau}^s(O)$ (See Figure \ref{fig_moving_moment})

\begin{figure}[h]
\centering
\includegraphics[width=10.5cm, height=6.25cm]{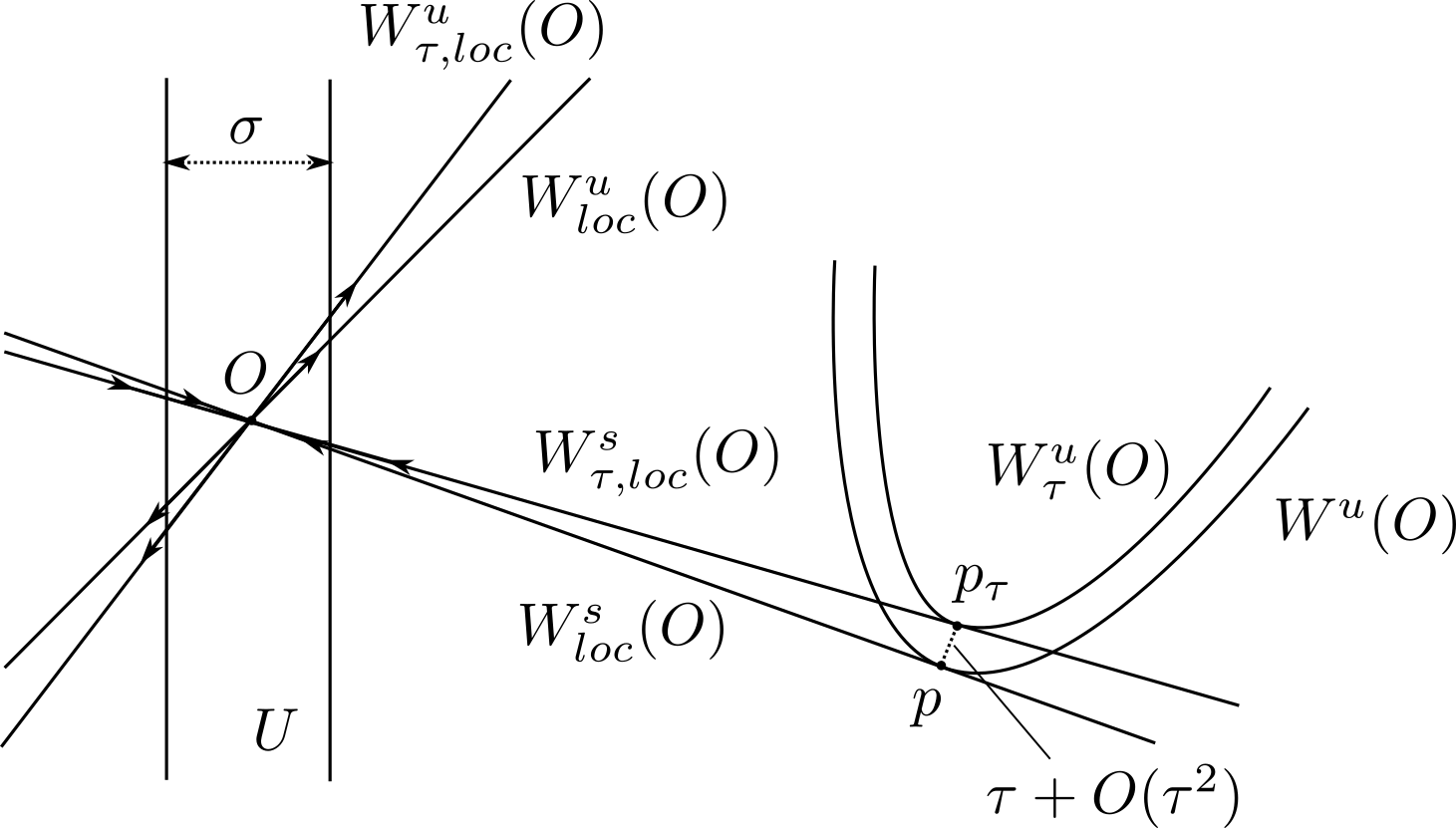}
\caption{Moving the moment of tangency.}
\label{fig_moving_moment}
\end{figure}

Note that the HT at $p_\tau$ is of the same order as the HT at $p$ under the original map $f$, up to small error.  These errors however can be eliminated exactly by performing the types of perturbations done in the previous section.  

We now show precisely how the homoclinic orbit under the map $\tilde f_\tau$ varies from the homoclinic orbit under $f$ as we vary $\tau$.  To begin, given any $\sigma$, there is a number $m \in \Z$ so that $f^{-qm}(p) \notin U$ and $f^{-q(m+1)}(p) \in U$.  Similarly, there exists an $M \in \Z$ so that $f^{qM}(p) \notin U$ and $f^{q(M+1)}(p) \in U$.  Note that as $\sigma \rightarrow 0$, then $m,M\rightarrow \infty$.  

We examine how the new homoclinic points $f^{qk}(p_\tau)$ change from $f^{qk}(p)$, for each $-m \leq k \leq M$.

\setcounter{lemma}{3}
\setcounter{sublemma}{0}

\begin{sublemma}\label{sublemma_rate_at_which_points_of_ht_move}
For each $-m \leq k \leq M$, we have
\begin{align}\begin{split}\label{equation_non_degeneracy_in_injectivity_condition}
    \frac{\del}{\del \tau}\bigg(\tilde f_\tau^{qk}(p)\cdot \vec n(f^{qk}(p)) \bigg)\bigg|_{\tau=0} &= \lambda^{-k}
    \\
    \frac{\del}{\del \tau}\bigg(\tilde f_\tau^{qk}(p)\cdot \vec \nu(f^{qk}(p)) \bigg)\bigg|_{\tau=0} &= 0,
\end{split}\end{align}
where $\lambda$ is the eigenvalue of $df^q(O)$ greater than 1, $\vec n(f^{qk}(p))$ is the unit normal vector of $W_{loc}^s(O)$ at $f^{qk}(p)$, and $\vec \nu(f^{qk}(p))$ is the unit tangent vector of $W_{loc}^s(O)$ at $f^{qk}(p)$.
\end{sublemma}
\begin{proof}
    We observe that this is true for $k=0$ by our construction of the map $\tilde f_\tau$.  Then, by the lambda lemma (or by considering the map in Birkhoff normal form), we have that the other moments of tangency move the following distances
    \begin{align}\begin{split}\nonumber
        ||f^{qk}(p_\tau) - f^{qk}(p)|| = \frac{\tau}{\lambda^k} + O(\tau^2)
    \end{split}\end{align}
    in the direction of $\vec n(f^{qk}(p))$.
\end{proof}

Now, for each $\sigma > 0$, label the embedding which satisfies Equation \eqref{equation_non_degeneracy_in_injectivity_condition} for $-m \leq k \leq M$ by $(f_{\sigma,\tau})_{\tau \ll 1}$.  Then, suppose that for each such embedding, the injectivity condition fails.  We examine these embeddings.  For the injectivity condition to fail, this means that for each $-m \leq k \leq M$ (except for at most two numbers), either there is an additional point $q_{\sigma,\tau}^k$ which is on the same vertical line going through $f_{\sigma,\tau}^{qk}(p)$ (case 1), or there is an infinite sequence of points $(q_{\sigma,\tau}^{k,j})_{j \in \N}$ which accumulates onto the vertical line going through $f_{\sigma,\tau}^{qk}(p)$ (case 2), where these points $q_{\sigma,\tau}^k$ and $q_{\sigma,\tau}^{k,j}$ belong to either the hyperbolic periodic orbit or to the homoclinic orbit.

Now since our hyperbolic periodic orbit is near the boundary, and the points $f_{\sigma,\tau}^{qk}(p)$ are all in the invariant manifolds of $O$, the points $q_{\sigma,\tau}^k$ and $q_{\sigma,\tau}^{k,j}$ cannot be a part of the periodic orbit.  Hence they must be a part of the homoclinic orbit, and since they also lie outside of the region $U$, there must be some $-m \leq n_k \leq M$ such that $q_{\sigma,\tau}^k = f_{\sigma,\tau}^{qn_k}(p)$, and $q_{M,\tau}^{k,j} = f_{\sigma,\tau}^{qn_{k,j}}(p)$  Observe this immediately implies the second case is impossibility, since there are only a finite number of points of the homoclinic orbit which are not in the region $U$.

\begin{figure}[h]
\centering
\includegraphics[width=11cm, height=8cm]{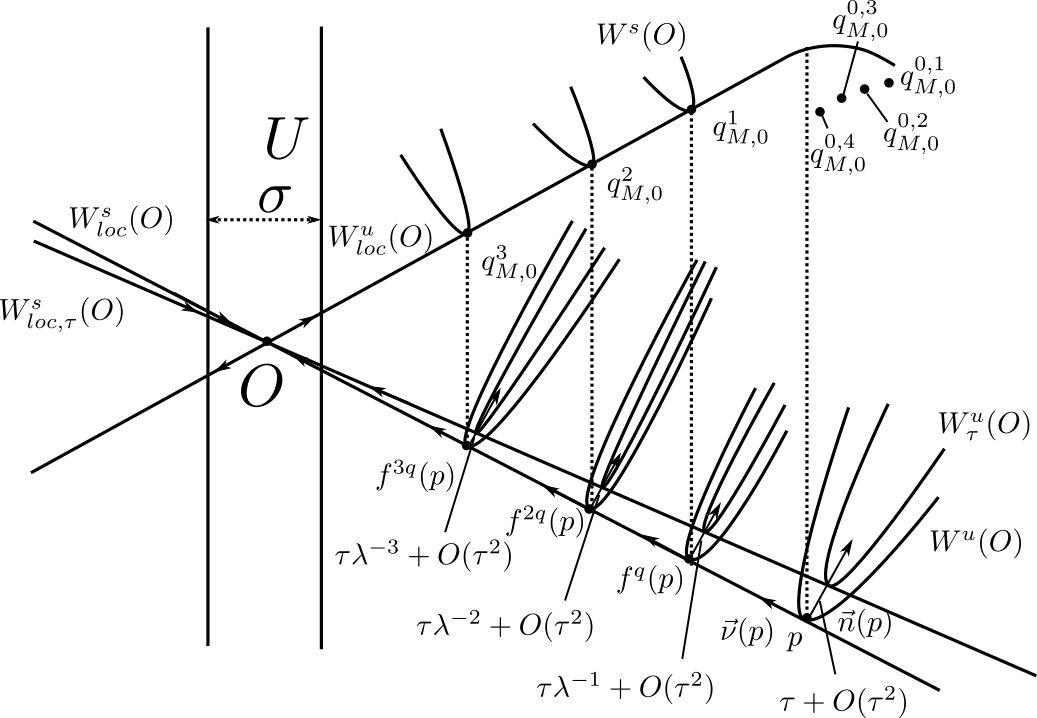}
\caption{Dynamics near $O$, assuming the injectivity condition fails.  In the figure, the injectivity condition fails in the way we labeled the second case around the point $p$, and in the first case around the points $f^q(p),f^{2q}(p),$ and $f^{3q}(p)$.}
\label{fig_injectivity_contradiction}
\end{figure}

Thus, we proceed to analyze the first case. We observe 
\begin{sublemma}\label{sublemma_they_move_at_same_horizontal_rate}
    If the injectivity condition fails for all $\tau$, the
    rate at which $f^{qk}(p_\tau)$ moves in the horizontal direction must match the rate at which $f^{qn_k}(p)$ moves in the horizontal direction.
\end{sublemma}
\begin{proof}
If the injectivity condition fails for all $\tau$, then as we vary $\tau$ each pair of points $(f_{\sigma,\tau}^{qk}(p),f_{\sigma,\tau}^{qn_k}(p))$ remain on the same vertical line.  This implies exactly the statement of the lemma, since otherwise they would split.
\end{proof}

Now, $f^{qk}(p)$ accumulates to $O$ as $k\rightarrow\infty$, which implies $f^{qn_k}(p)$ accumulates to $O$ as well.  Further, since $n_k < k$, we must have $n_k \rightarrow -\infty$ and $f^{qn_k}(p)$ approaches $O$ along $W_{loc}^u(O)$.  Then, we observe 
\begin{sublemma}\label{sublemma_they_approach_the_same_movement_direction}
    As $n_k \rightarrow -\infty$, the points $f^{n_k}(p)$ approaches $O$ along $W_{loc}^u(O)$.  Additionally,
    the normal vector of $W_{loc}^s(O)$ at $f^{n_k}(p)$ approaches the vector perpendicular to the the eigenvector of $df^q(O)$ corresponding to the eigenvalue greater than 1.  Similarly the normal vector of $W_{loc}^s(O)$ at $f^{qk}(p)$ approaches the vector perpendicular to the eigenvector of $df^q(O)$ corresponding to the eigenvalue less than 1.

    More precisely, if we examine the dynamics in a ball $B_\delta$ of radius $\delta$ centered at $O$, then for $\sigma \ll \delta$ there exist $k$ so that $f^{qk}(p)$ and $f^{n_k}(p)$ are within the $B_\delta$ and outside of $U$, and, within $B_\delta$, the normal vector of $W_{loc}^s(O)$ at $f^{n_k}(p)$ is the vector perpendicular to the the eigenvector of $df^q(O)$ corresponding to the eigenvalue greater than 1 with error of size $O(\delta)$ (and similarly the normal vector of $W_{loc}^s(O)$ at $f^{qk}(p)$ is the vector perpendicular to the eigenvector of $df^q(O)$ corresponding to the eigenvalue less than 1 with error of size $O(\delta)$).
\end{sublemma}
\begin{proof}
These statements follow from considering the problem in Birkhoff normal form and Equation \eqref{equation_non_degeneracy_in_injectivity_condition}.
\end{proof}

Now, we may arrive at a contradiction.  Lemma \eqref{sublemma_they_approach_the_same_movement_direction} implies that as $k \rightarrow \infty$, the direction each $f^{qn_k}(p_\tau)$ moves as $\tau$ varies approach the same direction, with $O(\delta)$ error.  Similarly for the points $f^{qk}(p)$.  This combined with Lemma \eqref{sublemma_they_move_at_same_horizontal_rate} implies the rate at which $f^{qk}(p_\tau)$ and $f^{qn_k}(p_\tau)$ move as $\tau$ varies must be the same up to error of size $O(\delta)$.  However, Lemma \eqref{sublemma_rate_at_which_points_of_ht_move} gives that future iterations move at different rates than past iterations based on $\lambda$.  So, we examine two pairs for some $k$.  We have, combining Lemma \eqref{sublemma_rate_at_which_points_of_ht_move} and Lemma \eqref{sublemma_they_move_at_same_horizontal_rate}
\begin{align}\begin{split}\nonumber
    &\frac{\del}{\del \tau}\bigg(\tilde f_\tau^{qk}(p)\cdot \hat n(f^{qk}(p)) \bigg)\bigg|_{\tau=0} = \lambda^{-k}\cos(\theta_s) +O(\delta)
    \\
    =&\frac{\del}{\del \tau}\bigg(\tilde f_\tau^{qn_k}(p)\cdot \hat n(f^{qk}(p)) \bigg)\bigg|_{\tau=0} = \lambda^{-n_k}\cos(\theta_u) + O(\delta),
\end{split}\end{align}
where $\hat n(f^{qk}(p))$ is horizontal component of $\vec n(f^{qk}(p))$ and $\hat n(f^{qk}(p))$ is horizontal component of $\vec n(f^{qk}(p))$, and $\theta_s$ is the angle the vector perpendicular to the eigenvector corresponding to the eigenvalue less than 1 makes with the horizontal axis, and $\theta_u$ is the angle the vector perpendicular to the eigenvector corresponding to the eigenvalue greater than 1 makes with the horizontal axis.

However, this cannot be true for all $k$, since as $k\rightarrow \infty$, we have $n_k \rightarrow -\infty$, and we may choose $\delta$ small for large $k$. Thus, we have a contradiction, and therefore must have there exist arbitrarily small $\tau$ so that the periodic orbit and the homoclinic orbit under the map $\tilde f_\tau$ satisfy the injectivity condition with respect to each other.

\end{proof}

\section{Construction of the tower}\label{sec_constructing_the_tower}
The main result of this section is to prove Lemma \eqref{lemma_finding_elliptic_rotation_of_order_n}.  To prove this lemma, we follow arguments from \cite{GST}, adapted to our setting of billiard orbits.  By Lemma \eqref{lemma_unfold_generically}, we may first off embed our map $f$ into a one-dimensional family of billiard maps $f_\tau$ parameterized by $\tau$ and with $f_0=f$, such that the quadratic HT of $f$ unfolds generically as $\tau$ varies.  we also label the periodic point associated with the quadratic HT by $O$. From here, the main ideas in proving the lemma are as follows:
\begin{enumerate}
    \item Find arbitrarily close to 0 values of $\tau$ so that $f_\tau$ has for periodic point $O$ an associated quadratic HT which unfolds generically as $\tau$ varies, as well as a homoclinic orbit corresponding to a transverse intersection of the invariant manifolds of $O$.
    \item For any small $\tau_0$ as above, find another map $\tilde f$ which is $C^\infty$ close to $f_{\tau_0}$ such that $\tilde f$ has an additional (distinct) quadratic HT tangency, also associated to the periodic point $O$.
    \item Inductively repeat a process which takes a map $\tilde f_k$ which has a $k$-order HT and a quadratic HT (both associated to the same periodic point $O$) and produces a map $f_{k+1}$ which has a HT of order $k+1$ and is arbitrarily $C^\infty$ close to $\tilde f_k$, then find another map $\tilde f_{k+1}$ which is arbitrarily $C^\infty$ close to $f_{k+1}$ and which has both a HT of order $k+1$ and a quadratic HT, both associated to $O$.
    \item After using the previous item $r$ times to obtain a map $f_{r}$ which has an $r^{th}$ order HT where $r=2n+5$, we find arbitrarily close to $f_r$ in the $C^n$ topology a map $\bar f$ in $C^\infty$ which has an elliptic periodic orbit of order $n$.
\end{enumerate}
We note two things.  The first is that throughout these steps, we occasionally arrive at maps which have HT of order some high order, say $n$, which we then want to embed into families of maps for which the $n$-order HT unfolds generically.  This is able to be done by Lemma \eqref{lemma_unfold_generically}.

The second thing we note is that we need at each step the injectivity condition to be satisfied.  This is able to be done by using the perturbation described in Lemma \eqref{lemma_on_injectivity}.

We now prove each point in our outline, which proves the theorem.
\begin{proof}
The first item in our outline follows from Lemma 2 of \cite{GST}, which we state here, applied now to our case of billiard maps:

\setcounter{lemma}{4}
\setcounter{sublemma}{0}

\begin{sublemma}
(Lemma 2 from \cite{GST}) If $f_\tau$ is a one-parameter family of $C^\infty$ billiard maps and $f_0$ has a saddle point $O$ with an orbit of quadratic homoclinic tangency which unfolds generically as $\tau$ varies, then arbitrarily close to $\tau = 0$ there exists a value of $\tau$ for which the billiard map $f_\tau$ has an orbit of quadratic homoclinic tangency (which unfolds generically as $\tau$ varies) and a secondary homoclinic orbit corresponding to a transverse intersection of the invariant manifolds of $O$.
\end{sublemma}

\begin{figure}[h]
\centering
\includegraphics[width=9cm, height=6.25cm]{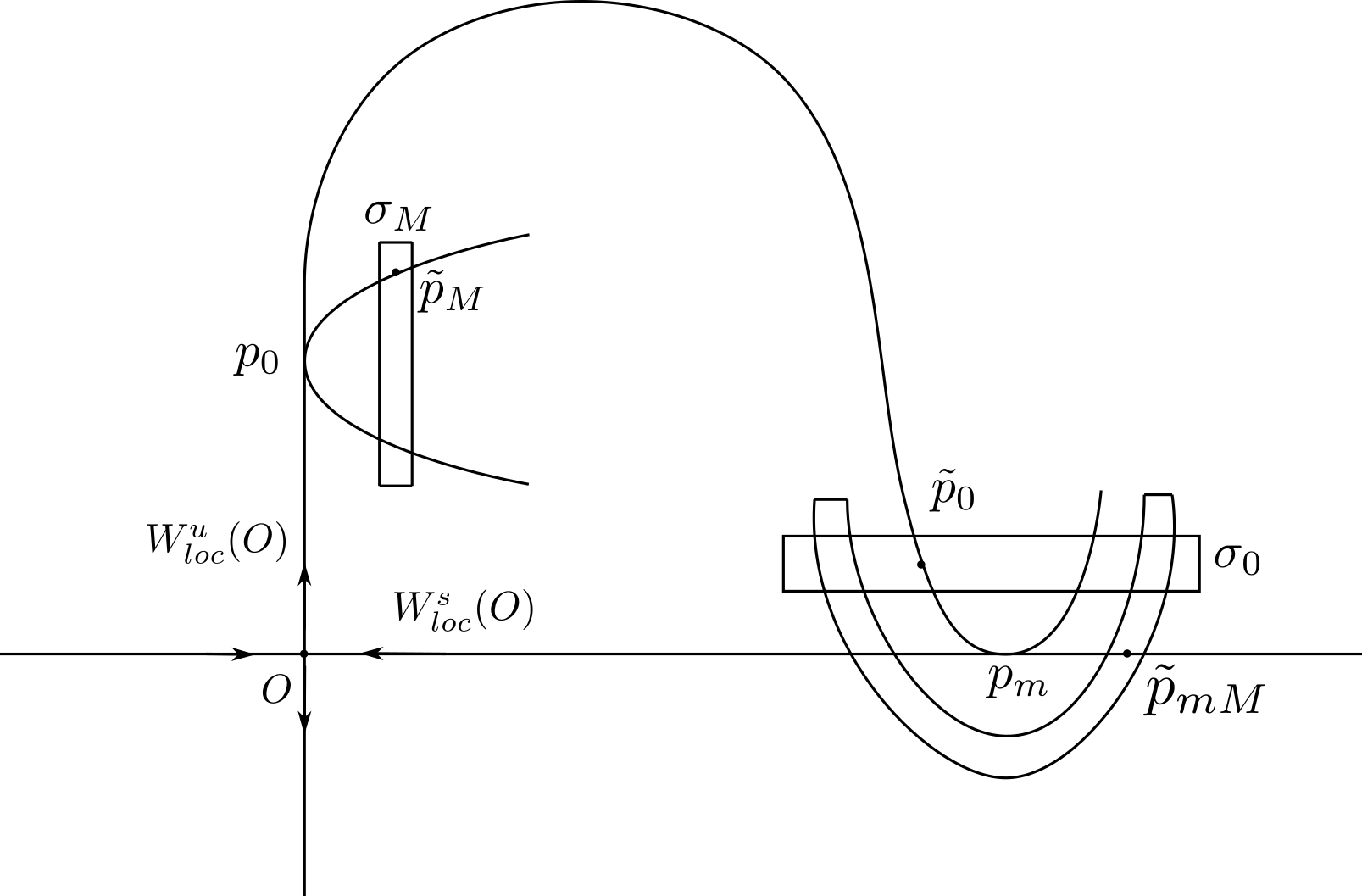}
\caption{An example case of how to obtain a secondary transverse intersection.  In this case, it already exists for $\tau=0$.  The box $\sigma_0$ is mapped after $M$ iterations of the map $f^q$ to $\sigma_M$.  This crosses the stable manifold near $p_0$, so there is some secondary $\tilde p_0$ in the unstable manifold which after $M$ iterations of $f^q$ is in the stable manifold at $\tilde p_M$.  Examining then $m$ more iterations, the point $\tilde p_{mM}$ corresponds to a secondary point of transverse intersection.  In other cases, we perturb to find such a situation.}
\end{figure}

The proof of this lemma in \cite{GST} only requires that we first embed our map into a one-parameter family of two-dimensional $C^\infty$ differomorphisms such that as $\tau$ varies the quadratic homoclinic tangency unfolds generically.  It was known before that analogs of this lemma were true in various settings (\cite{palis_takens},\cite{Newhouse4},\cite{GST2}), but what \cite{GST} contributed was the fact that even if the family of maps is conservative, there are still values of $\tau$ arbitrarily close to $0$ such that a secondary homoclinic orbit is produced.  Thus, since our billiard maps are conservative, we may apply this lemma.

From the existence of a transverse homoclinic orbit, we obtain a horseshoe $\Lambda$ containing $O$ and many points close to $W^s(O)$.  Each of the points in this horseshoe are saddle periodic points (with different periods).

For completeness we detail the construction of this horseshoe (see Figure \ref{figure_horseshoe}). Suppose we have the points $q_0,q_1,...q_M$ which are points of transverse intersection of $W^u(O)$ and $W^s(O)$ with $q_0=f^N(p), q_1=f^{N+1}(p),...,q_M=f^{N+M}(p)$.  We consider a rectangle $R_n$, with height $h_n$ and width $w_n$ centered at $O$, where $h_n \ll 1$ and $w_n$ is large enough so that $R_n$ contains $q_0$ in its interior.  Then $f^n(R_n)$ is a rectangle with height $h_n\lambda^n$ and width $w_n\lambda^{-n}$ centered at 0.  We choose our parameters so that $p=(0,1)$ is near the top of $f^n(R_n)$, i.e. we want $h_n\lambda^n/2 > 1$.  Then we note that $f^{N+M+n}(R_n)$ contains $q_0,...,q_M$.  Our horseshoe is then produced by $R_n$ with the map $f^{N+M+n}$, and contains points $O_i$ arbitrarily close to each $q_i$, as well as $O$.

\begin{figure}[h]
\centering
\includegraphics[width=9cm, height=6.25cm]{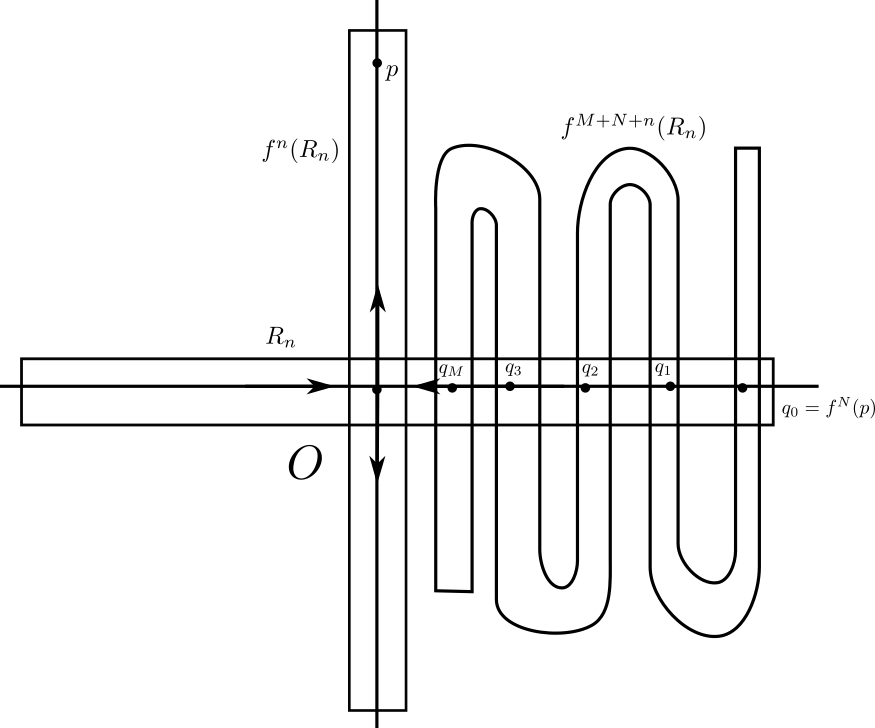}
\caption{The construction of a horseshoe.}
\label{figure_horseshoe}
\end{figure}

We now use two more lemmas in order to accomplish the second item in our outline.

From the map which has a transverse homoclinic orbit containing points $q_0$ and $q_1$, we may find close to $q_0$ and $q_1$ saddle points $O_0$ and $O_1$ respectively so that there is a quadratic heteroclinic tangency between $W^u(O_1)$ and $W^s(O_0)$, and a transverse heteroclinic intersection between $W^u(O_0)$ and $W^s(O_1)$ such that the corresponding heteroclinic cycle are what \cite{GST}, \cite{gavrilov_shilnikov} define as "of the third class" (which we also define momentarily).  This follows from Lemma 3 of \cite{GST}, which we mention here, again applied to the case of billiard maps.

\begin{sublemma}
(Lemma 3 from \cite{GST}) Under the conditions of lemma 2, arbitrarily close to $\tau = 0$, there exist values of $\tau$ for which the billiard map $f_\tau$ has a non-trivial transitive hyperbolic set $\Lambda$ which includes the point $O$ and two saddle periodic points $O_1$ and $O_2$ such that $W^u(O_2)$ and $W^s(O_1)$ have a quadratic heteroclinic tangency which unfolds generically as $\tau$ varies. In $\Lambda$ there exists also an orbit of transverse intersection of $W^u(O_1)$ and $W^s(O_2)$ such that the corresponding heteroclinic cycle belongs to the third class.
\end{sublemma}

The meaning of "of the third class" is a geometric one, relating to whether or not the quadratic tangency hits the stable manifold from above or below.  More concretely, when we define our map from $W^u(O_2)$ to $W^s(O_1)$ to send $(x_2,y_2)$ near $(0,y_2^-)$ to $(x_1,y_1)$ near $(x_1^+,0)$ such that

\begin{align}\begin{split}\nonumber
    x_1 - x_1^+ = ax_2 + b(y_2 - y_2^- ) + ..., \qquad y_1 = cx_2 + d(y_2 - y_2^- )^2 + ...
\end{split}\end{align}

while our map from $W^u(O_1)$ to $W^s(O_2)$ maps $(0,y_1^-)$ to $(x_2^+,0)$, then being of the third class means 
\begin{align}\begin{split}\label{of_the_third_class}
cy_1^-x_2^+ > 0.    
\end{split}\end{align}

Since the points $O_1$ and $O_2$ are close to points in $W^s(O)$ (namely the points $q_1$ and $q_2$), and can be arranged to be arbitrarily close by choosing large enough $n$ in the construction of our horseshoe, we note that we may arrange for the corresponding position on the boundary of the cylinder for these two points to be different - that is, these two points are not on the same vertical line.  We also can arrange for these to not be on the same vertical line as $O$.  This is because we set our $W^s_{loc}(O)$ to have a non-zero angle $\omega$, so that, close to $O$, each point in $W^s_{loc}(O)$ lies in a different vertical line.

Then, to complete the second item in our outline, we use Perturbation \ref{perturbation_which_varys_the_first_differential} in Appendix \ref{appendix_all_perturbations} to vary the eigenvalues of the differential at $O$.  Doing this, we may make a new quadratic heteroclinic tangency between $W^u(O_1)$ and $W^s(O_2)$ while retaining all the other homoclinic and heteroclinic tangencies.  This follows from Lemma 4 of \cite{GST}.  Before citing this we must define a moduli of local $\Omega-$conjugacy that is particular to cycles of the third class:
\begin{align}\begin{split}\nonumber
    \alpha = \frac{\ln(\lambda_1)}{\ln(\lambda_2)},
\end{split}\end{align}
where $\lambda_i$ is the multiplier greater than 1 for the manifolds at $O_i$.  Then, Lemma 4 of \cite{GST} states:

\begin{sublemma}
(Lemma 4 of \cite{GST}) Let $f_\eps$ be any smooth family of $C^\infty$ billiard maps such that all maps in the family have a heteroclinic cycle of the third class, i.e. there are two periodic points $O_1$ and $O_2$, an orbit of transverse intersection of $W^u(O_1)$ and $W^s(O_2)$, and an orbit of quadratic heteroclinic tangency between $W^u(O_2)$ and $W^s(O_1)$ which does not unfold as $\eps$ varies, plus Condition \eqref{of_the_third_class} holds.  If $\alpha$ changes monotonically with $\eps$, i.e. 
\begin{align}\begin{split}\nonumber
    \frac{\del \alpha (f_\eps)}{\del \eps} \neq 0,
\end{split}\end{align}
then there is a dense set of values of $\eps$ for which the map $f_\eps$ has a quadratic heteroclinic tangency (which unfolds generically as $\eps$ varies) between $W^u(O_1)$ and $W^s(O_2)$.
\end{sublemma}

Since this orbit must contain points close to $W^s_{loc}(O_1)$, we may arrange for this to also satisfy the injectivity condition with respect to all the other orbits, since $W^s_{loc}(O_1)$ must be parallel to $W^s_{loc}(O)$.

Now, since our points $O_1$ and $O_2$ are close to $W^s(O)$, we may by a small perturbation arrange for the newly constructed quadratic heteroclinic tangency between $W^u(O_1)$ and $W^s(O_2)$ to go between $W^u(O)$ and $W^s(O)$, so that we have a new quadratic homoclinic tangency.  We can accomplish this using an argument similar to that done in the previous section, by varying the angles of the manifolds of $O$.

This completes the second item in our outline.  Before going to the third item, we repeat this process $2n+4$ times in order to obtain a map with $2n+4$ coexisting, distinct points of quadratic HT belonging to different homoclinic orbits (the reason for doing it this number of times will be made clear in a moment).  For each time, we ensure that the injectivity condition continues to hold for each homoclinic orbit, which we can do since each newly constructed quadratic homoclinic orbit contains points which are in $W^s_{loc}(O_1)$, and since this is not a vertical line, we may take the new points of homoclinic tangency to all not lie on the same vertical line.  

Now, to prove the third item, we seek to use Lemma 5 of \cite{GST}.  We state this lemma here (again slightly reworded for our case).

\begin{sublemma}\label{lemma_making_n+1_tangency}
(Lemma 5 of \cite{GST}) Let $f_\eps$ (where $\eps=(\tau_0,...,\tau_{k-1},\nu)$) be a smooth ($k+1$)-parameter family of $C^\infty$ billiard maps which have a saddle periodic point $O$ such that at $\tau=0$ the manifolds $W^u(O)$ and $W^s(O)$ have a tangency of order $k$, and at $\nu=0$ these manifolds have a quadratic tangency.  Suppose that the tangency of order $k$ unfolds generically as $\tau$ varies, and the quadratic tangency unfolds generically as $\nu$ varies.  Then there exists a sequence $\eps_j \rightarrow 0$ such that the map $f_\eps$ has an orbit of tangency of order $(k+1)$ between $W^u(O)$ and $W^s(O)$ at $\eps=\eps_j$, while keeping the saddle point $O$ fixed.
\end{sublemma}

\begin{figure}[h]
\centering
\includegraphics[width=11cm, height=6.25cm]{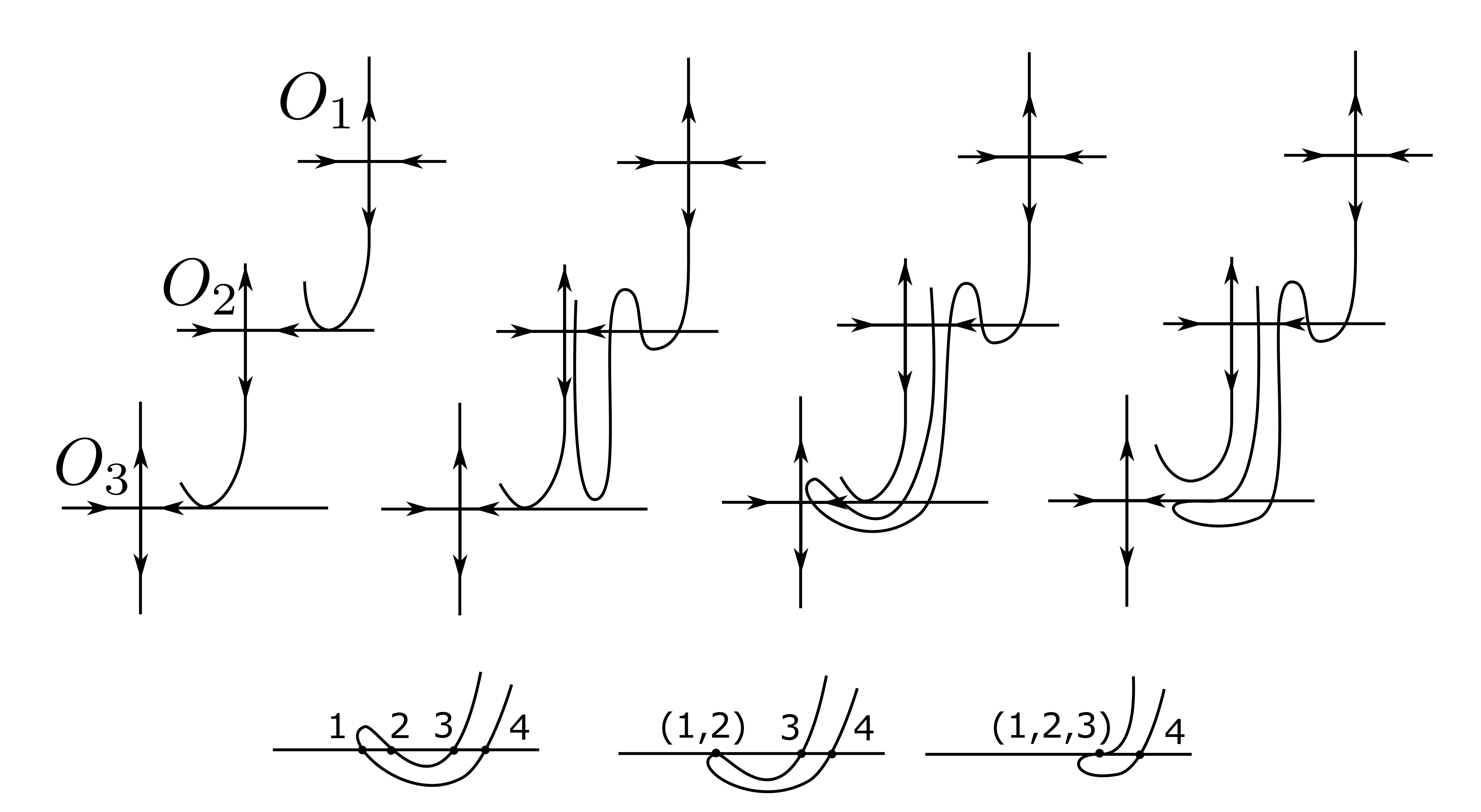}
\caption{An example of the process of creating a cubic HT out of a tower of two quadratic HT.  Essentially, we have a situation where there is an intersection between $W^u(O_1)$ and $W^s(O_3)$, which intersects at four points as in the diagram on the bottom left.  Then, we unfold generically to pull this back so that points $1$ and $2$ hit at the same point, creating a quadratic HT between $W^u(O_1)$ and $W^s(O_3)$.  From here, we unfold the two original quadratic HT again to merge points $1$ and $2$ into point $3$, thus creating a cubic HT. }
\end{figure}

So by Lemma \eqref{lemma_making_n+1_tangency} we take two such quadratic HT and make a cubic HT by using our perturbations from Lemma \eqref{lemma_unfold_generically}, applying Lemma \eqref{lemma_on_injectivity} on our newly created cubic to ensure the injectivity condition between this orbit and the remaining quadratic HT.  We then take this map, which contains now a cubic HT and one of our previously constructed quadratic HT, and apply Lemma \eqref{lemma_making_n+1_tangency} to create a quartic HT - and so on so that we obtain a map containing a HT of arbitrarily high order.

This complete our third item.

Now, since at each step we may arrange for the perturbations to be arbitrarily small in the $C^\infty$ topology, for any $\delta > 0$ we may arrange for our map final map $\bar f$ which contains  HT of order $r$ to be within $\delta$ of our original map $f$ in the $C^\infty$ topology.

We now do item 4 in our outline for this section to obtain an elliptic orbit which is a rotation up to order $n$.  From \cite{GST} we have

\begin{sublemma}\label{lemma_creating_an_ellpitic_periodic_point}
    (Lemma 8 of \cite{GST}). Let $f_\tau$ be a family of $C^\infty$ billiard maps for which a homoclinic tangency
    of order $2n+4$ unfolds generically. Then arbitrarily close to the moment of tangency there
    are values of parameters that correspond to the existence of an elliptic periodic orbit of the
    degeneracy order $n$.
\end{sublemma}

Thus, unfolding the tangency of order $2n+4$ generically, we obtain an elliptic periodic orbit of order $n$.

Now, when one considers the first return map in a neighborhood of this periodic point one obtains the following Birkhoff normal form:

\begin{align}\begin{split}\nonumber
    \overline{z} = e^{i\phi}z + o(|z|^n).
\end{split}\end{align}

So, our map in the $C^n$ topology is close to a rotation by $\phi$.

\end{proof}
Now observe that if this is a rational rotation, then we are in fact done with the proof of Theorem 1.  However, this may be an irrational number.  Thus, in the next section, we describe a perturbation of the boundary so that this becomes a rational number, thus completing the proof when we consider the billiard map composed with itself a number of times equal to the denominator of this rational number.

\section{Rotating the differential to order $n$}\label{sec_rotate_the_differential}

Here we prove that one may rotate the differential to the $n^{th}$ order by perturbing the curvature up to order $n$ at $n+3$ points.  The size of these perturbations are of order $\delta$ in all partial derivatives from degree 0 to $n$, where $\delta$ is the angle by which we rotate the differential.  Our main result in this section is Lemma \eqref{lemma_rotating_the_differential_to_order_n}.
\begin{proof}
We begin as in the statement of the lemma.  Then we expand our map in the form
\begin{align}\begin{split}\nonumber
    f^q(s_0+u,\phi_0+v) =  df^q(s_0,\phi_0)\begin{pmatrix}
    u \\ v
    \end{pmatrix} +
    \sum_{k=2}^{n+1}
    \Delta^{k}_{0}(u,v),
\end{split}\end{align}
where 
\begin{align}\begin{split}\nonumber
    \Delta_0^k(u,v) = \sum_{j=0}^k \delta_{0,j}^k u^{k-j} v^j,\text{ for }2 \leq k \leq n
\end{split}\end{align}
where $\delta_{0,j}^k$ are constants, and we also have that $\Delta^{n+1}_{0}(u,v)$ satisfies
\begin{align}\begin{split}\nonumber
    \frac{\del^j}{\del u^{j-i} \del v^i}\bigg(\Delta^{n+1}_{0}(u,v)\bigg)\Bigg|_{(u,v)=(0,0)} = 0\text{ for }0 \leq i \leq j \leq n.
\end{split}\end{align}  
Now, by Perturbation \ref{perturbation_which _lets_you_vary_all_derivatives_independently} for $n=0$, we may perturb to first order by size $O(\delta)$ so that, to first order, the leading term is rotated by a matrix $R_{\delta}$, which is a rotation matrix\footnote{At this step, one might wonder how we know we can necessarily achieve this specific differential while ensuring we do not leave the class of billiard maps.  This will be answered directly after this proof.} by the angle $\delta$.  In other words, we may find a map $f_1$ which is close (dependent on $\delta$) to $f$ so that  
\begin{align}\begin{split}\nonumber
    f^q_1(s_0+u,\phi_0+v) =  R_\delta df^q(s_0,\phi_0)\begin{pmatrix}
    u \\ v
    \end{pmatrix} +
    \sum_{k=2}^{n+1}
    \Delta^{k}_{1}(u,v),
\end{split}\end{align}
where 
\begin{align}\begin{split}\nonumber
    \Delta_1^k(u,v) = \sum_{j=0}^k \delta_{1,j}^k u^{k-j} v^j,
\end{split}\end{align}
with
\begin{align}\begin{split}\nonumber
    \frac{\del^{j}}{\del u^{j-i} \del v^i}\bigg( \Delta_1^{n} (u,v)\bigg)\Bigg|_{(u,v)=(0,0)}=0\text{ for }0 \leq i \leq j \leq n,
\end{split}\end{align}
and for $j \leq k$ and $0 \leq j \leq k$,
\begin{align}\begin{split}\nonumber
    |\delta_{1,j}^k - \delta_{0,j}^k| < C_{j,k,n}\delta
\end{split}\end{align}
for some constants $C_{j,k,n}$.  From here we again use Perturbation \ref{perturbation_which _lets_you_vary_all_derivatives_independently}, now with $n=1$ in order to eliminate the change that was introduced in the second order partial derivatives after the previous perturbation. Specifically, what we mean is we perturb so that we obtain a map $f_2$ such that
\begin{align}\begin{split}\nonumber
    f^q_2(s_0+u,\phi_0+v) &=  R_{\delta}df^q(s_0,\phi_0)\begin{pmatrix}
    u \\ v
    \end{pmatrix} + \Delta_0^2(u,v) + \Delta^{3}_{2}(u,v) + ... 
    \\
    &+ \Delta^{n+1}_{2}(u,v),
\end{split}\end{align}    
with
\begin{align}\begin{split}\nonumber
    \Delta_2^k(u,v) &= \sum_{j=0}^k \delta_{2,j}^k u^{k-j} v^j,
    \\
    |\delta_{2,j}^k-\delta_{0,j}^k| &\leq C_{j,k,n}\delta,
\end{split}\end{align}
for $k \geq 3$, and 
\begin{align}\begin{split}\nonumber
    \frac{\del^{j}}{\del u^{j-i} \del v^i}\bigg( \Delta_2^{n+1} (u,v)\bigg)\Bigg|_{(u,v)=(0,0)}=0\text{ for }0\leq i \leq j \leq n.
\end{split}\end{align}
We are able to do this because $|\delta_{1,j}^2-\delta_{0,j}^2|\leq C_{j,2,n}\delta$ for $0\leq j \leq 2$, which means the size of our perturbation can be of order $\delta$.  We then do the same process inductively. At each step we have a map $f_k$ of the form
\begin{align}\begin{split}\nonumber
    f_k^q(s_0+u,\phi_0+v) = R_{\delta}df^q(s_0,\phi_0)\begin{pmatrix}
    u \\ v
    \end{pmatrix} + \sum_{j=0}^{k} \Delta_0^j(u,v) + \sum_{j=k+1}^{n} \Delta^{j}_{k}(u,v),
\end{split}\end{align}
where
\begin{align}\begin{split}\nonumber
    \Delta_k^j(u,v) &= \sum_{i=0}^j \delta_{k,i}^j u^{j-i}v^i,
    \\
    |\delta_{k,i}^j-\delta_{0,i}^j| &\leq C_{i,j,n} \delta
\end{split}\end{align}
for $k+1 \leq j \leq n $ and $0 \leq i \leq j$, and with
\begin{align}\begin{split}\nonumber
    \frac{\del^{j}}{\del u^{j-i} \del v^i}\bigg( \Delta_k^{n} (u,v)\bigg)\Bigg|_{(u,v)=(0,0)}=0\text{ for }0\leq i \leq j \leq n,
\end{split}\end{align}  
and we perturb using Perturbation \ref{perturbation_which _lets_you_vary_all_derivatives_independently} with $n=k-1$ in order to obtain a map with the same properties, but with $k$ increased by 1.  In the end we arrive at a map $f_{n+1}$ which is close to $f$ in all derivatives up to degree $n$, dependent on $\delta$, so that
\begin{align}\begin{split}\nonumber
    f^q_{n+1}(s_0 + u,\phi_0 + v) =  R_{\delta}df^q(s_0,\phi_0)\begin{pmatrix}
    u \\ v
    \end{pmatrix} + \sum_{k=2}^{n}\Delta_0^k(u,v) + \Delta_{n+1}^{n+1}(u,v),
\end{split}\end{align}
with
\begin{align}\begin{split}\nonumber
    \frac{\del^{j}}{\del u^{j-i} \del v^i}\bigg( \Delta_{n+1}^{n+1} (u,v)\bigg)\Bigg|_{(u,v)=(0,0)}=0\text{ for }0\leq i \leq j \leq n.
\end{split}\end{align}
This completes the proof.
\end{proof}
Now, we answer an objection that one might have had in the previous proof.  It might not be clear at first glance how we know we can achieve the specific form of the differential we desire (i.e. the previous matrix but rotated by some angle).  More precisely, consider the set of 2 by 2 matrices:
\begin{align}\begin{split}\nonumber
    \{A \in \R ^{2 \times 2}:\text{ there exist a billiard map }f\text{ such that }df(x_0) = A\},
\end{split}\end{align}
and for higher orders
\begin{align}\begin{split}\nonumber
    \{A\in \R^{(n+1) \times 2}&:\text{ there exists a billiard map }f\text{ such that }
    \\
    &A_{i,1} = \frac{\del^n Proj_1(f)}{\del_s^{n-i}\del\phi^i}(s_0,\phi_0), A_{i,2} = \frac{\del^n Proj_2(f)}{\del_s^{n-i}\del\phi^i}(s_0,\phi_0),
    \\
    &\text{for }0 \leq i \leq n
    \},
\end{split}\end{align}
where $Proj_i(f)$ is the projection of $f$ onto its $i^{th}$ coordinate.  Then the question is "how do we know our target values are in these sets?"  For instance, how do we know that given a billiard map $f$ with differential $df(x_0)$ we may find another billiard map $f_1$ so that $df_1(x_0)=R_\delta df(x_0)$?  

The answer is that for the first order terms, we can change three terms to perfectly match what one would get from a rotation, and the fourth term necessarily is fixed due to the area preservation condition.

For the higher order terms, we know that all the target values we perturb to are a priori belonging to those of a billiard map, and are thus obtainable through our perturbations. 

\appendix 

\section{Perturbations and their Effects on the Jet of the Billiard Map}\label{appendix_all_perturbations}
Throughout this section, we let $\mathcal O = ((s_0,\phi_0),...,(s_{q-1},\phi_{q-1}))$ be a $q$-periodic obit.  We also let $x_0=(s_0,\phi_0)$ and $W_{loc}^u(x_0)$ be its local invariant unstable manifold, and similarly $W_{loc}^s(x_0)$ be its lcoal invariant stable manifold (all under the mapping $f$).  We use the following notation:
\begin{align}\begin{split}\nonumber
    x_k &= (s_k,\phi_k),
    \\
    l_k &= l(s_k,s_{k+1}) = ||\gamma(s_{k+1})-\gamma(s_k)||,
    \\
    \beta_k &= \sin(\phi_k),
    \\
    \kappa_k &= \kappa(s_k).
\end{split}\end{align}

The format of this section is to provide a description of a perturbation, along with the result it has on the partial derivatives of $f^q(x_0)$, followed by a proof which shows how the perturbation leads to that change in the derivatives.  
The first type of perturbation we do is the following:
\begin{perturbation}
    We change $\kappa$ in a small neighborhood around $s_k$ for some fixed $k$, and we choose this perturbation in such a way as to leave the tangent vector at $\gamma(s_k)$ unchanged.  We change $\kappa$ in this neighborhood so that in particular at the point $s_k$, the curvature $\kappa$ changes as:
    \begin{align}\begin{split}\nonumber
        \kappa_k \rightarrow \kappa_k + \eps.
    \end{split}\end{align}
    Given this perturbation, the change in the differential $df^q(x_0)$ is given by
    \begin{align}\begin{split}\nonumber
        df^q(x_0) \rightarrow  &df_\eps^q(x_0) =
        \\
       &df^q(x_0) \ + \eps df^{q-k}(x_k)\,B\,df^k(x_0),
    \end{split}\end{align}
    where
    \begin{align}\begin{split}\nonumber
        B &= \begin{bmatrix}
        0 && 0
        \\
        2 && 0
        \end{bmatrix}.
    \end{split}\end{align}
\end{perturbation}

\begin{figure}[h]
\centering
\includegraphics[width=8.5cm, height=6.25cm]{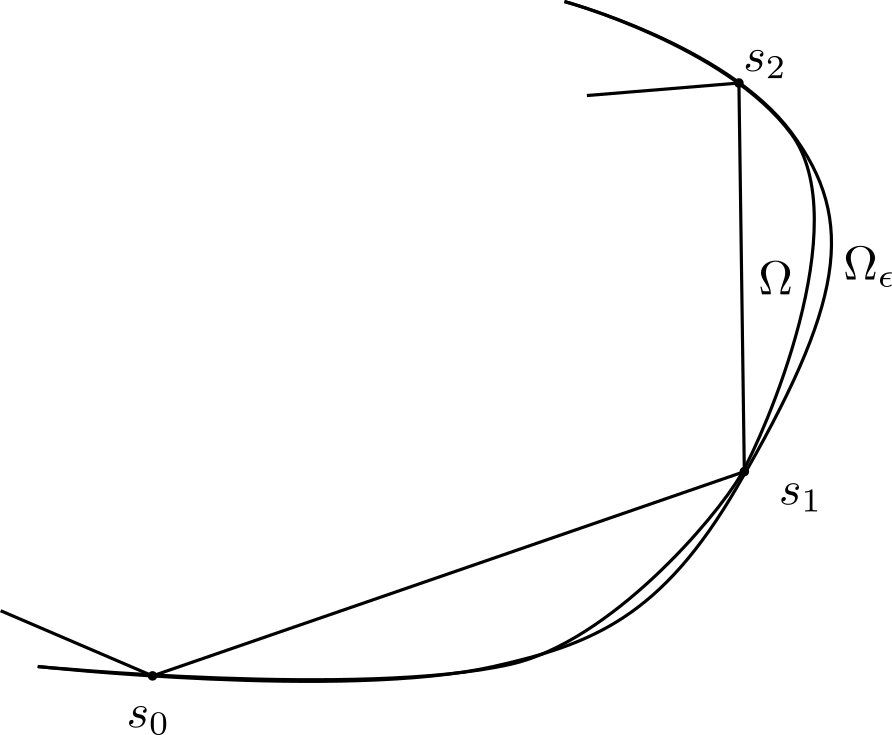}
\caption{The types of perturbations we consider. We change the curvature at $s_1$, while leaving the angle tangent of the boundary of the domain at $s_1$ fixed, thus leaving the orbit fixed while changing the differential of the billiard map.}
\end{figure}

\begin{proof}
\text{For positive integers $i,j$ we} first note that
\begin{align}\begin{split}\nonumber
    df^i(x_j)=\begin{bmatrix}
    \frac{\del s_{i+j}}{\partial s_j} & \frac{\del s_{i+j}}{\partial \phi_j} \\ \frac{\del \phi_{i+j}}{\partial s_j} & \frac{\del \phi_{i+j}}{\partial \phi_j}
    \end{bmatrix},
\end{split}\end{align}
and for $i>k>j$ that we also have the relation $df^i(x_j)=df^{i-k}(x_k)df^{k-j}(x_j)$.  Then, it follows that we have
\begin{align}\begin{split}\nonumber
    df^q(x_0)=df^{q-(k+1)}(x_{k+1})df(x_k)df(x_{k-1})df^{k-1}(x_0).
\end{split}\end{align}

Note that $df^{q-(k+2)}(x_{k+2})$ and $df^{k-1}(x_0)$ do not depend on $\eps$, while the others do:
\begin{align}\begin{split}\nonumber
    df(x_{k})&=df_\eps(x_{k})\\
    df(x_{k-1})&=df_\eps(x_{k-1}).
\end{split}\end{align}

We now recall the following equations (\cite{Katok} Theorem 4.2 in Part V):
\begin{align}\begin{split} \nonumber
    & \frac{\partial s_{k+1}}{\partial s_k} = \frac{\kappa_kl_k-\beta_k}{\beta_{k+1}}\\ \label{derivatives_2}
    & \frac{\partial s_{k+1}}{\partial \phi_k} = \frac{l_k}{\beta_{k+1}}\\ 
    & \frac{\partial \phi_{k+1}}{\partial s_k} = \frac{\kappa_k\kappa_{k+1}l_k-\kappa_k\beta_{k+1}-\kappa_{k+1}\beta_k}{\beta_{k+1}}\\ 
    & \frac{\partial \phi_{k+1}}{\partial \phi_k} = \frac{\kappa_{k+1}l_k-\beta_{k+1}}{\beta_{k+1}}.
\end{split}\end{align}
Using this, it is straightforward to show how the perturbations act on each of our differentials to first order. We have:
\begin{align}\begin{split}\nonumber
    df(x_k)\rightarrow &\ df_\eps(x_k)=df(x_k) + \eps\begin{bmatrix}
    \frac{l_k}{\beta_{k+1}} & 0\\
    \frac{\kappa_{k+1}l_k-\beta_{k+1}}{\beta_{k+1}}&0
    \end{bmatrix}\\
    = &\ df(x_k) + 
    \eps\begin{bmatrix}
    \frac{\del s^{k+1}}{\del\phi_k} & 0\\
    \frac{\del \phi^{k+1}}{\del\phi_k}&0
    \end{bmatrix}\\
    = &\ df(x_k) + df(x_k)
    \eps\begin{bmatrix}
    0 & 0\\
    1 & 0
    \end{bmatrix},
\end{split}\end{align}
and 
\begin{align}\begin{split}\nonumber
    df(x_{k-1})\rightarrow &\ df_\eps(x_{k-1})=df(x_{k-1}) + \eps\begin{bmatrix}
    0 & 0\\
    \frac{\kappa_{k-1}l_{k-1}-\beta_{k-1}}{\beta_{k}}&\frac{l_{k-1}}{\beta_{k}}
    \end{bmatrix}\\
    = &\ df(x_{k-1}) + \eps
    \begin{bmatrix}
    0 & 0\\
    \frac{\del s^k}{\del s_{k-1}}&\frac{\del s^k}{\del\phi_{k-1}}
    \end{bmatrix}\\
    = &\ df(x_{k-1}) + \eps
    \begin{bmatrix}
    0 & 0\\
    1 & 0
    \end{bmatrix}df(x_{k-1}).
\end{split}\end{align}
So, in total, we have that $\eps$ affects our differential by
\begin{align}\begin{split} \nonumber
    df(x_k)df(x_{k-1})\rightarrow df(x_k)df(x_{k-1}) + \eps \, \bigg(df(x_k)
    \begin{bmatrix}
    0 & 0\\
    2 & 0
    \end{bmatrix} df(x_{k-1})\bigg).
\end{split}\end{align}

\end{proof}

The next perturbation is used in the proofs of Lemmas \eqref{lemma_unfold_generically} and \eqref{lemma_on_injectivity}, and it enables us to smoothly lift and lower a point of homoclinic tangency.

\begin{perturbation}\label{perturbation_which_lets_you_move_unstable_manifold}
    Given a periodic point $O$ with a point of homoclinic tangency at $p$ and unit speed parameterization $p(t)$ (for $|t|$ small) of a section of $W^u(O)$ with $p(0)=p$, there is a smooth deformation obtained by using Perturbation  \ref{perturbation_which_varys_the_first_differential} which allows us to change the angle of the local unstable manifold $W_{loc}^u(O)$ by $C\eps + O(\eps^2)$ for some constant $C$, which deforms $p(t)$ to $p_\eps(t)$ such that for all $|t|$ sufficiently small the distance $\Phi_\eps(t)$ between the point $p_\eps(t)$ and the local stable manifold $W_{loc}^s(O)$ is given by
    \begin{align}\begin{split}\nonumber
        \Phi_\eps(t) = \Phi(t) + \eps + O(\eps^2),
    \end{split}\end{align}
    where $\Phi(t)$ is the unperturbed distance of $p(t)$ to the local stable manifold $W_{loc}^s(O)$.
\end{perturbation}

\begin{figure}[h]
\centering
\includegraphics[width=8.5cm, height=6cm]{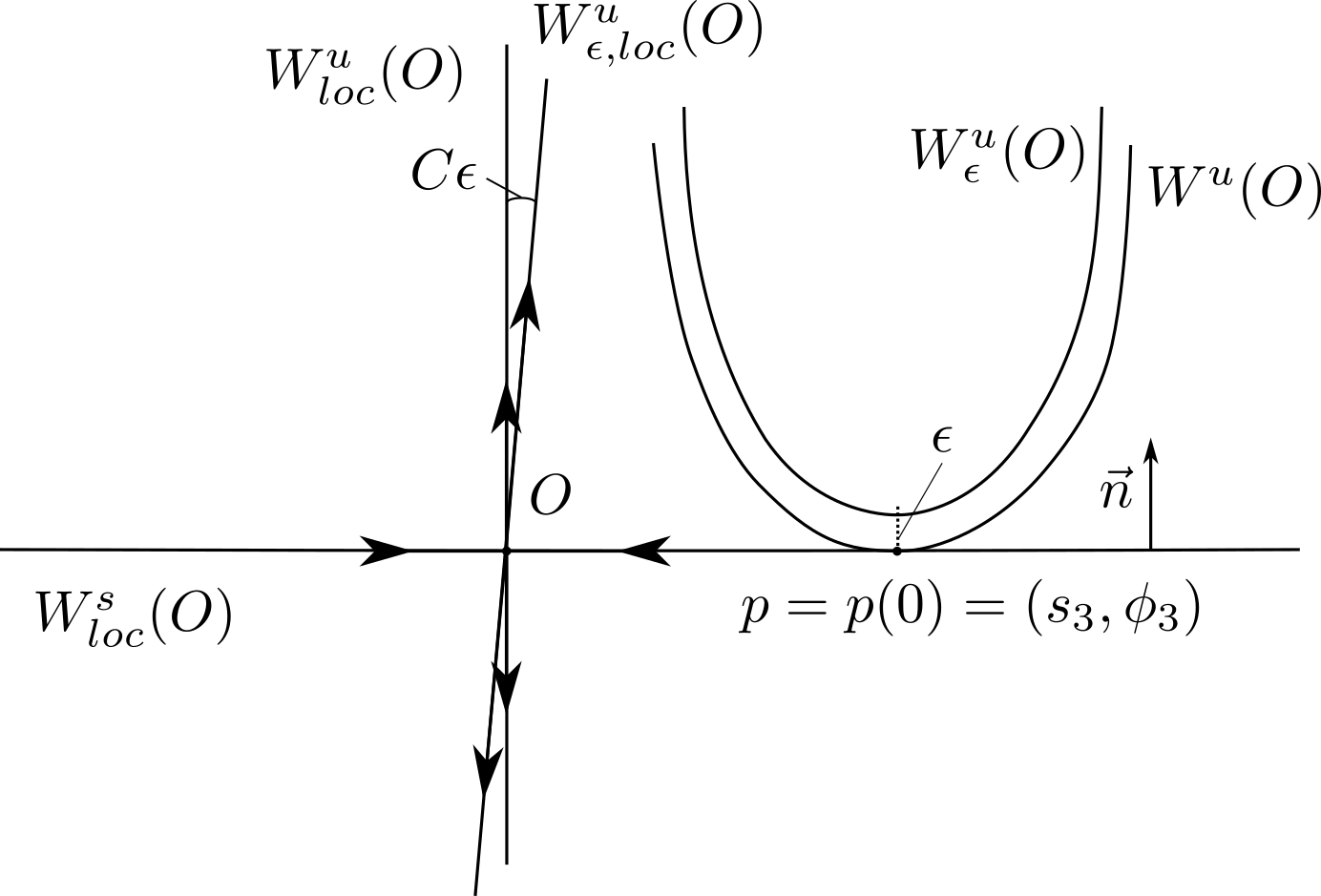}
\caption{By changing the angle of the unstable manifold, we may move the unstable manifold around a point of HT in the direction of the normal of the stable manifold at the moment of tangency.}
\end{figure}

\begin{proof}
    We let $U$ be a small strip of width $\sigma$ around the point $O$, and we perturb in $U$.  This corresponds to perturbing $\kappa$ on the region $(s_0-\sigma/2,s_0+\sigma/2)$.  We perturb by Perturbation  \ref{perturbation_which_varys_the_first_differential} which allows us to change the angle of the unstable manifold $W^u_{loc}$ by angle $\omega$.  We call the resulting billiard map $f_\omega$ and the resulting unstable manifold $W^u_{loc,\omega}(O)$.   We then define the points $a_\omega^-,a_\omega^+$ as the points within $W_{loc,\omega}^u(O) \cap \del U = \{a_\omega^-,a_\omega^+\}$, with $a_\omega^- < a_\omega^+$. Then we have
    
    \begin{align}\begin{split}\nonumber
        a_\omega^\pm - a_0^\pm &= \pm\frac{\sigma}{2}\bigg(\tan(\omega_0 + \omega)-\tan(\omega_0)\bigg) + O(\omega^2)
        \\
        &=\pm\frac{\sigma}{2}\omega + O(\omega^2).
    \end{split}\end{align}
    Consider the point $f^{-m}(p)$ where $m$ is such that $f^{-m}(p) \notin U$ and $f^{-(m+1)}(p) \in U$.  Consider a small ball $V$ around $f^{-m}(p)$, chosen small enough that $V\cap U = \emptyset$. 
    
    Now $W_{loc}^u(O)$ intersects $V$, which implies that for $\omega$ small enough, $W_{loc,\omega}^u(O)$ intersects $V$ as well.  We label this as $l$, and parameterize it by $l(t)$ for $t$ small such that $l(t)$ is a unit speed parameterization. We observe that for each point in this small section we have
    \begin{align}\begin{split}\nonumber
        d(l(t),W^u(O)) = \frac{\sigma \omega}{2} + O(\omega^2),
    \end{split}\end{align}
    where $d(l(t),W_{loc}^u(O))$ is the shortest distance of $l(t)$ to $W_{loc}^u(O)$.  Now under the dynamics of $f^q$ (i.e. the unperturbed map), by the lambda lemma (see \cite{siburg}) we have that $f^{nq}(l)\rightarrow W_{loc}^u(O)$ as $n \rightarrow \infty$.  In particular we have
    \begin{align}\begin{split}\nonumber
        d(f^{mq}(l(t)),W_{loc}^u(O)) = \frac{\sigma \omega}{2 \lambda^{mq}} + O(\omega^2\lambda^{-mq}).
    \end{split}\end{align}
    Now the dynamics outside of the region $U$ where we perturb remain unchanged as we vary $\omega$, and so since the dynamics which send $V$ to $T^m(V)$ are outside of the region $U$, we have
    \begin{align}\begin{split}\nonumber
        d(W_{loc,\eps}^u(O),W_{loc}^u(O)) = \frac{\sigma \omega}{2 \lambda^{mq}} + O(\omega^2\lambda^{-mq}).
    \end{split}\end{align}
    Thus if we let
    \begin{align}\begin{split}\nonumber
        \omega = \frac{2 \eps \lambda^{mq}}{\sigma}
    \end{split}\end{align}
    we have
    \begin{align}\begin{split}\nonumber
        d(W_{loc,\eps}^u(O),W_{loc}^u(O)) = \eps + O(\eps^2),
    \end{split}\end{align}
    which is the desired result.
\end{proof}

For our next perturbation, we label the $k^{th}$ derivative of the curvature at each point $s_i$ by $\kappa_i^{(k)}$.
\begin{perturbation}[Theorem \ref{theorem_perturb_many_independently}]\label{perturbation_which _lets_you_vary_all_derivatives_independently}
    Consider a perturbation which satisfies
    \begin{align}\begin{split}\nonumber
        \Delta\kappa_i^{(j)} &= 0 \text{ for }0 \leq j \leq n-1,
        \\
        \Delta\kappa_i^{(n)} &= \eps_i
    \end{split}\end{align}
    for $1 \leq i \leq n+2$ and that the orbit $\mathcal O$ remains fixed.  Then one may vary each element of 
\begin{align}\begin{split}\nonumber
\left\{\frac{\del^n s_{n+3}}{\del s_0^{n-k}\del \phi_0^k}:0 \leq k \leq n\right\}
\large\cup
\left\{\frac{\del^n \phi_{n+3}}{\del s_0^n}\right\}
\end{split}\end{align}
independently, provided each $(s_i,\phi_i)$ is sufficiently close to the boundary for $0 \leq i \leq n+3$.  More precisely, letting $\eps = (\eps_1,\cdots,\eps_{n+2})$,
    \begin{align}\begin{split}\nonumber
        \frac{\del\bigg(
        \frac{\del^n s_{n+3}}{\del s_0^{n}}, \frac{\del^n s_{n+3}}{\del s_0^{n-1}\del \phi_0},...,\frac{\del^n s_{n+3}}{\del \phi_0^n},\frac{\del^n \phi_{n+3}}{\del s_0^n}
        \bigg)(\eps)}{\del \eps} \neq 0,
    \end{split}\end{align}
    provided
    \begin{align}\begin{split}
        \frac{\del s_j}{\del \phi_i} \neq 0
    \end{split}\end{align}
    for $0 \leq i < j \leq n+3$.
\end{perturbation}
\begin{proof}
As the proof is quite lengthy and technical, we relegate it to Appendix \ref{appendix_perturb_nth_derivative_independently}.
\end{proof}

We also note the special case of this perturbation for $n=0$:
\begin{perturbation}\label{perturbation_which_varys_the_first_differential}
    Given a perturbation which changes the curvature in small neighborhoods around 3 points $(s_1,s_2,s_3)$ by $\Delta \kappa_i = \eps_i,$ and which leaves the orbit $\mathcal O$ fixed, we have
    \begin{align}\begin{split}\nonumber
        \frac{\del\bigg(
        \frac{\del s_{4}}{\del s_0}, \frac{\del s_{4}}{\del \phi_0},\frac{\del \phi_{4}}{\del s_0}
        \bigg)(\eps)}{\del \eps} \neq 0,
    \end{split}\end{align}
    provided
    \begin{align}\begin{split}\nonumber
        \frac{\del s_j}{\del \phi_i} \neq 0
    \end{split}\end{align}
    for $0 \leq i < j \leq 4$.  Thus, with this perturbation one may vary the angles and eigenvalues of the differential $df^4((s_0,\phi_0))$ independently to first order.
\end{perturbation}
\begin{proof}
    As the first part of the statement follows from the work in Appendix \ref{appendix_perturb_nth_derivative_independently}, we prove the second part of the statement.  It is a simple exercise in linear algebra, but we detail it here for completeness.
    
We consider a matrix $A=\begin{bmatrix}
a&b\\c&d
\end{bmatrix}$ with product of eigenvalues $1$ and eigenvalues and eigenvectors given by:
\begin{align}\begin{split}\nonumber
    \lambda & = \frac{a+d}{2}+\frac{1}{2}\sqrt{(a+d)^2-4\det(A)},\\
    \lambda^{-1} & = \frac{a+d}{2}-\frac{1}{2}\sqrt{(a+d)^2-4\det(A)},\\
    \vec{V}_\lambda & = \begin{bmatrix} b\\\lambda-a \end{bmatrix}\frac{1}{\sqrt{b^2+(\lambda-a)^2}}=\begin{bmatrix} \cos(\omega_1)\\\sin(\omega_1)
    \end{bmatrix}\\
    \vec{V}_{\lambda^{-1}} & = \begin{bmatrix} b\\\lambda^{-1}-a \end{bmatrix}\frac{1}{\sqrt{b^2+(\lambda^{-1}-a)^2}}=\begin{bmatrix} \cos(\omega_2)\\\sin(\omega_2)
    \end{bmatrix}.
\end{split}\end{align}
Thus if we perturb our matrix $A$ by
\begin{align}\begin{split}\nonumber
\begin{bmatrix}
a & b \\ c & d
\end{bmatrix} \rightarrow
\begin{bmatrix}
a+\eps_1 & b+\eps_2 \\ c+* & d+\eps_3
\end{bmatrix}
\end{split}\end{align}
where $*$ is chosen so that the overall perturbation maintains that the product of the eigenvalues is still 1, we can see how this effects $\lambda, \omega_1,$ and $\omega_2$, to first order.  Doing this yields the following:
\begin{align}\begin{split}\nonumber
    \Delta\omega_1 = & \eps_1\frac{b}{l_+^2(\lambda^2-1)} - \eps_2\frac{(\lambda-a)}{l_+^2} + \eps_3\frac{b\lambda^2}{l_+^2(\lambda^2-1)} + O(\eps^2)\\
    \Delta\omega_2 = & -\eps_1\frac{(2\lambda^2+1)b}{l_-^2(\lambda^2-1)} - \eps_2\frac{\lambda^{-1}-a}{(\lambda^{-1})^2} - \eps_3\frac{b\lambda^2}{l_-^2(\lambda^2-1)} + O(\eps^2)\\
    \Delta\lambda = & \eps_1\frac{\lambda^2}{\lambda^2-1} + \eps_3\frac{\lambda^2}{\lambda^2-1} + O(\eps^2),
\end{split}\end{align}
where $l_+^2 = b^2 + (a-\lambda)^2$, and $l_-^2 = b^2 + (a - \lambda^{-1})^2$.  From here then it is easy to show that we may vary $\omega_1, \omega_2, \lambda$ each independently for appropriately chosen $\eps$.  Applying this result to our differential, we are done.
    
\end{proof}

We now note the next type of perturbation we use.  Essentially, this type of perturbation states that if we have a perturbation which varies the curvature at $k$ points, we may instead consider a different perturbation at any 3 previous points that achieves the same change of the differential.

\begin{perturbation}\label{lemma_we_can_perturb_around_three_points_instead}
Given a perturbation of the curvature at points $s_{N-k},...,s_{N-1}$ which varies the differential of $f$ by
\begin{align}\begin{split}\nonumber
    df^N(x_1) \rightarrow df^N(x_1) + \Delta_1df^N(x_1)
\end{split}\end{align}
while keeping the orbit $\mathcal O$ fixed, there exists another perturbation of the curvature at any increasing subsequence of length $3$, $s_{n_1},s_{n_2},s_{n_3}$ such that this perturbation also keeps $\mathcal O$ fixed and changes the differential of $f$ by
\begin{align}\begin{split}\nonumber
    df^N(x_1) \rightarrow df^N(x_1) + \Delta_2df^N(x_1),
\end{split}\end{align}
so that
\begin{align}\begin{split}\nonumber
    \Delta_1df^N(x_1) = \Delta_2df^N(x_1),
\end{split}\end{align}
given the following holds:
\begin{align}\begin{split}\nonumber
    \frac{\del s_{j}}{\del \phi_i}\neq 0
\end{split}\end{align}
for $1 \leq i < j \leq N+1$.
\end{perturbation}
\begin{proof}
This follows from showing that one can vary the curvature at any $4$ points, $s_{k_1}<s_{k_2}<s_{k_3}<s_{k_4}$ so that the the resulting $\Delta df^N(x_1)=0$.  This is because of the following argument.  Supposing that we can, then we have

\begin{align}\begin{split}\nonumber
    \Delta \, df^N(x_1) =& \sum_{i=0}^{k-1} df^{k-i}(x_{N-k+i}) \, B \, df^{N-k+i-1}(x_1)\Delta\kappa_{N-k+i}
\end{split}\end{align}

And then if we consider each term with the three points $s_{n_1},s_{n_2},s_{n_3}$, we can find a perturbation $\Phi_i$ which changes the curvature around these points by $\Delta\kappa_{s_1},\Delta\kappa_{s_2},\Delta\kappa_{s_3}$ so that
\begin{align}\begin{split}\nonumber
    & df^{k-i}(x_{N-k+i})\,B\,df^{N-k+i-1}(x_1)\Delta_i\kappa_{N-k+i}\\
    +&df^{N-s_1}(x_{s_1})\,B\,df^{s_1-1}(x_1)\Delta_i\kappa_{s_1}\\
    +&df^{N-s_2}(x_{s_2})\,B\,df^{s_2-1}(x_1)\Delta_i\kappa_{s_2}\\
    +&df^{N-s_3}(x_{s_3})\,B\,df^{s_3-1}(x_1)\Delta_i\kappa_{s_3}.\\
    =&0.
\end{split}\end{align}

Then, plugging this into the previous equation, we have by Perturbation \eqref{perturbation_which_varys_the_first_differential}

\begin{align}\begin{split}\nonumber
    \Delta df^N(x_1) = -\sum_{i=0}^{k-1} 
    \bigg \{ &df^{N-s_1}(x_{s_1})\,B\,df^{s_1-1}(x_1)\Delta_i\kappa_{s_1}\\
    +&df^{N-s_2}(x_{s_2})\,B\,df^{s_2-1}(x_1)\Delta_i\kappa_{s_2}\\
    +&df^{N-s_3}(x_{s_3})\,B\,df^{s_3-1}(x_1)\Delta_i\kappa_{s_3} \bigg\}
\end{split}\end{align}

So if we define a perturbation by changing the curvature at points $s_j$ by
$-\sum_{i=0}^{i=k-1}\Delta_i\kappa_{s_j}$ for $j=1,2,3$, then we this perturbation 
gives the same $\Delta df^N(x_1)$ as our original perturbation. 

So, to prove the claim that we may vary at $4$ points without changing 
$df^N(x_1)$, we note that if we vary the curvature at 
$s_{k_1},s_{k_2},s_{k_3},s_{k_4}$ by $\Delta\kappa_{k_1},\Delta\kappa_{k_2},
\Delta\kappa_{k_3},\Delta\kappa_{k_4}$, we have
\begin{align}\begin{split}\nonumber
    \Delta df^N(x_1) =& \, df^{N-k_1}(x_{k_1})\,B\,df^{k_1-1}(x_1)\Delta\kappa_{k_1}\\
    &\, df^{N-k_2}(x_{k_2})\,B\,df^{k_2-1}(x_1)\Delta\kappa_{k_2}\\
    &\, df^{N-k_3}(x_{k_3})\,B\,df^{k_3-1}(x_1)\Delta\kappa_{k_3}\\
    &\, df^{N-k_4}(x_{k_4})\,B\,df^{k_4-1}(x_1)\Delta\kappa_{k_4}.
\end{split}\end{align}

Then, this is $0$ if and only if
\begin{align}\begin{split}\nonumber
    &df^{k_4-k_1}(x_{k_1}) \,B\,\Delta\kappa_{k_1}\\
    +&df^{k_4-k_2}(x_{k_2})\,B\,df^{k_2-k_1}(x_{k_1})\Delta\kappa_{k_2}\\
    +&df^{k_4-k_3}(x_{k_3})\,B\,df^{k_3-k_1}(x_{k_1})\Delta\kappa_{k_3}\\
    +&B\,df^{k_4-k_1}(x_{k_1})\Delta\kappa_{k_4}\\
    =&0.
\end{split}\end{align}

Expanding this out, we get this is the same as

\begin{align}\begin{split}\nonumber
    &
    \begin{bmatrix}
    \frac{\del s_{k_4}}{\del \phi_{k_1}} && 0 \\
    \frac{\del \phi_{k_4}}{\del \phi_{k_1}} && 0
    \end{bmatrix} \Delta \kappa_{k_1}\\
    +&
    \begin{bmatrix}
    \frac{\del s_{k_2}}{\del s_{k_1}}\frac{s_{k_4}}{\del \phi_{k_2}} && 
    \frac{\del s_{k_2}}{\del \phi_{k_1}}\frac{s_{k_4}}{\del \phi_{k_2}} \\
    \frac{\del s_{k_2}}{\del s_{k_1}}\frac{\phi_{k_4}}{\del \phi_{k_2}} && 
    \frac{\del s_{k_2}}{\del \phi_{k_1}}\frac{\phi_{k_4}}{\del \phi_{k_2}} \\
    \end{bmatrix} \Delta \kappa_{k_2}\\
    +&
    \begin{bmatrix}
    \frac{\del s_{k_3}}{\del s_{k_1}}\frac{s_{k_4}}{\del \phi_{k_3}} && 
    \frac{\del s_{k_3}}{\del \phi_{k_1}}\frac{s_{k_4}}{\del \phi_{k_3}} \\
    \frac{\del s_{k_3}}{\del s_{k_1}}\frac{\phi_{k_4}}{\del \phi_{k_3}} && 
    \frac{\del s_{k_3}}{\del \phi_{k_1}}\frac{\phi_{k_4}}{\del \phi_{k_3}} \\
    \end{bmatrix} \Delta \kappa_{k_3}\\
    +&
    \begin{bmatrix}
    0 && 0 \\
    \frac{\del s_{k_4}}{\del s_{k_1}} && \frac{\del s_{k_4}}{\del \phi_{k_1}}
    \end{bmatrix} \Delta \kappa_{k_4}\\
    &=0.
\end{split}\end{align}

Then, considering each term independently, we get the following four relations must hold:

\begin{align}\begin{split}\nonumber
    &\frac{\del s_{k_4}}{\del \phi_{k_1}}\Delta\kappa_{k_4} = \frac{\del s_{k_2}}{\del \phi_{k_1}}\frac{\del \phi_{k_4}}{\del \phi_{k_2}}\Delta \kappa_{k_2} + \frac{\del s_{k_3}}{\del \phi_{k_1}}\frac{\del \phi_{k_4}}{\del \phi_{k_3}}\Delta \kappa_{k_3},\\
    &\frac{\del s_{k_2}}{\del \phi_{k_1}}\frac{\del s_{k_4}}{\del \phi_{k_2}}\Delta\kappa_{k_2} +
    \frac{\del s_{k_3}}{\del \phi_{k_1}}\frac{\del s_{k_4}}{\del \phi_{k_3}}\Delta\kappa_{k_3} = 0,\\
    &\frac{\del s_{k_4}}{\del \phi_{k_1}}\Delta\kappa_{k_1} + \frac{\del s_{k_2}}{\del s_{k_1}}\frac{\del s_{k_4}}{\del \phi_{k_2}}\Delta\kappa_{k_2} + \frac{\del s_{k_3}}{\del s_{k_1}}\frac{\del s_{k_4}}{\del \phi_{k_3}}\Delta\kappa_{k_3} = 0,\\
    &\frac{\del s_{k_4}}{\del s_{k_1}}\Delta\kappa_{k_4} = \frac{\del \phi_{k_4}}{\del \phi_{k_1}}\Delta \kappa_{k_1}+
    \frac{\del s_{k_2}}{\del s_{k_1}}\frac{\del \phi_{k_4}}{\del \phi_{k_2}}\Delta \kappa_{k_2} + \frac{\del s_{k_3}}{\del s_{k_1}}\frac{\del \phi_{k_4}}{\del \phi_{k_3}}\Delta \kappa_{k_3}.
\end{split}\end{align}

Without loss of generality, we can set $\Delta \kappa_{k_4}=1$ since we may normalize.  Then, the first three relations determine $\Delta \kappa_{k_2}$, $\Delta \kappa_{k_3}$, and $\Delta \kappa_{k_1}$.  To see this, we note the second and third equation determine $\Delta \kappa_{k_2}$ and $\Delta \kappa_{k_3}$ if
\begin{align}\begin{split}\nonumber
    \det \begin{bmatrix}
    \frac{\del s_{k_2}}{\del \phi_{k_1}}\frac{\del s_{k_4}}{\del \phi_{k_2}} && \frac{\del s_{k_3}}{\del \phi_{k_1}}\frac{\del s_{k_4}}{\del \phi_{k_3}}
    \\
    \frac{\del s_{k_2}}{\del s_{k_1}}\frac{\del s_{k_4}}{\del \phi_{k_2}} && \frac{\del s_{k_3}}{\del s_{k_1}}\frac{\del s_{k_4}}{\del \phi_{k_3}}
    \end{bmatrix}\neq 0.
\end{split}\end{align}
This is true since
\begin{align}\begin{split}\nonumber
    &\frac{\del s_{k_4}}{\del \phi_{k_2}},\frac{\del s_{k_4}}{\del \phi_{k_3}} \neq 0,
    \\
    & \det \begin{bmatrix}
    \frac{\del s_{k_2}}{\del \phi_{k_1}} && \frac{\del s_{k_3}}{\del \phi_{k_1}}
    \\
    \frac{\del s_{k_2}}{\del s_{k_1}} && \frac{\del s_{k_3}}{\del s_{k_1}}
    \end{bmatrix} = -\frac{\del s_{k_3}}{\del \phi_{k_2}}\det(df(x_1))) \neq 0.
\end{split}\end{align}
Thus, we have
\begin{align}\begin{split}\nonumber
    \begin{bmatrix}
    \Delta \kappa_{k_2} \\ \Delta \kappa_{k_3}
    \end{bmatrix}
    =
    \begin{bmatrix}
    \frac{\del s_{k_2}}{\del \phi_{k_1}}\frac{\del s_{k_4}}{\del \phi_{k_2}} && \frac{\del s_{k_3}}{\del \phi_{k_1}}\frac{\del s_{k_4}}{\del \phi_{k_3}}
    \\
    \frac{\del s_{k_2}}{\del s_{k_1}}\frac{\del s_{k_4}}{\del \phi_{k_2}} && \frac{\del s_{k_3}}{\del s_{k_1}}\frac{\del s_{k_4}}{\del \phi_{k_3}}
    \end{bmatrix}^{-1}\begin{bmatrix}
    0 \\ -\frac{\del s_{k_4}}{\del \phi_{k_1}}
    \end{bmatrix}.
\end{split}\end{align}
From here, one uses the first and second equations to fix $\Delta \kappa_{k_4}$.  These equations yield
\begin{align}\begin{split}\nonumber
    \begin{bmatrix}
    \frac{\del s_{k_4}}{\del \phi_{k_1}}\Delta \kappa_{k_4} \\ 0
    \end{bmatrix} &= \begin{bmatrix}
    \frac{\del s_{k_2}}{\del \phi_{k_1}}\frac{\del \phi_{k_4}}{\del \phi_{k_2}} && \frac{\del s_{k_3}}{\del \phi_{k_1}}\frac{\del \phi_{k_4}}{\del \phi_{k_3}}
    \\
    \frac{\del s_{k_2}}{\del \phi_{k_1}}\frac{\del s_{k_4}}{\del \phi_{k_2}} && \frac{\del s_{k_3}}{\del \phi_{k_1}}\frac{\del s_{k_4}}{\del \phi_{k_3}}
    \end{bmatrix}\begin{bmatrix}
    \Delta \kappa_{k_2} \\ \Delta \kappa_{k_3}
    \end{bmatrix}
    \\
    &= \begin{bmatrix}
    \frac{\del s_{k_2}}{\del \phi_{k_1}}\frac{\del \phi_{k_4}}{\del \phi_{k_2}} && \frac{\del s_{k_3}}{\del \phi_{k_1}}\frac{\del \phi_{k_4}}{\del \phi_{k_3}}
    \\
    \frac{\del s_{k_2}}{\del \phi_{k_1}}\frac{\del s_{k_4}}{\del \phi_{k_2}} && \frac{\del s_{k_3}}{\del \phi_{k_1}}\frac{\del s_{k_4}}{\del \phi_{k_3}}
    \end{bmatrix}
    \begin{bmatrix}
    \frac{\del s_{k_2}}{\del \phi_{k_1}}\frac{\del s_{k_4}}{\del \phi_{k_2}} && \frac{\del s_{k_3}}{\del \phi_{k_1}}\frac{\del s_{k_4}}{\del \phi_{k_3}}
    \\
    \frac{\del s_{k_2}}{\del s_{k_1}}\frac{\del s_{k_4}}{\del \phi_{k_2}} && \frac{\del s_{k_3}}{\del s_{k_1}}\frac{\del s_{k_4}}{\del \phi_{k_3}}
    \end{bmatrix}^{-1}\begin{bmatrix}
    0 \\ -\frac{\del s_{k_4}}{\del \phi_{k_1}}
    \end{bmatrix}.
\end{split}\end{align}
Simplifying we obtain
\begin{align}\begin{split}\nonumber
    \Delta \kappa_{k_4}\bigg(\frac{\del s_{k_3}}{\del \phi_{k_2}}\frac{\del s_{k_4}}{\del \phi_{k_2}}\frac{\del s_{k_4}}{\del \phi_{k_3}}\bigg)\det(df(x_1)) &= \frac{\del s_{k_2}}{\del \phi_{k_1}}\frac{\del s_{k_3}}{\phi_{k_1}}\bigg(
    \frac{\del s_{k_4}}{\del \phi_{k_2}}\frac{\del \phi_{k_4}}{\del \phi_{k_3}} - \frac{\del s_{k_4}}{\del \phi_{k_3}}\frac{\del \phi_{k_4}}{\del \phi_{k_2}}
    \bigg)
    \\
    & = \frac{\del s_{k_2}}{\del \phi_{k_1}}\frac{\del s_{k_3}}{\phi_{k_1}}\frac{\del s_{k_3}}{\del \phi_{k_2}}\det(df(x_3))
\end{split}\end{align}

All that remains after this then is to verify that the fourth equation holds true with these values.  One may verify that it does.

\end{proof}

\section{Perturbing the $n$-th partial derivatives of the differential independently}\label{appendix_perturb_nth_derivative_independently}

The main outcome of this section is to prove the statement for Perturbation \ref{perturbation_which _lets_you_vary_all_derivatives_independently} (which is a more detailed description of Theorem \eqref{theorem_perturb_many_independently}).  Recall for this perturbation we consider a changes in the curvature such that
\begin{align}\begin{split}\nonumber
    &\Delta \kappa_i^{(j)}=0\text{ for }0\leq j \leq n-1\\
    &\Delta \kappa_i^{(n)}=\eps_i.
\end{split}\end{align}
for $1\leq i\leq n+2$.  To begin, we show how this effects the $n^{th}$ differentials of $s_{n+3},\phi_{n+3}$.  We have

\setcounter{maintheorem}{2}
\setcounter{lemma}{0}

\begin{lemma}
With the perturbation given as above, the changes to the $n^{th}$ differentials of $s_{n+3},\phi_{n+3}$ are

\begin{align}\begin{split}\nonumber
    &\Delta\frac{\del^n s_{n+3}}{\del s_0^{n-k} \del \phi_0^k} = 2\sum_{i=1}^{n+2}\frac{\del s_{n+3}}{\del \phi_i}\bigg(\frac{\del s_i}{\del s_0}\bigg)^{n-k}\bigg(\frac{\del s_i}{\del \phi_0}\bigg)^k\eps_i,\\
    &\Delta\frac{\del^n \phi_{n+3}}{\del s_0^{n-k} \del \phi_0^k} = 2\sum_{i=1}^{n+2}\frac{\del \phi_{n+3}}{\del \phi_i}\bigg(\frac{\del s_i}{\del s_0}\bigg)^{n-k}\bigg(\frac{\del s_i}{\del \phi_0}\bigg)^k\eps_i.
\end{split}\end{align}
\end{lemma}
\begin{proof}
    To see this we prove the statement for the derivatives of $s_{n+3}$ as the proof for the derivatives of $\phi_{n+3}$ is similar. We prove the result first for degree $1$ partial derivatives.  Recall that $df^{n+3}(x_0)$ depends on $\kappa_j$ by
    \begin{align}\begin{split}\nonumber
        df^{n+3}(x_0)
        =
        \kappa_jdf^{n+3-j}(x_{j})Bdf^j(x_0) + \text{ terms not involving }\kappa_j.
    \end{split}\end{align}
    Expanding the term multiplying $\kappa_j$, we have
    \begin{align}\begin{split}
        df^{n+3-j}(x_j)Bdf^j(x_0) = 2 \begin{bmatrix}
            \frac{\del s_{n+3}}{\del \phi_{j}}\frac{\del s_{j}}{\del s_0} && \frac{\del s_{n+3}}{\del \phi_{j}}\frac{\del s_{j}}{\del \phi_0} \\ \frac{\del \phi_{n+3}}{\del \phi_{j}}\frac{\del s_{j}}{\del s_0} && \frac{\del \phi_{n+3}}{\del \phi_{j}}\frac{\del s_{j}}{\del \phi_0}
        \end{bmatrix}.
    \end{split}\end{align}
    Thus, adding to each $\kappa_j$ some change $\eps_j$, we have to first order 
    \begin{align}\begin{split}\nonumber
        \Delta\frac{\del s_{n+3}}{\del s_0} &= 2\sum_{j=1}^{n+2}\frac{\del s_{n+3}}{\del \phi_{j+1}}\frac{\del s_{j+1}}{\del s_0}\eps_j,
        \\
        \Delta\frac{\del s_{n+3}}{\del \phi_0} &= 2\sum_{j=1}^{n+2}\frac{\del s_{n+3}}{\del \phi_{j+1}}\frac{\del s_{j+1}}{\del \phi_0}\eps_j.
    \end{split}\end{align}
    The higher order partials are similar: we observe 
    \begin{align}\begin{split}\nonumber
        \frac{\del^k}{\del s_0^{k-i}\del \phi_0^i}\bigg(df^{n+3}(x_0)\bigg) &= 2\kappa_j^{(k)}\bigg(\frac{\del s_j}{\del s_0}\bigg)^{k-i}\bigg(\frac{\del s_j}{\del \phi_0}\bigg)^{i}df^{n+3-j}(x_{j})Bdf^j(x_0)
        \\
        &+ R(s_0,\phi_0)
    \end{split}\end{align}
    where the remainder term $R$ only contains terms potentially containing $\kappa_j^{(i)}$ where $i$ is at most $k-1$.  Thus, if we perturb $\kappa_j$      by adding to its $n^{th}$ derivative $\eps_j$, we obtain to first order
    \begin{align}\begin{split}\nonumber
        \Delta\frac{\del^k}{\del s_0^{k-i}\del \phi_0^i}\bigg(df^{n+3}(x_0)\bigg) &= 2\sum_{j=1}^{n+2}\frac{\del s_{n+3}}{\del \phi_j}\bigg(\frac{\del s_j}{\del s_0}\bigg)^{k-i}\bigg(\frac{\del s_j}{\del \phi_0}\bigg)^{i}\eps_j,
    \end{split}\end{align}
    which is precisely the statement of the lemma.
\end{proof}

So, we have proven the $n^{th}$ partial derivatives depend linear on $\eps=(\eps_1,...,\eps_{n+2})$.  

observe that there are exactly $2(n+1)$ such partial derivatives.  As it turns out, however, there are really only $n+2$ degrees of freedom.  What we mean by this is that given $n+2$ of these terms, one may determine the other $n$. This is ultimately due to the area preservation relation
\begin{align}\begin{split}\nonumber
    \det\bigg(df^q(x_0)\bigg)=\frac{\sin(\phi_0)}{\sin(\phi_q)}.
\end{split}\end{align}
More precisely, we prove
\begin{lemma}
    For all billiard maps, the partial derivatives
    \begin{align}\begin{split}\nonumber
        \bigg\{\frac{\del^n \phi_{n+3}}{\del s_0^n},\frac{\del^n \phi_{n+3}}{\del s_0^{n-1}\del \phi_0^1},...,\frac{\del^n \phi_{n+3}}{\del s_0^{1}\del \phi_0^{n-1}}\bigg\}
    \end{split}\end{align}
    can be determined from
    \begin{align}\begin{split}\nonumber
        &\bigg\{\frac{\del^n s_{n+3}}{\del s_0^n},\frac{\del^n s_{n+3}}{\del s_0^{n-1}\del \phi_0^1},...,\frac{\del^n s_{n+3}}{\del \phi_0^{n-1}},\frac{\del^n \phi_{n+3}}{\del \phi_0^n}\bigg\}
        \\
        &\cup\bigg\{\text{all derivatives of $s_{n+3}$ and $\phi_{n+3}$ of order less than $n$}\bigg\}.
    \end{split}\end{align}
\end{lemma}
\begin{proof}
    We first observe the statement follows for $n=0$ from the area preservation, which explicitly gives us:
    \begin{align}\begin{split}\nonumber
        \frac{\del s_{n+3}}{\del s_0}\frac{\del \phi_{n+3}}{\del \phi_0} - \frac{\del s_{n+3}}{\del \phi_0}\frac{\del \phi_{n+3}}{\del s_0} = \frac{\sin(\phi_0)}{\sin(\phi_{n+3})}.
    \end{split}\end{align}
    We now show the statement of the lemma for $n=1$. To do this we take partial derivatives of this equation to obtain
    \begin{align}\begin{split}\nonumber
        &\frac{\del^2 s_{n+3}}{\del s_0^2}\frac{\del \phi_{n+3}}{\del \phi_0} + \frac{\del s_{n+3}}{\del s_0}\frac{\del^2 \phi_{n+3}}{\del s_0\del\phi_0} 
        \\
        &- \frac{\del^2 s_{n+3}}{\del s_0\del\phi_0}\frac{\del \phi_{n+3}}{\del s_0} - \frac{\del s_{n+3}}{\del \phi_0}\frac{\del^2 \phi_{n+3}}{\del s_0^2} = \frac{\del}{\del s_0}\bigg(\frac{\sin(\phi_0)}{\sin(\phi_{n+3})}\bigg),
        \\
        &\frac{\del^2 s_{n+3}}{\del s_0 \del\phi_0}\frac{\del \phi_{n+3}}{\del \phi_0} + \frac{\del s_{n+3}}{\del s_0}\frac{\del^2 \phi_{n+3}}{\del\phi_0^2} 
        \\
        &- \frac{\del^2 s_{n+3}}{\phi_0^2}\frac{\del \phi_{n+3}}{\del s_0} - \frac{\del s_{n+3}}{\del \phi_0}\frac{\del^2 \phi_{n+3}}{\del s_0\del \phi_0} = \frac{\del}{\del \phi_0}\bigg(\frac{\sin(\phi_0)}{\sin(\phi_{n+3})}\bigg).
    \end{split}\end{align}
    Thus, if we define vectors 
    \begin{align}\begin{split}\nonumber
        \vec x &= \bigg(
        \frac{\del^2 s_{n+3}}{\del s_0^2},
        \frac{\del^2 s_{n+3}}{\del s_0\del \phi_0},
        \frac{\del^2 s_{n+3}}{\del \phi_0^2},
        \frac{\del^2 \phi_{n+3}}{\del s_0^2},
        \frac{\del^2 \phi_{n+3}}{\del s_0\del \phi_0},
        \frac{\del^2 \phi_{n+3}}{\del \phi_0^2}
        \bigg)
        \\
        \vec y &= \bigg(
        \frac{\del}{\del s_0}\bigg(\frac{\sin(\phi_0)}{\sin(\phi_{n+3})}\bigg),
        \frac{\del}{\del \phi_0}\bigg(\frac{\sin(\phi_0)}{\sin(\phi_{n+3})}\bigg)
        \bigg),
    \end{split}\end{align}
    we have
    \begin{align}\begin{split}\nonumber
        A_2\vec x = \vec y,
    \end{split}\end{align}
    where $A_2$ is given by
    \setcounter{MaxMatrixCols}{20}
    \begin{align}\begin{split}\nonumber
        A_2 = \begin{bmatrix}
        \frac{\del \phi_{n+3}}{\del \phi_0} &&
        -\frac{\del \phi_{n+3}}{\del s_0} &&
        0 &&
        -\frac{\del s_{n+3}}{\del \phi_0} &&
        \frac{\del s_{n+3}}{\del s_0} &&
        0 \\
        0 &&
        \frac{\del \phi_{n+3}}{\del \phi_0} &&
        -\frac{\del \phi_{n+3}}{\del s_0} &&
        0 &&
        -\frac{\del s_{n+3}}{\del \phi_0} &&
        \frac{\del s_{n+3}}{\del s_0} 
        \end{bmatrix}.
    \end{split}\end{align}
    We thus note that this implies we may write $\bigg\{\frac{\del^2 \phi_{n+3}}{\del s_0^2},
    \frac{\del^2 \phi_{n+3}}{\del s_0\del \phi_0}\bigg\}$ in terms of the other components of $\vec x$ and the terms in $\vec y$ since the submatrix
    \begin{align}\begin{split}\nonumber
        \begin{bmatrix}
        -\frac{\del s_{n+3}}{\del \phi_0} &&
        \frac{\del s_{n+3}}{\del s_0} \\
        0 &&
        -\frac{\del s_{n+3}}{\del \phi_0}
        \end{bmatrix}
    \end{split}\end{align}
    has determinant $\bigg(\frac{\del s_{n+3}}{\del \phi_0}\bigg)^2 \neq 0$.
    Similar results hold for higher order partials.  To see this, we continue to take partial derivatives with respect to $s_0$ and $\phi_0$ of the previous two equations.  If we examine only the terms where $n^{th}$ order partial derivatives of $\phi_{n+3}$ appear, we obtain
    \begin{align}\nonumber
        A_n \vec v = \vec b_n
    \end{align}
    where
    \begin{align}\begin{split}\nonumber
        \vec v &= \bigg(
        \frac{\del^n \phi_{n+3}}{\del s_0^n},...,
        \frac{\del^n \phi_{n+3}}{\del s_0\del\phi^{n-1}}
        \bigg),
        \\
        A_n &= \begin{bmatrix}
        -\frac{\del s_{n+3}}{\del \phi_0} && \frac{\del s_{n+3}}{\del s_0} && 0 && \cdots && 0
        \\
        0 && -\frac{\del s_{n+3}}{\del \phi_0} && \frac{\del s_{n+3}}{\del s_0} && \cdots && 0
        \\
        0 && 0 && -\frac{\del s_{n+3}}{\del \phi_0} && \cdots && 0
        \\
        \vdots && \vdots && \vdots && \ddots && \vdots
        \\
        0 && 0 && 0 && \cdots && -\frac{\del s_{n+3}}{\del \phi_0}
        \end{bmatrix},
    \end{split}\end{align}
    and $\vec b_n$ contains only terms in
    \begin{align}\begin{split}\nonumber
        \bigg\{\frac{\del^n s_{n+3}}{\del s_0^n},\frac{\del^n s_{n+3}}{\del s_0^{n-1}\del \phi_0^1},...,\frac{\del^n s_{n+3}}{\del \phi_0^{n-1}},\frac{\del^n \phi_{n+3}}{\del \phi_0^n}\bigg\}
    \end{split}\end{align}
    and derivatives of $s_{n+3}$ and $\phi_{n+3}$ of order less than $n$.  Thus, since $\det A_n = \bigg(-\frac{\del s_{n+3}}{\del \phi_0}\bigg)^n\neq 0$, the proof is complete.
\end{proof}
Thus, we seek to vary only all the partial derivatives order $n$ of $s_{n+3}$, and the term $\frac{\del^n \phi_{n+3}}{\del \phi_0^n}$, since the rest are determined by these.  We thus examine the following $n+2$ by $n+2$ matrix:

\begin{align}\begin{split}\nonumber
    &M_{k,l}=\frac{\del s_{n+3}}{\del \phi_l}\bigg(\frac{\del s_l}{\del s_0}\bigg)^{n-(k-1)}\bigg(\frac{\del s_l}{\del \phi_0}\bigg)^{k-1}\text{ for }1\leq k \leq n+1, 1\leq l \leq n+2,\\
    &M_{n+2,l}=\frac{\del \phi_{n+3}}{\del \phi_l}\bigg(\frac{\del s_l}{\del \phi_0}\bigg)^{n}\text{ for }1\leq l \leq n+2
\end{split}\end{align}
and show it has non-zero determinant.  This implies we may vary each of the differentials of order $n$ of $s_{n+3}$ and $\frac{\del^n \phi_{n+3}}{\del \phi_0^n}$ independently.  

To obtain this result, we perform row and column operations on the matrix $(M_{i,j})_{1\leq i,j\leq n+2}$ to be lower triangular, and we find the diagonal entries to be non-zero.  These statements are summarized in the following lemma:
\begin{lemma}\label{lemma_perturb_nth_differential_independently_technical}
    Given the matrix $M$ as defined above, we may perform a series of row and column operations to reduce the matrix into a lower triangular matrix with diagonal entries given by
    \begin{align}\begin{split}\nonumber
    M_{k,k}&=\frac{\del s_{n+3}}{\del \phi_k}
        \bigg\{\bigg[\frac{\del s_k}{\del \phi_1}\det(df(x_0))\bigg]
        ...\\
        &\qquad
        \bigg[\frac{\del s_1}{\del s_0}...\frac{\del s_{k-2}}{\del s_0}\frac{\del s_k}{\del \phi_{k-1}}\det(df^{k-1}(x_0))\bigg]\bigg\}\times\\
        &\qquad \bigg\{\bigg[\frac{\del s_1}{\del s_0}...\frac{\del s_{k-1}}{\del s_0}\frac{\del s_{k+1}}{\del \phi_k}\det(df^k(x_0))\bigg]...\\
        &\qquad \bigg[\frac{\del s_1}{\del s_0}...\frac{\del s_{k-1}}{\del s_0}\frac{\del s_{n+1}}{\del \phi_k}\det(df^k(x_0))\bigg]\bigg\}\times\\
        &\qquad \bigg[
        \bigg(\frac{\del s_k}{\del s_0}\bigg)^{n-(k-1)}
        \bigg(\frac{\del s_{k+1}}{\del s_0}\bigg)^{n-k}\bigg(\frac{\del s_{k+2}}{\del s_0}\bigg)^{n-(k+1)}...\bigg(\frac{\del s_n}{\del s_0}\bigg)\bigg]
    \end{split}\end{align}
    for $1\leq k \leq n+1$, and
    \begin{align}\begin{split}\nonumber
        M_{n+2,n+2} &= \det(df(x_{n+2}))\bigg(\frac{\del s_{n+3}}{\del \phi_0}\bigg)^n\bigg[\frac{\del s_{n+2}}{\del \phi_1}...\frac{\del s_{n+2}}{\del \phi_{n+1}}\bigg]\times\\
    &\qquad\bigg[\bigg(\frac{\del s_{n+1}}{\del \phi_1}...\frac{\del s_{n+1}}{\del \phi_n}\bigg)^2\bigg(\frac{\del s_{n}}{\del \phi_1}...\frac{\del s_{n}}{\del \phi_{n-1}}\bigg)^2...\bigg(\frac{\del s_2}{\del \phi_1}\bigg)^2\bigg].
    \end{split}\end{align}
    These are all non-zero provided
    \begin{align}\begin{split}\nonumber
        \frac{\del s_{j}}{\del \phi_{i}} \neq 0
    \end{split}\end{align}
    for $0\leq i < j \leq n+3$.  We note that when these entries along the diagonal are non-zero, then we have that the determinant of $M$ is non-zero.
\end{lemma}

\begin{proof}
We proceed to show this using row reduction to get the submatrix of the first $n+1$ rows and columns to row reduced form.  We perform the following row operations for $2\leq k \leq n+1$:

\begin{align}\begin{split}\nonumber
    &R_k \rightarrow R_k\bigg(\frac{\del s_1}{\del s_0}\bigg)^{k-1},\\
    &R_k \rightarrow R_k - R_1\bigg(\frac{\del s_1}{\del \phi_0}\bigg)^{k-1}.
\end{split}\end{align}

This then gives us for $2\leq k \leq n+1$, $2\leq l \leq n+2$
\begin{align}\begin{split}\nonumber
    M_{k,l}&=\frac{\del s_{n+3}}{\del \phi_l}\bigg(\frac{\del s_l}{\del s_0}\bigg)^{n-(k-1)}\bigg(\frac{\del s_l}{\del \phi_0}\bigg)^{k-1}\bigg(\frac{\del s_1}{\del s_0}\bigg)^{k-1} \\
    &\qquad- \frac{\del s_{n+3}}{\del \phi_l}\bigg(\frac{\del s_l}{\del s_0}\bigg)^n\bigg(\frac{\del s_1}{\del \phi_0}\bigg)^{k-1}\\
    &= \frac{\del s_{n+3}}{\del \phi_l}\bigg(\frac{\del s_l}{\del s_0}\bigg)^{n-(k-1)}\bigg[\bigg(\frac{\del s_l}{\del \phi_0}\bigg)^{k-1}\bigg(\frac{\del s_1}{\del s_0}\bigg)^{k-1}\\
    &\qquad - \bigg(\frac{\del s_l}{\del s_0}\bigg)^{k-1}\bigg(\frac{\del s_1}{\del \phi_0}\bigg)^{k-1}\bigg]\\
    &= \frac{\del s_{n+3}}{\del \phi_l}\bigg(\frac{\del s_l}{\del s_0}\bigg)^{n-(k-1)}\bigg[\frac{\del s_l}{\del \phi_0}\frac{\del s_1}{\del s_0}-\frac{\del s_l}{\del s_0}\frac{\del s_1}{\del \phi_0}\bigg]\Sigma_2(k,l)\\
    &=\frac{\del s_{n+3}}{\del \phi_l}\bigg(\frac{\del s_l}{\del s_0}\bigg)^{n-(k-1)}\bigg[\frac{\del s_l}{\del \phi_1}\det(df(x_0))\bigg]\Sigma_2(k,l).
\end{split}\end{align}

where we define
\begin{align}\begin{split}\nonumber
    \Sigma_2(k,l)=\sum_{i=0}^{k-2}\bigg(\frac{\del s_l}{\del \phi_0}\bigg)^{k-2-i}\bigg(\frac{\del s_1}{\del s_0}\bigg)^{k-2-i}\bigg(\frac{\del s_l}{\del s_0}\bigg)^i\bigg(\frac{\del s_1}{\del \phi_0}\bigg)^i.
\end{split}\end{align}

Now we use the second row to eliminate the entries in the second column below the second row.  So we make the following operations for $3\leq k \leq n+1$:

\begin{align}\begin{split}\nonumber
    &R_k\rightarrow R_k\bigg(\frac{\del s_2}{\del s_0}\bigg)^{k-2},\\
    &R_k \rightarrow R_k - R_2\Sigma_2(k,2).
\end{split}\end{align}

We then obtain for $3\leq k\leq n+1$, $3,\leq l,\leq n_2$

\begin{align}\begin{split}\nonumber
    M_{k,l}&=\frac{\del s_{n+3}}{\del \phi_l}\bigg(\frac{\del s_l}{\del s_0}\bigg)^{n-(k-1)}\bigg[\frac{\del s_l}{
    \del \phi_1}\det(df(x_0))\bigg]\Sigma_2(k,l)\bigg(\frac{\del s_2}{\del s_0}\bigg)^{k-2}\\
    &\qquad- \frac{\del s_{n+3}}{\del \phi_l}\bigg(\frac{\del s_l}{\del s_0}\bigg)^{n-1}\bigg[\frac{\del s_l}{\del \phi_1}\det(df(x_0))\bigg]\Sigma_2(k,2)\\
    &=\frac{\del s_{n+3}}{\del \phi_l}\bigg(\frac{\del s_l}{\del s_0}\bigg)^{n-(k-1)}\bigg[\frac{\del s_l}{\del \phi_1}\det(df(x_0))\bigg]\times\\
    &\qquad\bigg[\Sigma_2(k,l)\bigg(\frac{\del s_2}{\del s_0}\bigg)^{k-2}-\bigg(\frac{\del s_l}{\del s_0}\bigg)^{k-2}\Sigma_2(k,2)\bigg]
\end{split}\end{align}

Here we examine this last term more carefully.  we have
\begin{align}\begin{split}\nonumber
    &\bigg[\Sigma_2(k,l)\bigg(\frac{\del s_2}{\del s_0}\bigg)^{k-2}-\bigg(\frac{\del s_l}{\del s_0}\bigg)^{k-2}\Sigma_2(k,2)\bigg]\\
    &=\sum_{i=0}^{k-2}\bigg[\bigg(\frac{\del s_l}{\del \phi_0}\bigg)^{k-2-i}\bigg(\frac{\del s_1}{\del s_0}\bigg)^{k-2-i}\bigg(\frac{\del s_l}{\del s_0}\bigg)^i\bigg(\frac{\del s_1}{\del \phi_0}\bigg)^i\bigg(\frac{\del s_2}{\del s_0}\bigg)^{k-2}\\
    &\qquad-\bigg(\frac{\del s_2}{\del \phi_0}\bigg)^{k-2-i}\bigg(\frac{\del s_1}{\del s_0}\bigg)^{k-2-i}\bigg(\frac{\del s_2}{\del s_0}\bigg)^{i}\bigg(\frac{\del s_1}{\del \phi_0}\bigg)^{i}\bigg(\frac{\del s_l}{\del s_0}\bigg)^{k-2}\bigg]\\
    &=\sum_{i=0}^{k-3}\bigg(\frac{\del s_1}{\del s_0}\bigg)^{k-2-i}\bigg(\frac{\del s_1}{\del \phi_0}\bigg)^{i}\bigg(\frac{\del s_l}{\del s_0}\bigg)^i\bigg(\frac{\del s_2}{\del s_0}\bigg)^i\times\\
    &\qquad\bigg[\bigg(\frac{\del s_l}{\del \phi_0}\bigg)^{k-2-i}\bigg(\frac{\del s_2}{\del s_0}\bigg)^{k-2-i}-\bigg(\frac{\del s_2}{\del \phi_0}\bigg)^{k-2-i}\bigg(\frac{\del s_l}{\del s_0}\bigg)^{k-2-i}\bigg]\\
    &=\bigg(\frac{\del s_1}{\del s_0}\bigg)\sum_{i_1=0}^{k-3}\bigg(\frac{\del s_1}{\del s_0}\bigg)^{k-3-i_1}\bigg(\frac{\del s_1}{\del \phi_0}\bigg)^{i_1}\bigg(\frac{\del s_l}{\del s_0}\bigg)^{i_1}\bigg(\frac{\del s_2}{\del s_0}\bigg)^{i_1}\times\\
    &\qquad\bigg[\bigg(\frac{\del s_l}{\del \phi_0}\bigg)\bigg(\frac{\del s_2}{\del s_0}\bigg)-\bigg(\frac{\del s_2}{\del \phi_0}\bigg)\bigg(\frac{\del s_l}{\del s_0}\bigg)\bigg]\times\\
    &\qquad\sum_{i_2=0}^{k-3-i_1}\bigg(\frac{\del s_2}{\del s_0}\bigg)^{k-3-i_1-i_2}\bigg(\frac{\del s_l}{\del \phi_0}\bigg)^{k-3-i_1-i_2}\bigg(\frac{\del s_l}{\del s_0}\bigg)^{i_2}\bigg(\frac{\del s_2}{\del \phi_0}\bigg)^{i_2}\\
    &= \bigg(\frac{\del s_1}{\del s_0}\bigg)\bigg(\frac{\del s_l}{\del \phi_2}\bigg)\det(df^2(x_0))\Sigma_3(k,l),
\end{split}\end{align}
where we define
\begin{align}\begin{split}\nonumber
    \Sigma_3(k,l)&=\sum_{i_1=0}^{k-3}\sum_{i_2=0}^{k-3-i_1}\bigg(\frac{\del s_l}{\del \phi_0}\bigg)^{k-3-i_1-i_2}\times\\
    &\qquad\bigg(\frac{\del s_1}{\del s_0}\bigg)^{k-3-i_1}\bigg(\frac{\del s_2}{\del s_0}\bigg)^{k-3-i_2}\times\\
    &\qquad\bigg(\frac{\del s_1}{\del \phi_0}\bigg)^{i_1}\bigg(\frac{\del s_2}{\del \phi_0}\bigg)^{i_2}\bigg(\frac{\del s_l}{\del s_0}\bigg)^{i_1+i_2}
\end{split}\end{align}
Thus we obtain for $3\leq k\leq n+1$, $3,\leq l,\leq n+2$

\begin{align}\begin{split}\nonumber
    M_{k,l}=\frac{\del s_{n+3}}{\del \phi_l}\bigg(\frac{\del s_l}{\del s_0}\bigg)^{n-(k-1)}\bigg[\frac{\del s_l}{\del \phi_1}\det(df(x_0))\bigg]\bigg[\frac{\del s_1}{\del s_0}\frac{\del s_l}{\del \phi_2}\det(df^2(x_0))\bigg]\Sigma_3(k,l).
\end{split}\end{align}
We will do one more reduction explicitly before describing the inductive process.  So for $4\leq k \leq n+1$ we perform the following row operations:
\begin{align*}\begin{split}\nonumber
    &R_k \rightarrow R_k\bigg(\frac{\del s_3}{\del s_0}\bigg)^{k-3},\\
    &R_k \rightarrow R_k-R_3\Sigma_3(k,3).
\end{split}\end{align*}
We obtain then for $4\leq k \leq n+1$, $4 \leq l \leq n+2$
\begin{align}\begin{split}\nonumber
    &M_{kl}=\bigg\{\frac{\del s_{n+3}}{\del \phi_l}\bigg(\frac{\del s_l}{\del s_0}\bigg)^{n-(k-1)}\bigg[\frac{\del s_l}{\del \phi_1}\det(df(x_0))\bigg]\times\\
    &\qquad\bigg[\frac{\del s_1}{\del s_0}\frac{\del s_l}{\del \phi_2}\det(df^2(x_0))\bigg]\Sigma_3(k,l)\bigg(\frac{\del s_3}{\del s_0}\bigg)^{k-3}\bigg\}\\
    &-\bigg\{\frac{\del s_{n+3}}{\del \phi_l}\bigg(\frac{\del s_l}{\del s_0}\bigg)^{n-2}\bigg[\frac{\del s_l}{\del \phi_1}\det(df(x_0))\bigg]\times\\
    &\qquad\bigg[\frac{\del s_1}{\del s_0}\frac{\del s_l}{\del \phi_2}\det(df^2(x_0))\bigg]\Sigma_3(k,3)\bigg\}\\
    &=\frac{\del s_{n+3}}{\del \phi_l}\bigg(\frac{\del s_l}{\del s_0}\bigg)^{n-(k-1)}\bigg[\frac{\del s_l}{\del \phi_1}\det(df(x_0))\bigg]\bigg[\frac{\del s_1}{\del s_0}\frac{\del s_l}{\del \phi_2}\det(df^2(x_0))\bigg]\times\\
    &\qquad\bigg[\bigg(\frac{\del s_3}{\del s_0}\bigg)^{k-3}\Sigma_3(k,l)-\bigg(\frac{\del s_l}{\del s_0}\bigg)^{k-3}\Sigma_3(k,3)\bigg].
\end{split}\end{align}
And here we examine the last term again.  We obtain
\begin{align}\begin{split}\nonumber
    &\bigg[\bigg(\frac{\del s_3}{\del s_0}\bigg)^{k-3}\Sigma_3(k,l)-\bigg(\frac{\del s_l}{\del s_0}\bigg)^{k-3}\Sigma_3(k,2)\bigg]\\
    &=\sum_{i_1=0}^{k-3}\sum_{i_2=0}^{k-3-i_1}\bigg\{\bigg[\bigg(\frac{\del s_l}{\del \phi_0}\bigg)^{k-3-i_1-i_2}\bigg(\frac{\del s_1}{\del s_0}\bigg)^{k-3-i_1}\times\\
    &\qquad\bigg(\frac{\del s_2}{\del s_0}\bigg)^{k-3-i_2}\bigg(\frac{\del s_1}{\del \phi_0}\bigg)^{i_1}\bigg(\frac{\del s_2}{\del \phi_0}\bigg)^{i_2}\bigg(\frac{\del s_l}{\del s_0}\bigg)^{i_1+i_2}\bigg(\frac{\del s_3}{\del s_0}\bigg)^{k-3}\bigg]\\
    &\qquad - \bigg[\bigg(\frac{\del s_3}{\del \phi_0}\bigg)^{k-3-i_1-i_2}\bigg(\frac{\del s_1}{\del s_0}\bigg)^{k-3-i_1}\times\\
    &\qquad\bigg(\frac{\del s_2}{\del s_0}\bigg)^{k-3-i_2}\bigg(\frac{\del s_1}{\del \phi_0}\bigg)^{i_1}\bigg(\frac{\del s_2}{\del \phi_0}\bigg)^{i_2}\bigg(\frac{\del s_3}{\del s_0}\bigg)^{i_1+i_2}\bigg(\frac{\del s_l}{\del s_0}\bigg)^{k-3}\bigg]\bigg\}\\
    &=\frac{\del s_1}{\del s_0}\frac{\del s_2}{\del s_0}\sum_{i_1=0}^{k-4}\sum_{i_2=0}^{k-4-i_1}\bigg\{\bigg(\frac{\del s_1}{\del s_0}\bigg)^{k-4-i_1} 
    \bigg(\frac{\del s_2}{\del s_0}\bigg)^{k-4-i_2}
    \times
    \\
    &\qquad \bigg(\frac{\del s_1}{\del \phi_0}\bigg)^{i_1}
    \bigg(\frac{\del s_2}{\del \phi_0}\bigg)^{i_2}\bigg(\frac{\del s_l}{\del s_0}\bigg)^{i_1+i_2}\bigg(\frac{\del s_2}{\del s_0}\bigg)^{i_1+i_2} \times\\
    &\bigg[\bigg(\frac{\del s_l}{\del \phi_0}\bigg)^{k-3-i_1-i_2}\bigg(\frac{\del s_3}{\del s_0}\bigg)^{k-3-i_1-i_2}-\bigg(\frac{\del s_3}{\del \phi_0}\bigg)^{k-3-i_1-i_2}\bigg(\frac{\del s_l}{\del s_0}\bigg)^{k-3-i_1-i_2}\bigg]\bigg\}\\
    &= \frac{\del s_1}{\del s_0}\frac{\del s_2}{\del s_0}\frac{\del s_l}{\del \phi_3}\det(df^3(x_0))\Sigma_4(k,l)
\end{split}\end{align}
where here we define $\Sigma_4(k,l)$ by
\begin{align}\begin{split}\nonumber
    &\Sigma_4(k,l) = \sum_{i_1=0}^{k-4}\sum_{i_2=0}^{k-4-i_1}\sum_{i_3=0}^{k-4-i_1-i_2}\bigg\{\bigg(\frac{\del s_l}{\del \phi_0}\bigg)^{k-4-i_1-i_2-i_3}\times\\
    &\qquad \bigg(\frac{\del s_1}{\del s_0}\bigg)^{k-4-i_1}
    \bigg(\frac{\del s_2}{\del s_0}\bigg)^{k-4-i_2}
    \bigg(\frac{\del s_3}{\del s_0}\bigg)^{k-4-i_3}\times\\
    &\qquad \bigg(\frac{\del s_1}{\del \phi_0}\bigg)^{i_1}
    \bigg(\frac{\del s_2}{\del \phi_0}\bigg)^{i_2}
    \bigg(\frac{\del s_3}{\del \phi_0}\bigg)^{i_3}\times\\
    &\qquad \bigg(\frac{\del s_l}{\del s_0}\bigg)^{i_1+i_2+i_3}.
\end{split}\end{align}
So we get for $4\leq k \leq n+1$, $4 \leq l \leq n+2$
\begin{align}\begin{split}
    &M_{k,l}=\frac{\del s_{n+3}}{\del \phi_l}\bigg(\frac{\del s_l}{\del s_0}\bigg)^{n-(k-1)}\bigg[\frac{\del s_l}{\del \phi_1}\det(df(x_0))\bigg]\bigg[\frac{\del s_1}{\del s_0}\frac{\del s_l}{\del \phi_2}\det(df^2(x_0))\bigg]\times\\
    &\qquad \bigg[\frac{\del s_1}{\del s_0}\frac{\del s_2}{\del s_0}\frac{\del s_l}{\del \phi_3}\det(df^3(x_0))\bigg]\Sigma_4(k,l).
\end{split}\end{align}
We proceed inductively.  In the end we obtain for $1\leq k \leq n+1$, $k \leq l \leq n+2$
\begin{align}\begin{split}\nonumber
    M_{k,l} &= \frac{\del s_{n+3}}{\del \phi_l}\bigg(\frac{\del s_l}{\del s_0}\bigg)^{n-(k-1)}\bigg[\frac{\del s_l}{\del \phi_1}\det(df(x_0))\bigg]...\\
    &\qquad\bigg[\frac{\del s_1}{\del s_0}...\frac{\del s_{k-2}}{\del s_0}\frac{\del s_l}{\del \phi_{k-1}}\det(df^{k-1}(x_0))\bigg].
\end{split}\end{align}

So now we eliminate upwards for each row to obtain the top left $n+1$ by $n+1$ matrix to have zeroes off of the diagonal.  After this we will do column operations to eliminate the last column, and then we will show that all the values along the diagonal are non-zero, which implies the determinant is not zero.

We do this for the first few terms and then describe the inductive proof.  We do the following row operations:
\begin{align}\begin{split}\nonumber
    &R_1 \rightarrow R_1\bigg[\frac{\del s_2}{\del \phi_1}\det(df(x_0))\bigg]\\
    &R_1 \rightarrow R_1-R_2\frac{\del s_2}{\del s_0}.
\end{split}\end{align}

Then we get 
\begin{align}\begin{split}\nonumber
    M_{1,1}=\frac{\del s_{n+3}}{\del \phi_1}\bigg(\frac{\del s_1}{\del s_0}\bigg)^n\bigg[\frac{\del s_2}{\del \phi_1}\det(df(x_0))\bigg],
\end{split}\end{align}

And we get for $3\leq l \leq n+2$
\begin{align}\begin{split}\nonumber
    &M_{1,l}=\frac{\del s_{n+3}}{\del \phi_l}\bigg(\frac{\del s_l}{\del s_0}\bigg)^n\bigg[\frac{\del s_2}{\del \phi_1}\det(df(x_0))\bigg]\\
    &\qquad - \frac{\del s_{n+3}}{\del \phi_l}\bigg(\frac{\del s_l}{\del s_0}\bigg)^{n-1}\bigg[\frac{\del s_l}{\del \phi_1}\det(df(x_0))\bigg]\frac{\del s_2}{\del s_0}\\
    &= \frac{\del s_{n+3}}{\del \phi_l}\bigg(\frac{\del s_l}{\del s_0}\bigg)^{n-1}\det(df(x_0))\bigg[\frac{\del s_l}{\del s_0}\frac{\del s_2}{\del \phi_1}-\frac{\del s_l}{\del \phi_l}\frac{\del s_2}{\del s_0}\bigg]\\
    &=-\frac{\del s_{n+3}}{\del \phi_l}\bigg(\frac{\del s_l}{\del s_0}\bigg)^{n-1}\bigg[\det(df^2(x_0))\frac{\del s_l}{\del \phi_2}\frac{\del s_1}{\del s_0}\bigg].
\end{split}\end{align}

Now we perform the following row operations:
\begin{align}\begin{split}\nonumber
    &R_1 \rightarrow R_1\bigg[\frac{\del s_3}{\del \phi_1}\det(df(x_0))\bigg]\bigg[\frac{\del s_2}{\del s_0}\bigg],\\
    &R_1 \rightarrow R_1 + R_3\frac{\del s_3}{\del s_0}\bigg[\frac{\del s_2}{\del s_0}\bigg],\\
    &R_2 \rightarrow R_2\bigg[\frac{\del s_3}{\del \phi_2}\frac{\del s_1}{\del s_0}\det(df^2(x_0))\bigg],\\
    &R_2 \rightarrow R_2 - R_3\frac{\del s_3}{\del s_0}.
\end{split}\end{align}

So we obtain
\begin{align}\begin{split}\nonumber
    &M_{1,1} = \frac{\del s_{n+3}}{\del \phi_1}\bigg(\frac{\del s_1}{\del s_0}\bigg)^n\bigg[\frac{\del s_2}{\del \phi_1}\det(df(x_0))\bigg]
    \bigg[
    \frac{\del s_3}{\del \phi_1}\det(df(x_0))
    \bigg]\bigg[\frac{\del s_2}{\del s_0}\bigg],\\
    &M_{2,2} = \frac{\del s_{n+3}}{\del \phi_2}\bigg(\frac{\del s_2}{\del s_0}\bigg)^{n-1}\bigg[\frac{\del s_2}{\del \phi_1}\det(df(x_0))\bigg]
    \bigg[
    \frac{\del s_3}{\del \phi_2}\frac{\del s_1}{\del s_0}\det(df^2(x_0))
    \bigg],
\end{split}\end{align}

And we get for $4 \leq l \leq n+2$
\begin{align}\begin{split}\nonumber
    M_{1,l}=\frac{\del s_{n+3}}{\del \phi_l}\bigg(\frac{\del s_l}{\del s_0}\bigg)^{n-2}\bigg[\frac{\del s_1}{\del s_0}\frac{\del s_l}{\del \phi_2}\det(df^2(x_0))\bigg]
    \bigg[
    \frac{\del s_1}{\del s_0}\frac{\del s_2}{\del s_0}\frac{\del s_l}{\del \phi_3}\det(df^3(x_0))
    \bigg],
    \\
    M_{2,l}=-\frac{\del s_{n+3}}{\del \phi_l}\bigg(\frac{\del s_l}{\del s_0}\bigg)^{n-2}\bigg[\frac{\del s_l}{\del \phi_1}\det(df(x_0))\bigg]
    \bigg[
    \frac{\del s_1}{\del s_0}\frac{\del s_2}{\del s_0}\frac{\del s_l}{\del \phi_3}\det(df^3(x_0))
    \bigg].
\end{split}\end{align}
Now we eliminate one more.  We do the following row operations:
\begin{align}\begin{split}\nonumber
     &R_1 \rightarrow R_1 \bigg[\frac{\del s_4}{\del \phi_1}\det(df(x_0))\bigg]\bigg[\frac{\del s_2}{\del s_0}\frac{\del s_3}{\del s_0}\bigg],\\
     &R_1\rightarrow R_1 - R_4\frac{\del s_4}{\del s_0}\bigg[\frac{\del s_2}{\del s_0}\frac{\del s_3}{\del s_0}\bigg],\\
     &R_2 \rightarrow R_2\bigg[\frac{\del s_1}{\del s_0}\frac{\del s_4}{\del \phi_2}\det(df^2(x_0))\bigg]\bigg[\frac{\del s_3}{\del s_0}\bigg],\\
     &R_2 \rightarrow R_2 + R_4\frac{\del s_4}{\del s_0}\bigg[\frac{\del s_3}{\del s_0}\bigg],\\
     &R_3 \rightarrow R_3\bigg[\frac{\del s_1}{\del s_0}\frac{\del s_2}{\del s_0}\frac{\del s_4}{\del \phi_3}\det(df^3(x_0))\bigg],\\
     &R_3 \rightarrow R_3 - R_4\frac{\del s_4}{\del s_0}.
\end{split}\end{align}

This yields
\begin{align}\begin{split}\nonumber
    &M_{1,1} = \frac{\del s_{n+3}}{\del \phi_1}\bigg(\frac{\del s_1}{\del s_0}\bigg)^n\bigg[\frac{\del s_2}{\del \phi_1}\det(df(x_0))\bigg]
    \bigg[
    \frac{\del s_3}{\del \phi_1}\det(df(x_0))
    \bigg]\times\\
    &\qquad 
    \bigg[
    \frac{\del s_4}{\del \phi_1}\det(df(x_0))
    \bigg]
    \bigg[
    \bigg(\frac{\del s_2}{\del s_0}\bigg)^2
    \frac{\del s_3}{\del s_0}
    \bigg],
    \\
    &M_{2,2}=\frac{\del s_{n+3}}{\del \phi_2}\bigg(\frac{\del s_2}{\del s_0}\bigg)^{n-1}\bigg[\frac{\del s_2}{\del \phi_1}\det(df(x_0))\bigg]
    \bigg[
    \frac{\del s_3}{\del \phi_2}\frac{\del s_1}{\del s_0}\det(df^2(x_0))
    \bigg]\times\\
    &\qquad \bigg[\frac{\del s_1}{\del s_0}\frac{\del s_4}{\del \phi_2}\det(df^2(x_0))\bigg]
    \bigg[
    \frac{\del s_3}{\del s_0}
    \bigg],
    \\
    &M_{3,3} = \frac{\del s_{n+3}}{\del \phi_3}\bigg(\frac{\del s_3}{\del s_0}\bigg)^{n-2}\bigg[\frac{\del s_3}{\del \phi_1}\det(df(x_0))\bigg]\bigg[\frac{\del s_1}{\del s_0}\frac{\del s_3}{\del \phi_2}\det(df^2(x_0))\bigg]\times\\
    &\qquad \bigg[\frac{\del s_1}{\del s_0}\frac{\del s_2}{\del s_0}\frac{\del s_4}{\del \phi_3}\det(df^3(x_0))\bigg],
\end{split}\end{align}
And for $5 \leq l \leq n+2$ we have
\begin{align}\begin{split}\nonumber
    &M_{1,l} = -\frac{\del s_{n+3}}{\del \phi_l}\bigg(\frac{\del s_l}{\del s_0}\bigg)^{n-3}\bigg[\frac{\del s_1}{\del s_0}\frac{\del s_l}{\del \phi_2}\det(df^2(x_0))\bigg]\times\\
    &\qquad \bigg[\frac{\del s_1}{\del s_0}\frac{\del s_2}{\del s_0}\frac{\del s_l}{\del \phi_3}\det(df^3(x_0))\bigg]\bigg[\frac{\del s_1}{\del s_0}\frac{\del s_2}{\del s_0}\frac{\del s_3}{\del s_0}\frac{\del s_l}{\del \phi_4}\det(df^4(x_0))\bigg],
    \\
    &M_{2,l} = \frac{\del s_{n+3}}{\del \phi_l}\bigg(\frac{\del s_l}{\del s_0}\bigg)^{n-3}\bigg[\frac{\del s_l}{\del \phi_1}\det(df(x_0))\bigg]\times\\
    &\qquad \bigg[\frac{\del s_1}{\del s_0}\frac{\del s_2}{\del s_0}\frac{\del s_l}{\del \phi_3}\det(df^3(x_0))\bigg]\bigg[\frac{\del s_1}{\del s_0}\frac{\del s_2}{\del s_0}\frac{\del s_3}{\del s_0}\frac{\del s_l}{\del \phi_4}\det(df^4(x_0))\bigg],
    \\
    &M_{3,l} = -\frac{\del s_{n+3}}{\del \phi_l}\bigg(\frac{\del s_l}{\del s_0}\bigg)^{n-3}\bigg[\frac{\del s_l}{\del \phi_1}\det(df(x_0))\bigg]\times\\
    &\qquad \bigg[\frac{\del s_1}{\del s_0}\frac{\del s_l}{\del \phi_2}\det(df^2(x_0))\bigg]\bigg[\frac{\del s_1}{\del s_0}\frac{\del s_2}{\del s_0}\frac{\del s_3}{\del s_0}\frac{\del s_l}{\del \phi_4}\det(df^4(x_0))\bigg].
\end{split}\end{align}
Proceeding inductively can get the final result.  We obtain for $1 \leq k \leq n+1$
\begin{align}\begin{split}\nonumber
    &M_{k,k}=\frac{\del s_{n+3}}{\del \phi_k}
    \bigg\{\bigg[\frac{\del s_k}{\del \phi_1}\det(df(x_0))\bigg]
    ...\\
    &\qquad
    \bigg[\frac{\del s_1}{\del s_0}...\frac{\del s_{k-2}}{\del s_0}\frac{\del s_k}{\del \phi_{k-1}}\det(df^{k-1}(x_0))\bigg]\bigg\}\times\\
    &\qquad \bigg\{\bigg[\frac{\del s_1}{\del s_0}...\frac{\del s_{k-1}}{\del s_0}\frac{\del s_{k+1}}{\del \phi_k}\det(df^k(x_0))\bigg]...\\
    &\qquad \bigg[\frac{\del s_1}{\del s_0}...\frac{\del s_{k-1}}{\del s_0}\frac{\del s_{n+1}}{\del \phi_k}\det(df^k(x_0))\bigg]\bigg\}\times\\
    &\qquad \bigg[
    \bigg(\frac{\del s_k}{\del s_0}\bigg)^{n-(k-1)}
    \bigg(\frac{\del s_{k+1}}{\del s_0}\bigg)^{n-k}\bigg(\frac{\del s_{k+2}}{\del s_0}\bigg)^{n-(k+1)}...\bigg(\frac{\del s_n}{\del s_0}\bigg)\bigg],
    \\
    &M_{k,n+2}=(-1)^{n+k-1}\frac{\del s_{n+3}}{\del \phi_{n+2}}\bigg[\frac{\del s_{n+2}}{\del \phi_1}\det(df(x_0))\bigg]...\\
    &\qquad ...\bigg[\frac{\del s_1}{\del s_0}...\frac{\del s_{k-2}}{\del s_0}\frac{\del s_{n+2}}{\del \phi_{k-1}}\det(df^{k-1}(x_0))\bigg]\times\\
    &\qquad \bigg[\frac{\del s_1}{\del s_0}...\frac{\del s_k}{\del s_0}\frac{\del s_{n+2}}{\del \phi_{k+1}}\det(df^{k+1}(x_0))\bigg]...\\
    &\qquad ...\bigg[\frac{\del s_1}{\del s_0}...\frac{\del s_n}{\del s_0}\frac{\del s_{n+2}}{\del \phi_{n+1}}\det(df^{n+1}(x_0))\bigg].
\end{split}\end{align}
The second equation here will be used in the next result.
\end{proof}

Now we proceed to perform column operations to eliminate column $n+2$, and then we see that this gives in index $M_{n+2,n+2}$ the desired result claimed in the previous lemma.

First we note that if we define $X_k$ and $Y_k$ by
\begin{align}\begin{split}\nonumber
    &M_{k,k}=\text{LCM}(M_{kk},M_{k,n+2})X_k,\\
    &M_{k,n+2}=\text{LCM}(M_{kk},M_{k,n+2})Y_k,
\end{split}\end{align}
where $LCM$ is the least common multiple when we consider the terms symbolically (not numerically). Then we have
\begin{align}\begin{split}\nonumber
    &X_k = \frac{\del s_{n+3}}{\del \phi_k}\bigg[\frac{\del s_k}{\del \phi_1}...\frac{\del s_k}{\del \phi_{k-1}}\bigg]\bigg[\frac{\del s_{k+1}}{\del \phi_k}...\frac{\del s_{n+1}}{\del \phi_k}\bigg],\\
    &Y_k = (-1)^{n-k+1}\frac{\del s_{n+3}}{\del \phi_{n+2}}\bigg[\frac{\del s_{n+2}}{\del \phi_1}...\frac{\del s_{n+2}}{\del \phi_{k-1}}\bigg]\bigg[\frac{\del s_{n+2}}{\del \phi_{k+1}}...\frac{\del s_{n+2}}{\del \phi_{n+1}}\bigg]\times\\
    &\qquad \det(df(x_k))...\det(df^{n-k+1}(x_k))
\end{split}\end{align}

We perform the following row operations:
\begin{align}\begin{split}\nonumber
    &C_{n+2} \rightarrow C_{n+2}X_1\\
    &C_{n+2} \rightarrow C_{n+2}-C_1Y_1\\
    &C_{n+2} \rightarrow C_{n+2}X_2\\
    &C_{n+2} \rightarrow C_{n+2}-C_2Y_2X_1\\
    &...\\
    &C_{n+2} \rightarrow C_{n+2}X_{n+1}\\
    &C_{n+2} \rightarrow C_{n+2}-C_{n+1}Y_{n+1}X_n...X_1.
\end{split}\end{align}

This yields in $M_{n+2,n+2}$ the following:

\begin{align}\begin{split}\nonumber
    &M_{n+2,n+2}=X_1...X_{n+1}\bigg\{\frac{\del \phi_{n+3}}{\del \phi_{n+2}}\bigg(\frac{\del s_{n+2}}{\del \phi_0}\bigg)^n -\sum_{i=1}^{n+1}\frac{\del \phi_{n+3}}{\del \phi_i}\bigg(\frac{\del s_i}{\del \phi_0}\bigg)^n\frac{Y_i}{X_i}\bigg\}.
\end{split}\end{align}
We now claim the following.
\begin{lemma}
The above equation is equal to
\begin{align}\begin{split}\nonumber
    &\det(df(x_{n+2}))\bigg(\frac{\del s_{n+3}}{\del \phi_0}\bigg)^n\bigg[\frac{\del s_{n+2}}{\del \phi_1}...\frac{\del s_{n+2}}{\del \phi_{n+1}}\bigg]\times\\
    &\qquad\bigg[\bigg(\frac{\del s_{n+1}}{\del \phi_1}...\frac{\del s_{n+1}}{\del \phi_n}\bigg)^2\bigg(\frac{\del s_{n}}{\del \phi_1}...\frac{\del s_{n}}{\del \phi_{n-1}}\bigg)^2...\bigg(\frac{\del s_2}{\del \phi_1}\bigg)^2\bigg],
\end{split}\end{align}
which is non-zero.
\end{lemma}
\begin{proof}
Now
\begin{align}\begin{split}\nonumber
    &X_1...X_{n+1} = \bigg[\frac{\del s_{n+3}}{\del \phi_1}...\frac{\del s_{n+3}}{\del \phi_{n+1}}\bigg]\times\\
    &\qquad\bigg[\bigg(\frac{\del s_{n+1}}{\del \phi_1}...\frac{\del s_{n+1}}{\del \phi_n}\bigg)^2\bigg(\frac{\del s_{n}}{\del \phi_1}...\frac{\del s_{n}}{\del \phi_{n-1}}\bigg)^2...\bigg(\frac{\del s_2}{\del \phi_1}\bigg)^2\bigg].
\end{split}\end{align}

We begin by dividing both sides of the supposed equality by $X_1...X_{n+1}$.  We obain on one side
\begin{align}\begin{split}\nonumber
    \det(df(x_{n+2}))\bigg(\frac{\del s_{n+3}}{\del \phi_0}\bigg)^n\frac{\bigg[\frac{\del s_{n+2}}{\del \phi_1}...\frac{\del s_{n+2}}{\del \phi_{n+1}}\bigg]}{\bigg[\frac{\del s_{n+3}}{\del \phi_1}...\frac{\del s_{n+3}}{\del \phi_{n+1}}\bigg]}.
\end{split}\end{align}

Dividing by $\bigg[\frac{\del s_{n+2}}{\del \phi_1}...\frac{\del s_{n+2}}{\del \phi_{n+1}}\bigg]$ and multiplying by $\bigg[\frac{\del s_{n+3}}{\del \phi_1}...\frac{\del s_{n+3}}{\del \phi_{n+1}}\bigg]$, we obtain on one side
\begin{align}\begin{split}\nonumber
    \det(df(x_{n+2}))\bigg(\frac{\del s_{n+3}}{\del \phi_0}\bigg)^n,
\end{split}\end{align}
and on the other
\begin{align}\begin{split}\label{eqn_rhs_of_n_th_inductive_step}
    &\frac{\del \phi_{n+3}}{\del \phi_{n+2}}\bigg(\frac{\del s_{n+2}}{\del \phi_0}\bigg)^n
    \,
    \frac{\bigg[\frac{\del s_{n+3}}{\del \phi_1}...\frac{\del s_{n+3}}{\del \phi_{n+1}}\bigg]}{\bigg[\frac{\del s_{n+2}}{\del \phi_1}...\frac{\del s_{n+2}}{\del \phi_{n+1}}\bigg]}
    \,
    \\
    &-\sum_{i=1}^{n+1}\frac{\del \phi_{n+3}}{\del \phi_i}\bigg(\frac{\del s_i}{\del \phi_0}\bigg)^n\frac{Y_i}{X_i}
    \,
    \frac{\bigg[\frac{\del s_{n+3}}{\del \phi_1}...\frac{\del s_{n+3}}{\del \phi_{n+1}}\bigg]}{\bigg[\frac{\del s_{n+2}}{\del \phi_1}...\frac{\del s_{n+2}}{\del \phi_{n+1}}\bigg]}
    \,.
\end{split}\end{align}
Now each term in this sum gives
\begin{align}\begin{split}\nonumber
    &(-1)^{n-1+1}\frac{\del \phi_{n+3}}{\del \phi_i}\bigg(\frac{\del s_i}{\del \phi_0}\bigg)^n\times
    \\
    &\frac{\frac{\del s_{n+3}}{\del \phi_{n+2}}\bigg[\frac{\del s_{n+2}}{\del \phi_1}...\frac{\del s_{n+2}}{\del \phi_{i-1}}\bigg]\bigg[\frac{\del s_{n+2}}{\del \phi_{i+1}}...\frac{\del s_{n+2}}{\del \phi_{n+1}}\bigg]}{\frac{\del s_{n+3}}{\del \phi_i}\bigg[\frac{\del s_{i}}{\del \phi_1}...\frac{\del s_i}{\del \phi_{i-1}}\bigg]\bigg[\frac{\del s_{i+1}}{\del \phi_i}...\frac{\del s_{n+1}}{\del \phi_i}\bigg]}\times
    \\
    &\det(df(x_i))...\det(df^{n-i+1}(x_i))\times
    \\
    &\frac{\bigg[\frac{\del s_{n+3}}{\del \phi_1}...\frac{\del s_{n+3}}{\del \phi_{n+1}}\bigg]}{\bigg[\frac{\del s_{n+2}}{\del \phi_1}...\frac{\del s_{n+2}}{\del \phi_{n+1}}\bigg]},
\end{split}\end{align}
which simplifies to
\begin{align}\begin{split}\nonumber
    &(-1)^{n-i+1}\frac{\del \phi_{n+3}}{\del \phi_i}\bigg(\frac{\del s_i}{\del \phi_0}\bigg)^n\times
    \\
    &\frac{\frac{\del s_{n+3}}{\del \phi_{n+2}}
    \,
    \bigg[\frac{\del s_{n+3}}{\del \phi_1}...\frac{\del s_{n+3}}{\del \phi_{n+1}}\bigg]\det(df(x_i))...\det(df^{n-i+1}(x_i))}{\frac{\del s_{n+3}}{\del \phi_i}\frac{\del s_{n+2}}{\del \phi_i}\bigg[\frac{\del s_{i}}{\del \phi_1}...\frac{\del s_i}{\del \phi_{i-1}}\bigg]\bigg[\frac{\del s_{i+1}}{\del \phi_i}...\frac{\del s_{n+1}}{\del \phi_i}\bigg]}.
\end{split}\end{align}
Now, we prove the equality by induction.  Suppose it is true for $n-1$.  Then, by relabelling points $x_1,...,x_{n+2}$ to $x_2,...,x_{n+3}$ and multiplying both sides by $\frac{\del s_{n+3}}{\del \phi_0}$, we obtain the equality
\begin{align}\begin{split}\label{eqn_rhs_of_n_minus_1_th_unductive_step}
    &\det(df(x_{n+2}))\bigg(\frac{\del s_{n+3}}{\del \phi_0}\bigg)^n\\
    &=\frac{\del \phi_{n+3}}{\del \phi_{n+2}}\bigg(\frac{\del s_{n+2}}{\del \phi_0}\bigg)^{n-1}\frac{\del s_{n+3}}{\del \phi_0}\frac{\del s_{n+2}}{\del \phi_1}\frac{\bigg[\frac{\del s_{n+3}}{\del \phi_2}...\frac{\del s_{n+3}}{\del \phi_{n+1}}\bigg]}{\bigg[\frac{\del s_{n+2}}{\del \phi_1}....\frac{\del s_{n+2}}{\del \phi_{n+1}}\bigg]}
    \\
    &-\sum_{i=1}^{n}(-1)^{n-i}\frac{\del \phi_{n+3}}{\del \phi_{i+1}}\bigg(\frac{\del s_{i+1}}{\del \phi_0}\bigg)^{n-1}\frac{\del s_{i+1}}{\del \phi_1}\frac{\del s_{n+3}}{\del \phi_0}\times
    \\
    &\frac{\frac{\del s_{n+3}}{\del \phi_{n+2}}
    \,
    \bigg[\frac{\del s_{n+3}}{\del \phi_2}...\frac{\del s_{n+3}}{\del \phi_{n+1}}\bigg]\det(df(x_{i+1}))...\det(df^{n-i}(x_{i+1}))
    }{\frac{\del s_{n+3}}{\del \phi_{i+1}}\frac{\del s_{n+2}}{\del \phi_{i+1}}\bigg[\frac{\del s_{i+1}}{\del \phi_1}...\frac{\del s_{i+1}}{\del \phi_i}\bigg]\bigg[\frac{\del s_{i+2}}{\del \phi_{i+1}}...\frac{\del s_{n+1}}{\del \phi_{i+1}}\bigg]}.
\end{split}\end{align}
We now show that the right hand side of this equation is equal to \eqref{eqn_rhs_of_n_th_inductive_step}.  To do this we take each $\frac{\del s_{n+3}}{\del \phi_0}$ in every term and split it up as
\begin{align*}\begin{split}\nonumber
    \frac{\del s_{n+3}}{\del \phi_0} = \frac{\del s_{n+3}}{\del s_1}\frac{\del s_1}{\del \phi_0}+\frac{\del s_{n+3}}{\del \phi_1}\frac{\del \phi_1}{\del \phi_0}.
\end{split}\end{align*}
Now we take a single $\frac{\del s_i}{\del \phi_0}$ in each term of \eqref{eqn_rhs_of_n_th_inductive_step} for $i > 1$ and split it up as
\begin{align}\begin{split}\nonumber
    \frac{\del s_i}{\del \phi_0} = \frac{\del s_i}{\del s_1}\frac{\del s_1}{\del \phi_0} + \frac{\del s_i}{\del \phi_1}\frac{\del \phi_1}{\del \phi_0}.
\end{split}\end{align}
Now we subtract \eqref{eqn_rhs_of_n_th_inductive_step} from \eqref{eqn_rhs_of_n_minus_1_th_unductive_step}.  Notice that the terms with $\frac{\del s_{n+3}}{\del \phi_1}\frac{\del \phi_1}{\del \phi_0}$ and the terms with $\frac{\del s_i}{\del \phi_1}\frac{\del \phi_1}{\del \phi_0}$ cancel out, and we are left with the difference being
\begin{align*}\begin{split}\nonumber
    &\frac{\del \phi_{n+3}}{\del \phi_{n+2}}\bigg(\frac{\del s_{n+2}}{\del \phi_0}\bigg)^{n-1}\frac{\bigg[\frac{\del s_{n+3}}{\del \phi_2}...\frac{\del s_{n+3}}{\del \phi_{n+1}}\bigg]}{\bigg[\frac{\del s_{n+2}}{\del \phi_1}...\frac{\del s_{n+2}}{\del \phi_{n+1}}\bigg]}\bigg[\frac{\del s_{n+2}}{\del \phi_1}\frac{\del s_{n+3}}{\del s_1}-\frac{\del s_{n+3}}{\del \phi_1}\frac{\del s_{n+2}}{\del s_1}\bigg]\frac{\del s_1}{\del \phi_0}
    \\
    &-\sum_{i=1}^n\bigg\{(-1)^{n-i}\frac{\del \phi_{n+3}}{\del \phi_{i+1}}\bigg(\frac{\del s_{i+1}}{\del \phi_0}\bigg)^{n-1}\times
    \\
    &\qquad\frac{
    \,
    \frac{\del s_{n+3}}{\del \phi_{n+2}}
    \bigg[
    \frac{\del s_{n+3}}{\del \phi_2}...\frac{\del s_{n+3}}{\del \phi_{n+1}}
    \bigg]
    \det(df(x_{i+1}))...\det(df^{n-i}(x_{i+1}))
    }{
    \,
    \frac{\del s_{n+3}}{\del \phi_{i+1}}\frac{\del s_{n+2}}{\del \phi_{i+1}}
    \bigg[\frac{\del s_{i+1}}{\del \phi_1}...\frac{\del s_{i+1}}{\del \phi_i}\bigg]\bigg[\frac{\del s_{i+2}}{\del \phi_{i+1}}...\frac{\del s_{n+1}}{\del \phi_{i+1}}\bigg]
    }\times
    \\
    &\qquad\bigg[\frac{\del s_{i+1}}{\del \phi_1}\frac{\del s_{n+3}}{\del s_1} - \frac{\del s_{i+1}}{\del s_1}\frac{\del s_{n+3}}{\del \phi_1}\bigg]\frac{\del s_1}{\del \phi_0}\bigg\}
    \\
    &\qquad - \bigg\{(-1)^n\frac{\del \phi_{n+3}}{\del \phi_1}\bigg(\frac{\del s_1}{\del \phi_0}\bigg)^n\times
    \\
    &\qquad\qquad \frac{\frac{\del s_{n+3}}{\del \phi_{n+2}}
    \,
    \bigg[\frac{\del s_{n+3}}{\del \phi_2}...\frac{\del s_{n+3}}{\del \phi_{n+1}}\bigg]\det(df(x_1))...\det(df^n(x_1))}{\frac{\del s_{n+2}}{\del \phi_1}
    \bigg[
    \frac{\del s_2}{\del \phi_1}...\frac{\del s_{n+1}}{\del \phi_1}
    \bigg]}\bigg\}.
\end{split}\end{align*}
We now show that this difference is equal to $0$.  First, we divide by $\frac{\del s_1}{\del \phi_0}$, and by $\frac{\del s_{n+3}}{\del \phi_2}...\frac{\del s_{n+3}}{\del \phi_{n+1}}$.  We also note that 
\begin{align}\begin{split}\nonumber
    \bigg[\frac{\del s_{i+1}}{\del \phi_1}\frac{\del s_{n+3}}{\del s_1}-\frac{\del s_{i+1}}{\del s_1}\frac{\del s_{n+3}}{\del \phi_1}\bigg] = -\frac{\del s_{n+3}}{\del \phi_{i+1}}\det(df^i(x_1)).
\end{split}\end{align}
So we have
\begin{align}\begin{split}\nonumber
    &-\frac{\del \phi_{n+3}}{\del \phi_{n+2}}\bigg(\frac{\del s_{n+2}}{\del \phi_0}\bigg)^{n-1}
    \frac{
    \,
    \frac{\del s_{n+3}}{\del \phi_{n+2}}\det(df^{n+1}(x_1))
    }{
    \,
    \bigg[\frac{\del s_{n+2}}{\del \phi_1}...\frac{\del s_{n+2}}{\del \phi_{n+1}}\bigg]
    }\\
    &+\sum_{i=1}^n\bigg\{(-1)^{n-i}\frac{\del \phi_{n+3}}{\del \phi_{i+1}}\bigg(\frac{\del s_{i+1}}{\del \phi_0}\bigg)^{n-1}\times\\
    &\qquad \frac{
    \,
    \frac{\del s_{n+3}}{\del \phi_{n+2}}\det(df(x_{i+1}))...\det(df^{n-i}(x_{i+1}))\det(df^i(x_1))
    }{
    \,
    \frac{\del s_{n+2}}{\del \phi_{i+1}}
    \bigg[\frac{\del s_{i+1}}{\del \phi_1}...\frac{\del s_{i+1}}{\del \phi_i}\bigg]\bigg[\frac{\del s_{i+2}}{\del \phi_{i+1}}...\frac{\del s_{n+1}}{\del \phi_{i+1}}\bigg]
    }\bigg\}
    \\
    &\qquad-\bigg\{(-1)^n\frac{\del \phi_{n+3}}{\del \phi_1}\bigg(\frac{\del s_1}{\del \phi_0}\bigg)^{n-1}\times
    \\
    &\qquad\qquad
    \frac{
    \,
    \frac{\del s_{n+3}}{\del \phi_{n+2}}\det(df(x_1))...\det(df^n(x_1))
    }{
    \,
    \frac{\del s_{n+2}}{\del \phi_1}\bigg[\frac{\del s_2}{\del \phi_1}...\frac{\del s_{n+1}}{\del \phi_1}\bigg]
    }\bigg\}.
\end{split}\end{align}
Now we can divide each term by $\frac{\del s_{n+3}}{\del \phi_{n+2}}$.  If one examines the determinants, they will see that we can also divide by $\det(df^n(x_1))$.  Doing so yields
\begin{align}\begin{split}\nonumber
    &-\frac{\del \phi_{n+3}}{\del \phi_{n+2}}\bigg(\frac{\del s_{n+2}}{\del \phi_0}\bigg)^{n-1}
    \frac{
    \,
    \det(df(x_{n+1}))
    }{
    \,
    \bigg[\frac{\del s_{n+2}}{\del \phi_1}...\frac{\del s_{n+2}}{\del \phi_{n+1}}\bigg]
    }\\
    &+\sum_{i=0}^{n-1}\bigg\{(-1)^{n-i}\frac{\del \phi_{n+3}}{\del \phi_{i+1}}\bigg(\frac{\del s_{i+1}}{\del \phi_0}\bigg)^{n-1}\times\\
    &\qquad \frac{
    \,
    \det(df(x_{i+1}))...\det(df^{n-i-1}(x_{i+1}))
    }{
    \,
    \frac{\del s_{n+2}}{\del \phi_{i+1}}
    \bigg[\frac{\del s_{i+1}}{\del \phi_1}...\frac{\del s_{i+1}}{\del \phi_i}\bigg]\bigg[\frac{\del s_{i+2}}{\del \phi_{i+1}}...\frac{\del s_{n+1}}{\del \phi_{i+1}}\bigg]
    }\bigg\}
    \\
    &+\frac{\del \phi_{n+3}}{\del \phi_{n+1}}\bigg(\frac{\del s_{n+1}}{\del \phi_0}\bigg)^{n-1}
    \frac{1}
    {
    \,
    \frac{\del s_{n+2}}{\del \phi_{n+1}}\bigg[\frac{\del s_{n+1}}{\del \phi_1}...\frac{\del s_{n+1}}{\del \phi_n}\bigg]
    }
\end{split}\end{align}
Now, since we are assuming the equality holds for $n-1$, we note this says
\begingroup
\begin{align}\begin{split}\nonumber
    &\det(df(x_{n+1}))\bigg(\frac{\del s_{n+2}}{\del \phi_0}\bigg)^{n-1}
    \\
    &=\frac{\del \phi_{n+2}}{\del \phi_{n+1}}\bigg(\frac{\del s_{n+1}}{\del \phi_0}\bigg)^{n-1}
    \,
    \frac{\bigg[\frac{\del s_{n+2}}{\del \phi_1}...\frac{\del s_{n+2}}{\del \phi_n}\bigg]}{\bigg[\frac{\del s_{n+1}}{\del \phi_1}...\frac{\del s_{n+1}}{\del \phi_n}\bigg]}
    \,
    \\
    &-\sum_{i=1}^{n}(-1)^{n-i}\frac{\del \phi_{n+2}}{\del \phi_i}\bigg(\frac{\del s_i}{\del \phi_0}\bigg)^{n-1}\times
    \\
    &\frac{\frac{\del s_{n+2}}{\del \phi_{n+1}}
    \,
    \bigg[\frac{\del s_{n+2}}{\del \phi_1}...\frac{\del s_{n+2}}{\del \phi_{n}}\bigg]\det(df(x_i))...\det(df^{n-i}(x_i))}{\frac{\del s_{n+2}}{\del \phi_i}\frac{\del s_{n+1}}{\del \phi_i}\bigg[\frac{\del s_{i}}{\del \phi_1}...\frac{\del s_i}{\del \phi_{i-1}}\bigg]\bigg[\frac{\del s_{i+1}}{\del \phi_i}...\frac{\del s_{n}}{\del \phi_i}\bigg]}
    \\
    \,
    &=\frac{\del \phi_{n+2}}{\del \phi_{n+1}}\bigg(\frac{\del s_{n+1}}{\del \phi_0}\bigg)^{n-1}
    \,
    \frac{\bigg[\frac{\del s_{n+2}}{\del \phi_1}...\frac{\del s_{n+2}}{\del \phi_n}\bigg]}{\bigg[\frac{\del s_{n+1}}{\del \phi_1}...\frac{\del s_{n+1}}{\del \phi_n}\bigg]}
    \,
    \\
    &+\sum_{i=0}^{n-1}(-1)^{n-i}\frac{\del \phi_{n+2}}{\del \phi_{i+1}}\bigg(\frac{\del s_{i+1}}{\del \phi_0}\bigg)^{n-1}\times
    \\
    &\frac{\frac{\del s_{n+2}}{\del \phi_{n+1}}
    \,
    \bigg[\frac{\del s_{n+2}}{\del \phi_1}...\frac{\del s_{n+2}}{\del \phi_{n}}\bigg]\det(df(x_{i+1}))...\det(df^{n-i-1}(x_{i+1}))}{\frac{\del s_{n+2}}{\del \phi_{i+1}}\frac{\del s_{n+1}}{\del \phi_{i+1}}\bigg[\frac{\del s_{i+1}}{\del \phi_1}...\frac{\del s_{i+1}}{\del \phi_i}\bigg]\bigg[\frac{\del s_{i+2}}{\del \phi_{i+1}}...\frac{\del s_{n}}{\del \phi_{i+1}}\bigg]}.
\end{split}\end{align}
\endgroup
Substituting this into the previous equation yields
\begin{align*}\nonumber
    &-\frac{\del \phi_{n+3}}{\del \phi_{n+2}}
    \,
    \frac{
    \,
    \frac{\del \phi_{n+2}}{\del \phi_{n+1}}\bigg(\frac{\del s_{n+1}}{\del \phi_0}\bigg)^{n-1}
    }{
    \,
    \frac{\del s_{n+2}}{\del \phi_{n+1}}\bigg[\frac{\del s_{n+1}}{\del \phi_1}...\frac{\del s_{n+1}}{\del \phi_n}\bigg]}\\
    &+\sum_{i=0}^{n-1}\Bigg\{-\frac{\del \phi_{n+3}}{\del \phi_{n+2}}\bigg\{(-1)^{n-i}\frac{\del \phi_{n+2}}{\del \phi_{i+1}}\bigg(\frac{\del s_{i+1}}{\del \phi_0}\bigg)^{n-1}\times
    \\
    &\frac{\det(df(x_{i+1}))...\det(df^{n-i-1}(x_{i+1}))}{\frac{\del s_{n+2}}{\del \phi_{i+1}}\frac{\del s_{n+1}}{\del \phi_{i+1}}\bigg[\frac{\del s_{i+1}}{\del \phi_1}...\frac{\del s_{i+1}}{\del \phi_i}\bigg]\bigg[\frac{\del s_{i+2}}{\del \phi_{i+1}}...\frac{\del s_{n}}{\del \phi_{i+1}}\bigg]}\bigg\}
    \\
    &+\bigg\{
    (-1)^{n-i}\frac{\del \phi_{n+3}}{\del \phi_{i+1}}\bigg(\frac{\del s_{i+1}}{\del \phi_0}\bigg)^{n-1}\times\\
    &\qquad \frac{
    \,
    \det(df(x_{i+1}))...\det(df^{n-i-1}(x_{i+1}))
    }{
    \,
    \frac{\del s_{n+2}}{\del \phi_{i+1}}
    \bigg[\frac{\del s_{i+1}}{\del \phi_1}...\frac{\del s_{i+1}}{\del \phi_i}\bigg]\bigg[\frac{\del s_{i+2}}{\del \phi_{i+1}}...\frac{\del s_{n+1}}{\del \phi_{i+1}}\bigg]
    }
    \bigg\}
    \Bigg\}
    \\
    &+\frac{\del \phi_{n+3}}{\del \phi_{n+1}}\bigg(\frac{\del s_{n+1}}{\del \phi_0}\bigg)^{n-1}
    \frac{1}
    {
    \,
    \frac{\del s_{n+2}}{\del \phi_{n+1}}\bigg[\frac{\del s_{n+1}}{\del \phi_1}...\frac{\del s_{n+1}}{\del \phi_n}\bigg]
    }.
\end{align*}
Then for each term, writing
\begin{align}\begin{split}\nonumber
    \frac{\del \phi_{n+3}}{\del \phi_{i+1}}=\frac{\del \phi_{n+3}}{\del s_{n+2}}\frac{\del s_{n+2}}{\del \phi_{i+1}}+\frac{\del \phi_{n+3}}{\del \phi_{n+2}}\frac{\del \phi_{n+2}}{\del \phi_{i+1}}
\end{split}\end{align}
we observe that all the terms with $\frac{\del \phi_{n+3}}{\del \phi_{n+2}}$ cancel out, and we are left with
\begin{align}\begin{split}\nonumber
    &-\frac{\del \phi_{n+3}}{\del s_{n+2}}\bigg(\frac{\del s_{n+1}}{\del \phi_0}\bigg)^{n-1}
    \frac{1}
    {
    \,
    \bigg[\frac{\del s_{n+1}}{\del \phi_1}...\frac{\del s_{n+1}}{\del \phi_n}\bigg]
    }
    \\
    &-
    \sum_{i=0}^{n-1}\Bigg\{(-1)^{n-i}\bigg(\frac{\del s_{i+1}}{\del \phi_0}\bigg)^{n-1}
    \times\\
    &\frac{
    \,
    \det(df(x_{i+1}))...\det(df^{n-i-1}(x_{i+1}))
    }{
    \,
    \frac{\del s_{n+1}}{\del \phi_{i+1}}\bigg[\frac{\del s_{i+1}}{\del \phi_1}...\frac{\del s_{i+1}}{\del \phi_i}\bigg]\bigg[\frac{\del s_{i+2}}{\del \phi_{i+1}}...\frac{\del s_n}{\del \phi_{i+1}}\bigg]
    }
    \frac{\del \phi_{n+3}}{\del s_{n+2}}\Bigg\}.
\end{split}\end{align}
Then, dividing by $\frac{\del \phi_{n+3}}{\del s_{n+2}}$ and multiplying by $\bigg[\frac{\del s_{n+1}}{\del \phi_1}...\frac{\del s_{n+1}}{\del \phi_n}\bigg]$, we have
\begin{align}\begin{split}\nonumber
    &-\bigg(\frac{\del s_{n+1}}{\del \phi_0}\bigg)^{n-1} 
    \\
    &- 
    \sum_{i=0}^{n-1}
    \Bigg\{
    (-1)^{n-i}
    \bigg(\frac{\del s_{i+1}}{\del \phi_0}\bigg)^{n-1}\times
    \\
    &\frac{
    \,
    \bigg[\frac{\del s_{n+1}}{\del \phi_1}...\frac{\del s_{n+1}}{\del \phi_i}\bigg]
    \bigg[
    \frac{\del s_{n+1}}{\del \phi_{i+2}}...\frac{\del s_{n+1}}{\del \phi_n}\bigg]\det(df(x_{i+1}))...\det(df^{n-i-1}(x_{i+1}))
    }
    {
    \,
    \bigg[\frac{\del s_{i+1}}{\del \phi_1}...\frac{\del s_{i+1}}{\del \phi_i}\bigg]
    \bigg[
    \frac{\del s_{i+2}}{\del \phi_{i+1}}...\frac{\del s_{n}}{\del \phi_{i+1}}\bigg]
    }
    \Bigg\}.
\end{split}\end{align}
We now prove that this is $0$ for all $n$ by induction.  Note that the base case of $n=2$ is easily verifiable.  For the rest, suppose it holds for $n-2$.  If one proceeds in the same fashion as before (breaking up the partials), they arrive at the conclusion that this is in fact equivalent to the step before being $0$.
\end{proof}
In fact, since all the calculations performed were simple algebra and did not use the explicit values of the partial derivatives in terms of the curvature of the boundary or anything to do with the fact that it is a billiard map, we can actually simply relabel the points $x_1,...,x_{n+2}$ to $x_{k_1},...,x_{k_{n+2}}$ where $k_1,...,k_{n+2}$ is any arbitrary increasing sequence of numbers, and obtain that by varying the curvature at any $n+2$ points, we may vary the $n^{th}$ differentials of $s_q$ and $\phi_q$ with $n+2$ degrees of freedom (which is the best one can do with billiard maps).

Thus, this combined with the fact that if you vary the curvature at points that are not consecutive points in the orbit then there are no $\eps^2$ terms, we have that we can vary the billiard map exactly up to the $n^{th}$ differential generically by varying in small neighborhoods of $n+2$ points on the boundary.

\section{Finding domains with a saddle point containing a homoclinic tangency}\label{section_Constructing_an_Orbit_with_Homoclinic_Tangency}
Our main result in this section is to prove Lemma \eqref{lemma_creating_quadratic_ht}.
\begin{proof}
We begin with a smooth $C^\infty$ domain $\Omega_0$.  By (\cite{PetkovStoyanov} Theorem 6.4.1) there exists arbitrarily close to $\Omega_0$ a domain $\Omega_1$ such that the eigenvalues of the differential of the billiard map composed with itself $q$ times at each $q$ periodic point are not any root of unity, for every $q$.  Then, due to a result of Aubry Mather theory which states that for a minimal orbit $O$ with rotation number $1/q$ there exists minimal orbits $O_\pm$ which accumulate to $O$ under forward and backward iterations (respectively), this implies that every minimal periodic orbit of $\Omega_1$ is hyperbolic.  

Next, we combine two results.  It is shown in \cite{Lazutkin} (Page 146 lemma 14.6) that for a sufficiently smooth boundary, one may find a coordinate system so that the billiard map takes the form
\begin{align}\begin{split}\nonumber
    f(\xi,\eta) = (\xi + \eta + A(\xi,\eta)\eta^N, \eta + B(\xi,\eta)\eta^{N+1}).
\end{split}\end{align}
This implies there exist rotational invariant KAM curves which accumulate onto the boundary.  For large $q$, then, we may find minimal orbits of rotation number $1/q$ trapped between these invariant curves. For instance, we may have the distance between these two KAM curves to be of order $1/q^{10}$. 

Now, by a result in \cite{Mather} we have that 
for these orbits, for two consecutive points $p_1=(x_1,y_1),p_2=(x_2,y_2) = f(p_1)$, there exist points $O_1=(\xi_1,\eta_1)$ and $O_2=(\xi_2,\eta_2)$ such that $x_1 < \xi_1,\xi_2 < x_2$ and the orbit of $O_1$ is a heteroclinic orbit from $p_1$ to $p_2$, and the orbit of $O_2$ is a heteroclinic orbit from $p_2$ to $p_1$. Thus, we have a heteroclinic cycle between the points $p_1$ and $p_2$ (see Figure \ref{fig_creating_ht})

\begin{figure}[h]
\centering
\includegraphics[width=12cm, height=4cm]{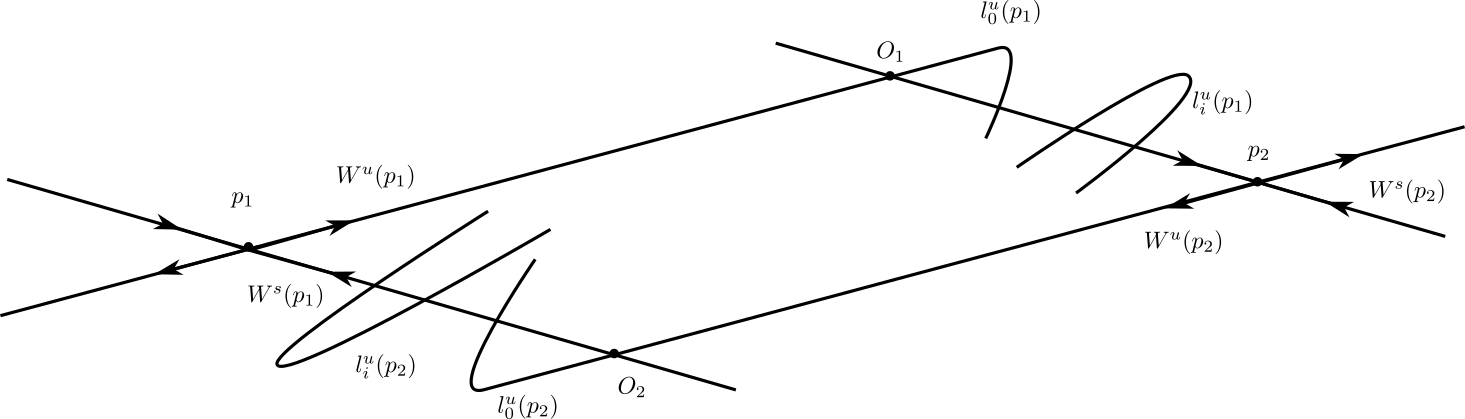}
\caption{The dynamics around $p_1$ and $p_2$, showing the heteroclinic points at $O_1$ and $O_2$.}
\label{fig_creating_ht}
\end{figure}

We now let $l^u_0(p_1)$ be a section of the unstable manifold of $p_1$ which contains in its interior $O_1$.  Then we define its iterations as $l^u_i(p_1) = T^i(l^u_0(p_1))$.  First we state the following lemma:

\setcounter{maintheorem}{1}
\setcounter{lemma}{1}

\begin{sublemma}
    We either have that $l^u_i(p_1)$ already crosses the stable manifold of $p_2$ transversally, or we may perturb the system so that this is achieved (i.e. so that $O_1$ is a point of transversal intersection).
\end{sublemma}  
\begin{proof}
    Using Perturbation \ref{perturbation_which_lets_you_move_unstable_manifold}, we may perturb to obtain two transversal intersections.
\end{proof}
Similarly, we may take the intersection of $W^u(p_2)$ and $W^s(p_1)$ at $O_2$ to be transverse.  We also define $l^s_0(p_1)$ to be a small section of the stable manifold of $p_1$ which contains in its interior $O_2$, and its iterations as $T^i(l^s_0(p_1)) = l^s_i(p_1).$

We let $W^u_{loc}(p_2) (W^s_{loc}(p_2))$ be 
a section of the unstable (stable) manifold 
of $p_2$ containing in its interior $p_2$. 
Then by the lambda lemma (see \cite{siburg}), $\lim_{i \rightarrow \infty}l^u_i(p_1) = W^u_{loc}(p_2)$, and $\lim_{i \rightarrow \infty}l^s_{-i}(p_1) = W^s_{loc}(p_2)$.

Now, by \cite{Moser} there exist Birkhoff normal form coordinates around our fixed points.  In these coordinates our billiard map may be conjugated by a function $N$ so that $T = N \circ f^q \circ N^{-1}$ has the form
\begin{align}\begin{split}\nonumber
    T(\xi,\eta) = (\Delta(\xi\eta)\xi ,\Delta(\xi\eta)^{-1}\eta),
\end{split}\end{align}
where
\begin{align}\begin{split}\nonumber
    \Delta(\xi\eta) = \lambda + \sum_{k=1}^\infty a_k(\xi\eta)^k,
\end{split}\end{align}
and $\lambda$ is the eigenvalue of $df^q(p_1)$ with magnitude greater than 1.  This equation is valid when the product $\xi\eta$ is small. We also may normalize so that this is valid for $\xi,\eta \sim 1$.  

We now define a small neighborhood $U^s=U^s(\delta)$ for some $0<\delta \ll 1$ defined in these coordinates by $\{-\delta < \xi < 0, -1/2 < \eta < 1/2\}$.  Then by our considerations above we have for some $-M$ and some subset $\tilde l^s_{-M}(p_1)\subset l^s_{-M}(p_1)$, that $\tilde l^s_{-M}(p_1) \subset U^s,$ and 
$$
\tilde l^s_{-M}(p_1) \cap U^s \subset \bigg(\{-\delta < \xi < 0, \eta = -1/2\} \cup \{-\delta < \xi < 0, \eta = 1/2\}\bigg),
$$
and $\tilde l^s_{-M}(p_1)$ is as close as 
we would like in the $C^\infty$ topology to a vertical line in these coordinates.

Then, we also define $U^u=U^u(\delta)$ in these coordinates by $\{-1/2 < \xi < 1/2, 0 < \eta < \delta\}$.  We similarly have some large $N$ and some subset $\tilde l^u_N(p_1) \subset l^u_N(p_1)$ that $\tilde l^u_N(p_1) \subset U^u$, 
$$\tilde l^u_N(p_1) \cap U^u \subset \bigg(\{ \xi = -1/2, 0 \leq \eta \leq \delta\} \cup \{ \xi = 1/2, 0 \leq \eta \leq \delta\}\bigg),
$$
and  $\tilde l^u_N(p_1)$ is as close as we would like to a horizontal line in these coordinates.  

Thus, there must be a transverse intersection between $l^s_{-M}(p_1)$ and $l^u_N(p_1)$.  We now make a few more definitions:  We let $O_1^+$ be another point of transverse heteroclinic intersection of $W^u(p_1)$ and $W^s(p_2)$ which is the next point of intersection after $O_1$, and similarly for $O_2^+$.  We then define a section of $W^u(p_1)$ with $O_1^+$ in its interior, and label it as $l^{u+}_0(p_1)$.  We define its iterations as  $l^{u+}_i(p_1)=T^i(l^{u+}_0(p_1))$.  We do the same thing near $O_2$, and define these as $l^{s+}_i(p_1)=T^i(l^{s+}_0(p_1))$.  Observe that one may extend along $W^u(p_1)$ from $l^u_i(p_1)$ to $l^{u+}_i(p_1)$.  We label this as $L^u_i(p_1)$.  We now define the region between $L^u_i(p_1)$ and the strip of $W^s(p_2)$ from $O_1$ to $O_1^+$ as $\Lambda^u_i(p_1)$.   We do the same for the other heteroclinic point $O_2$ and define these regions as $\Lambda^s_i(p_1)$ (see Figure \ref{fig_creating_ht_birkhoff})

\begin{figure}[h]
\centering
\includegraphics[width=10cm, height=8cm]{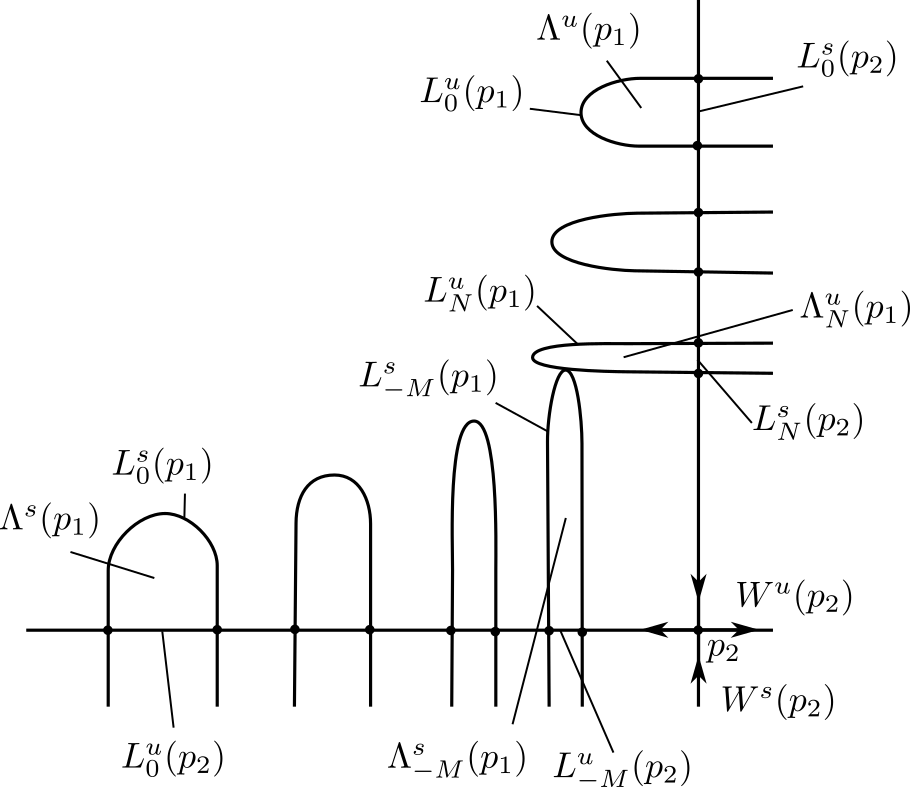}
\caption{The dynamics near $p_2$ in Birkhoff normal form.}
\label{fig_creating_ht_birkhoff}
\end{figure}

First observe that the boundaries of these regions are the union of two connected components which are sections of either the stable or unstable manifold of either $p_1$ or $p_2$.  More precisely we define $\del\Lambda^u_i(p_1) = L^u_i(p_1)\cup L^s_i(p_2)$, and $\del\Lambda^s_i(p_1) = L^s_i(p_1)\cup L^u_i(p_2)$. We also note that 
\begin{align}\begin{split}\label{boundary_length_goes_to_zero}
    \lim_{i \rightarrow \infty} \text{length}(L^s_i(p_1)) = 0
\end{split}\end{align}

Now the point of transverse intersection we found implies that for some numbers $M$ and $N$, we have $\Lambda^u_N(p_1)\cap \Lambda^s_{-M}(p_1)$ is not empty.

We mention one more fact: Since our map is $C^{N-2}$ close to integrable near the boundary (see Lemma \eqref{sublemma_close_to_integrable} in Appendix \ref{appendix_nearly_integrable_near_boundary}), we have by the Stable Manifold Theorem (see \cite{siburg}) that $W^u(p_1)$ and $W^s(p_2)$ are $C^{N-2}$ close to the manifolds $W^u(\hat p_1)=W^s(\hat p_2)$ under an integrable map, which is given explicitly by
\begin{align}\begin{split}\nonumber
    y = (C + q^{-N-1}U(x))^{1/2},
\end{split}\end{align}
for some smooth $U(x)$. Thus we have $L^u_i(p_1)$ is a graph over the strip of $W^s(p_2)$ from $T^i(O_1)$ to $T^i(O_1^+)$.

We now have the elements required to prove the following:
\begin{sublemma}
    Outside a neighborhood of $p_1$, there is a section $L^u(p_1)$ of $W^u(p_1)$ which intersections $W^s(p_2)$ at its endpoints and satisfies
    \begin{align}\begin{split}\nonumber
        dist(L^u(p_1),W^s(p_1)) = O(q^{(-K+1)/2})
    \end{split}\end{align}
    where $dist(L^u(p_1),W^s(p_1))$ is the shortest distance from a point in $L^u(p_1)$ to a point in $W^s(p_1)$, and $K$ is the degree of our Lazutkin coordinates.
\end{sublemma}
\begin{proof}
We fix $N$.  Suppose that for all $m > -M$ we still have $\Lambda^u_N(p_1)\cap \Lambda^s_{m}(p_1)$ is not empty.  Then, we have that as $m \rightarrow \infty$, the section of $\del\Lambda^s_m(p_1)$ which intersects $\Lambda^u_N(p_1)$ becomes arbitrarily small due to \eqref{boundary_length_goes_to_zero}.  Hence, $L^s_i(p_1)$ approaches arbitrarily close to $L^u_i(p_1)$.  This would prove the lemma.

In the other case, there is a smallest $m$ so that $\Lambda^u_N(p_1)\cap \Lambda^s_{m-M}(p_1)$ is empty and $\Lambda^u_N(p_1)\cap \Lambda^s_{m-1-M}(p_1)$ is not empty, which implies $\Lambda^u_{N+1}(p_1)\cap \Lambda^s_{m-M}(p_1)$ is not empty.

If $\Lambda^u_N(p_1)\cap \Lambda^s_{m-M}(p_1)$ is a single point, then we are done.  If not, there are at least two points of intersection between $L^u_N(p_1)$ and $L^s_{m-M}(p_1)$.  We examine these in Birkhoff normal form coordinates near $p_2$. Let $q_0=(-a,b)$ be the point in $L^u_0(p_1)$ which minimizes the $\xi$ coordinate. Then we have $q_{N} = (-\lambda^{N}a,\lambda^{-N}b)$ and $q_{N+1} = (-\lambda^{N+1}a,\lambda^{-N-1}b)$, and these lie in two regions separated by $L^s_{m-M}(p_1)$.  The distance between the $\xi$ coordinate of these two points is then
\begin{align}\begin{split}\nonumber
    \text{Proj}_1(q_{N}-q_{N+1})=\lambda^{N}(\lambda-1)a    
\end{split}\end{align}
Now, we recall our billiard map is
\begin{align}\begin{split}\nonumber
    f^q(x,y) = (x + qy + O(q^2y^K), y + O(qy^{K+1}))
\end{split}\end{align}
so that we have  $\lambda = 1 + O(q^{(-K+1)/2})$.  Thus the distance between our two points is
\begin{align}\begin{split}\nonumber
    \text{Proj}_1(q_{N}-q_{N+1})=O(a\lambda^Nq^{(-K+1)/2})
\end{split}\end{align}
which implies the distance from $q_N$ to $L^s_{m-M}(p_1)$ is of the same order.  This implies that the distance from $q_0$ to $L^s_{m-N-M}(p_1)$ is $O(q^{(-K+1)/2})$
\end{proof}
Now, if we perturb by Perturbation \ref{perturbation_which_lets_you_move_unstable_manifold}, moving the unstable manifold of $p_1$ while keeping the stable manifold of $p_1$ fixed, with the order of the perturbation as $q^{(-N+1)/2}$, we may achieve that the homoclinic intersection is tangential instead of transverse.

\end{proof}

\section{Billiard maps close to integrable near the boundary}\label{appendix_nearly_integrable_near_boundary}

The main result of this appendix is to show

\begin{sublemma}\label{sublemma_close_to_integrable}
    For each $N$ there exists an integrable map $g$ such that close to the boundary the billiard map $f$ is $C^{N}$ close $g$. 
\end{sublemma}
\begin{proof}
We begin with a smooth $C^\infty$ boundary $\Omega$ and associated billiard map $f$, written in Lazutkin coordinates as
\begin{align}\begin{split}\nonumber
    f^q
    \begin{pmatrix}
    x \\ y
    \end{pmatrix}
    =
    \begin{pmatrix}
    x + qy + qA(x,y)y^{N}
    \\
    y + qB(x,y)y^{N+1}.
    \end{pmatrix}
\end{split}\end{align}

We consider the Hamiltonian system with a Hamiltonian given by
\begin{align}\begin{split}\nonumber
    H(x,y) = q\frac{y^2}{2} - q^{-N}U(x),
\end{split}\end{align}
so that
\begin{align}\begin{split}\nonumber
    \dot x &= qy
    \\
    \dot y &= q^{-N}U'(x).
\end{split}\end{align}
Then we have the time-1 map which we label by $g(x,y)=(\hat x,\hat y)$ is given by
\begin{align}\begin{split}\nonumber
    \hat x &= x + qy + \frac{1}{2}q^{-N+1}U'(x) + \frac{1}{6}q^{-N+1}U''(x)(qy)+...
    \\
    &= x + qy + q^{-N+1}\sum_{k=2}^\infty \frac{(qy)^{k-2}}{k!}U^{(k-1)}(x) + q^{-2N}C(x,y)
    \\
    &= x + qy + q\hat A(x,y)y^N + q^{-2N}C(x,y)
    \\
    \hat y &= y + q^{-N}U'(x) + \frac{q^{-N+1}}{2}U''(x)y + ...
    \\
    &= y + q^{-N}\sum_{k=1}^\infty \frac{(qy)^{k-1}}{k!}U^{(k)}(x) + q^{-2N+1}D(x,y)
    \\
    &= y + q\hat B(x,y)y^{N+1} + q^{-2N+1}D(x,y)
\end{split}\end{align}
where $C(x,y)$ and $D(x,y)$ are smooth remainders, and we label
\begin{align}\begin{split}\nonumber
    \hat A(x,y) &= \frac{1}{(qy)^{N+2}}\sum_{k=2}^\infty \frac{(qy)^k}{k!}U^{(k-1)}(x) = \frac{\int_x^{x+qy} \bigg(U(s)-U(x)\bigg)ds}{(qy)^{N+2}}
    \\
    \hat B(x,y) &= \frac{1}{(qy)^{N+2}}\sum_{k=1}^\infty \frac{(qy)^k}{k!}U^{(k)}(x) = \frac{U(x+qy)-U(x)}{(qy)^{N+2}}
\end{split}\end{align}
and choose $U(x)$ so that
\begin{align}\begin{split}\nonumber
    \hat B(x,\frac{1}{q}) &= U(x+1)-U(x) = B(x,\frac{1}{q}).
\end{split}\end{align}
Now the difference between this map and the billiard map is
\begin{align}\begin{split}\nonumber
    x' - \hat x &= q(A(x,y)-\hat A(x,y))y^N - q^{-2N}C(x,y),
    \\
    y' - \hat y &= q(B(x,y)-\hat B(x,y))y^{N+1} - q^{-2N+1}D(x,y).
\end{split}\end{align}
This implies that if we take $y \sim \frac{1}{q}$, then we have for each $0 \leq k \leq N$
\begin{align}\begin{split}\nonumber
    \bigg|\del_x^{N-k}\del_y^k\bigg(f(x,y) - g(x,y)\bigg)\bigg| < C q^{-N+k+1}
\end{split}\end{align}
Hence the billiard map is $C^{N-2}$ close to the map $g$.
\end{proof}

\section{Absolutely T-Periodic in terms of the Billiard Map}\label{appendix_abs_t_periodic_implies_identity}
The main result of this section is to prove 
\begin{theorem}\label{theorem_abs_t_periodic_implies_identity}
    Given a $C^\infty$ billiard map $f$ with associated domain $\Omega$, its corresponding geodesic flow map $F^t:T'\Omega\rightarrow T'\Omega$, and a $q$-periodic point $x_0=(s_0,\phi_0)$, if the differential $df^q(x_0)$ is the identity up to order $n$, then the map $F$ is absolutely $T-$periodic up to order $n$ at $(\gamma(s_0),\frac{(\gamma(s_1)-\gamma(s_0))}{||(\gamma(s_1)-\gamma(s_0))||})$ with period $T=L_0$ where $L_0$ is the length of the orbit of $x_0$.
\end{theorem}
\begin{proof}
We consider 
\begin{align}\begin{split}\nonumber
    L(s_0,...,s_{q-1})=\sum_{i=0}^{q-1}l_i(s_i,s_{i+1})
\end{split}\end{align}
where $l_i(s_i,s_{i+1})$ is the distance between the points $\gamma(s_i)$ and $\gamma(s_{i+1})$.  Recall

\begin{align}\begin{split}\nonumber
    &\frac{\del l_i(s_i,s_{i+1})}{\del s_i} = -\cos(\phi_i^+),\\
    &\frac{\del l_i(s_i,s_{i+1})}{\del s_{i+1}} = \cos(\phi_{i+1}^-)
\end{split}\end{align}

where $\phi_i^+$ is the angle between $\gamma(s_{i+1})-\gamma(s_i)$ and $\gamma'(s_i)$, and $\phi_{i+1}^-$ is the angle between $\gamma(s_{i+1})-\gamma(s_i)$ and $\gamma'(s_{i+1})$.  In the case case where we have a billiard orbit (angle of incidence equal to angle of reflection) these are the same.

Here we consider every $s_i=(s_0,\phi_0)$ as a function of $s_0$ and $\phi_0$, where $(s_i,\phi_i)=f^i(s_0,\phi_0)$.  Then we prove the following lemma:

\setcounter{maintheorem}{3}
\setcounter{lemma}{0}

\begin{lemma}\label{lemma_L_variation}
    Given a $C^\infty$ billiard map,a $q$-periodic point $x_0=(s_0,\phi_0)$, and the length of this orbit \begin{align}\begin{split}\nonumber
        L(x_0)=l_0(s_0,s_1(s_0,\phi_0))+...+l_{q-1}(s_{q-1}(s_0,\phi_0),s_q(s_0,\phi_0)),
    \end{split}\end{align} if the differential is the identity up to order $n$, i.e.
    \begin{align}\begin{split}\nonumber
        df^q(x_0)=Id +F(x_0)
    \end{split}\end{align}
    such that 
    \begin{align}\begin{split}\nonumber
        \frac{\del^k F(s_0,\phi_0)}{\del^j s_0 \del ^{k-j} \phi_0} = 0\text{ for }0\leq j \leq k, 0\leq k \leq n,
    \end{split}\end{align}
    then
    \begin{align}\begin{split}\nonumber
        \frac{\del^kL(s_0,\phi_0)}{\del^j s_0 \del ^{k-j} \phi_0} = 0\text{ for }0\leq j \leq k, 0\leq k \leq n.
    \end{split}\end{align}
\end{lemma}

\begin{proof}
We first prove it for the first derivatives.  We obtain
\begin{align}\begin{split}\nonumber
    \frac{\del L}{\del s_0} &= \sum_{i=0}^{q-1}\frac{\del l_i(s_i(s_0,\phi_0),s_{i+1}(s_0,\phi_0))}{\del s_0}\\
    &= \sum_{i=0}^{q-1}-\cos(\phi_i)\frac{\del s_i}{\del s_0}+\cos(\phi_{i+1})\frac{\del s_{i+1}}{\del s_0}\\
    &= \cos(\phi_q)\frac{\del s_q}{\del s_0} - \cos(\phi_0),
\end{split}\end{align}
so when evaluated at $(s_0,\phi_0)$ this becomes
\begin{align}\begin{split}\nonumber
    \frac{\del L}{\del s_0}=\cos(\phi_0)\bigg(\frac{\del s_q}{\del s_0}-1\bigg).
\end{split}\end{align}
Similarly, we have
\begin{align}\begin{split}\nonumber
    \frac{\del L}{\del \phi_0} &= \sum_{i=0}^{q-1}\frac{\del l_i(s_i(s_0,\phi_0),s_{i+1}(s_0,\phi_0))}{\del \phi_0}\\
    &= \sum_{i=0}^{q-1}-\cos(\phi_i)\frac{\del s_i}{\del \phi_0}+\cos(\phi_{i+1})\frac{\del s_{i+1}}{\del \phi_0}\\
    &= \cos(\phi_q)\frac{\del s_q}{\del \phi_0},
\end{split}\end{align}
so when evaluated at $(s_0,\phi_0)$ this becomes
\begin{align}\begin{split}\nonumber
    \frac{\del L}{\del s_0}=\cos(\phi_0)\frac{\del s_q}{\del \phi_0}.
\end{split}\end{align}
Thus for both of these, with the condition that the differential is the identity up to order $n$, these are $0$.

For higher orders, we note that for $1\leq j\leq k$ and $1\leq k\leq n,$

\begin{align}\begin{split}\nonumber
    \frac{\del^kL}{\del^{k-j}s_0 \del^j \phi_0} &= \frac{\del^{k-1}}{\del^{k-j} s_0 \del \phi_0^{j-1}}\bigg(\frac{\del L}{\del \phi_0}\bigg)\\
    &= \sum_{i=0}^{j-1}\sum_{l=0}^{k-j}\frac{\del^{k-i-1-l}}{\del s^{k-j-l}\del \phi_0^{j-i-1}}\bigg(\cos(\phi_q)\bigg)\frac{\del^{i+l}}{\del^is_0\del^l\phi_0}\bigg(\frac{\del s_q}{\del \phi_0}\bigg).
\end{split}\end{align}
In this case,
\begin{align}\begin{split}\nonumber
    \frac{\del^{i+l}}{\del^is_0\del^l\phi_0}\bigg(\frac{\del s_q}{\del \phi_0}\bigg)=0
\end{split}\end{align}
for all these $i,l$ as a consequence of our differential being the identity up to order $n$, and so 
\begin{align}\begin{split}
    \frac{\del^kL}{\del^{k-j}s_0 \del^j \phi_0}=0
\end{split}\end{align}
for $1\leq j\leq k, 1\leq k\leq n$.

The other one to check is
\begin{align}\begin{split}\nonumber
    \frac{\del^kL}{\del s_0^k} &= \frac{\del^{k-1}}{\del s_0^{k-1}}\bigg(\frac{\del L}{\del s_0}\bigg)\\
    &= \sum_{i=0}^{k-1}\bigg[\frac{\del^{k-1-i}}{\del s_0^{k-1-i}}\bigg(\cos(\phi_q)\bigg)\frac{\del^i}{\del s_0^i}\bigg(\frac{\del s_q}{\del s_0}\bigg)\bigg] - \frac{\del^{k-1}}{\del s_0^{k-1}}\bigg(\cos(\phi_0)\bigg),
\end{split}\end{align}
and since 
\begin{align}\begin{split}\nonumber
    \frac{\del^i}{\del s_0^i}\bigg(\frac{\del s_q}{\del s_0}\bigg)=0
\end{split}\end{align}
for $i\geq 1$ and
\begin{align}\begin{split}
    \frac{\del^{k-1-i}}{\del s_0^{k-1-i}}\bigg(\cos(\phi_q)\bigg) = \frac{\del^{k-2-i}}{\del s_0^{k-2-i}}\bigg(-\sin(\phi_q)\frac{\del \phi_q}{\del s_0}\bigg) =0
\end{split}\end{align}
because of our differential being the identity up to order $n$, we have that 
\begin{align}\begin{split}\nonumber
    \frac{\del^kL}{\del s_0^k}=0
\end{split}\end{align}
for $1 \leq k \leq n$ as well.
\end{proof}
Now we claim that Lemma \eqref{lemma_L_variation} implies theorem \ref{theorem_abs_t_periodic_implies_identity}.  Recall that being \textbf{absolutely $T-$periodic up to order $n$ with periodic $T$ at $(x_0,y_0,\eta_0,\xi_0)$} means the map

\begin{align}\begin{split}\nonumber
    |F^T(x,y,\eta,\xi)-(x,y,\eta,\xi)|^2
\end{split}\end{align}
is $0$ up to order $n$ at $(x,y,\eta,\xi)=(x_0,y_0,\eta_0,\xi_0)$, where we interpret points and direction vectors to be equivalent to points and direction vectors obtained if you extend the geodesic into all of $\R^2$ by ignoring the boundary of $\Omega$ and then reflect across $\{t\gamma'(s):t\in \R \}$ where $\gamma(s)$ is a point in $\del \Omega$ that lies in our orbit. 

Consider $f^q(x_0+x)$, where $x=(s,\phi)$.  By assumption of our differential being the identity up to order $n$, expanding this we obtain
\begin{align}\begin{split}\nonumber
    f^q(x_0 + x) = x_0 + x + O(x^{n+1}).     
\end{split}\end{align}

From Lemma \eqref{lemma_L_variation}, we also have
\begin{align}\begin{split}
    L(x_0+x) = L_0 + O(x^{n+1}).
\end{split}\end{align}
Then since the distance in order to hit the boundary again after $q$ iterates for $x_0$ and $x_0+x$ is $O(x^{n+1})$, we have
\begin{align}\begin{split}\nonumber
    &F^{L(x_0+x)}\bigg(\gamma(s_0+s),\frac{\gamma(f(s_0+s,\phi_0+\phi)|_1)-\gamma(s_0+s)}{||\gamma(f(s_0+s,\phi_0+\phi)|_1)-\gamma(s_0+s)||}\bigg)\\
    \qquad &=F^{L_0}\bigg(\gamma(s_0+s),\frac{\gamma(f(s_0+s,\phi_0+\phi)|_1)-\gamma(s_0+s)}{||\gamma(f(s_0+s,\phi_0+\phi)|_1)-\gamma(s_0+s)||}\bigg)\\
    \qquad &+ O\bigg(\bigg|\bigg(\gamma(s_0+s),\frac{\gamma(f(s_0+s,\phi_0+\phi)|_1)-\gamma(s_0+s)}{||\gamma(f(s_0+s,\phi_0+\phi)|_1)-\gamma(s_0+s)||}\bigg)\bigg|^{n+1}\bigg),
\end{split}\end{align}
which implies
\begin{align}\begin{split}\nonumber
    &F^{L_0}\bigg(\gamma(s_0+s),\frac{\gamma(f(s_0+s,\phi_0+\phi)|_1)-\gamma(s_0+s)}{||\gamma(f(s_0+s,\phi_0+\phi)|_1)-\gamma(s_0+s)||}\bigg)\\
    &\qquad=\bigg(\gamma(s_0+s),\frac{\gamma(f(s_0+s,\phi_0+\phi)|_1)-\gamma(s_0+s)}{||\gamma(f(s_0+s,\phi_0+\phi)|_1)-\gamma(s_0+s)||}\bigg) \\
    &\qquad + O\bigg(\bigg|\bigg(\gamma(s_0+s),\frac{\gamma(f(s_0+s,\phi_0+\phi)|_1)-\gamma(s_0+s)}{||\gamma(f(s_0+s,\phi_0+\phi)|_1)-\gamma(s_0+s)||}\bigg)\bigg|^{n+1}\bigg),
\end{split}\end{align}
which proves the theorem.
\end{proof}

\section*{Potential Applications and Further areas of study}
As mentioned in the introduction, there is a relation between billiard orbits and the laplace spectrum, and hence Weyl's law.  As a result, there is the possibility of using the types of periodic orbits constructed in this paper to obtain some information on the second term in Weyl's law.  The method may involve analyzing the wave trace map expanded around the absolutely periodic orbits constructed here.

Another direction of further study would be to see how much more of the theory of Gonchenko Shilnikov and Turaev can be applied to billiard maps.  In their paper \cite{GST} there are a host of theorems which would be interesting to see hold in billiard maps, but which involve some additional perturbations aside from perturbations which unfold high order HT generically.  for instance, their Theorem 5 about the density of maps which have homoclinic bands would be of particular interest on potentially obtaining an open set of periodic orbits.  See also \cite{berger} for a general framework.

Finally, the last potential further direction would be to see if there are any other relations between the higher derivatives of the billiard map which may allow one to perturb on a finite number of points to eliminate higher derivatives.  What we mean is, in this paper we need $n$ points to eliminate $n+3$ derivatives of the billiard map (and obtain a map which is the identity up to order $n+3$).  If one could find hidden degeneracies in the higher derivatives however, then while keeping $n$ fixed one would be able to by the same techniques eliminate an arbitrarily high number of derivatives of the billiard map, and obtain a $C^{n+2}$ smooth convex domain which has an absolutely periodic orbit of infinite order.  

\section*{acknowledgements}
\markboth{Acknowledgements}{Acknowledgements}

The author would like to thank his advisor Vadim Kaloshin for recommending the problem and for his invaluable insight and guidance, as well as Corentin Fierobe, Alexei Glutsyuk, and Anton Gorodetski for their helpful comments while editing, as well as Pierre Berger for pointing out the potential of obtaining para-universality for a generic billiard.

\end{document}